\documentclass{article}     

\usepackage{amsmath,amsthm,amsfonts,graphicx}
\usepackage{multirow}
\usepackage{diagbox}
\usepackage{multirow}
\def\smallddots{\mathinner{\raise7pt\hbox{.}\raise4pt\hbox{.}\raise1pt\hbox{.}}} 
\def\smallsdots{\mathinner{\raise1pt\hbox{.}\raise4pt\hbox{.}\raise7pt\hbox{.}}}

\DeclareMathOperator{\diag}{diag}

\DeclareMathOperator{\rank}{rank}
\DeclareMathOperator{\nrank}{nrank}

\newtheorem{theorem}{Theorem}[section]
\newtheorem{outline}{Outline}[section]
\numberwithin{equation}{section}
\numberwithin{table}{section}
\newtheorem{lemma}{Lemma}[section]

\newtheorem{corollary}{Corollary}[section]

\newtheorem{algorithm}{Algorithm}[section]
\newtheorem{example}{Example}[section]
\newtheorem{definition}{Definition}[section]
\newtheorem{assumption}{Assumption}[section]

\newtheorem{remark}{Remark}[section]

\newtheorem{problem}{Problem}[section]
\begin{document}
 
\title{Superfast Low-Rank Approximation and  \\  Least Squares Regression
%\footnote{Supported by NSF Grant CCF--1116736 and
%PSC CUNY Award  68862--00 46}
\thanks {Some results of this paper have been presented at 
the Eleventh International Computer Science Symposium in Russia 
(CSR'2016) in St. Petersburg, Russia,  June 2016,  at 
the International Workshop on Sparse Direct Solvers,  Purdue University,
Center for Computational and Applied Mathematics (CCAM), 
 November 12--13, 2016, and 
 the SIAM Conference on Computational Science and
Engineering, February--March 2017, Atlanta, Georgia, USA, and the INdAM Meeting on Structured Matrices
in Numerical Linear Algebra: Analysis, Algorithms and Applications, Cortona,  Italy, September 4-8, 2017.} }
\author{Victor Y. Pan} 
\author{Victor Y. Pan$^{[1, 2, 3],[a]}$,
Qi Luan$^{[2],[b]}$, 
John Svadlenka$^{[3],[c]}$, and  Liang Zhao$^{[1, 3],[d]}$
\\
\and\\
$^{[1]}$ Department of Computer Science \\ 
Lehman College of the City University of New York \\
Bronx, NY 10468 USA \\
$^{[2]}$ Ph.D. Program in Mathematics \\
The Graduate Center of the City University of New York \\
New York, NY 10036 USA \\
$^{[3]}$ Ph.D. Program in Computer Science \\ 
The Graduate Center of the City University of New York \\
New York, NY 10036 USA \\
$^{[a]}$ victor.pan@lehman.cuny.edu \\ 
http://comet.lehman.cuny.edu/vpan/  \\
$^{[b]}$ qi\_luan@yahoo.com \\
$^{[c]}$ jsvadlenka@gradcenter.cuny.edu \\ $^{[d]}$ lzhao1@gc.cuny.edu \\
} 
\date{}

\maketitle

%------------------------------------------------------------------------------
%------------------------------------------------------------------------------

\begin{abstract}  
\noindent \begin{itemize}
  \item%1
Both Least 
  Squares Regression and Low Rank Approximation of a 
matrix\footnote{Hereafter we use the  acronyms {\em LSR} and  {\em LRA}.} are fundamental for Matrix Computations and Big Data Mining and Analysis and both are hot research subjects.
\item%2
The matrices that represent  Big Data 
are frequently so immense that one can only access a tiny fraction of their entries and thus needs {\em superfast algorithms},  which use {\em sublinear time and memory space}, in contrast to {\em fast algorithms}, which use linear time and space. 
\item%3
Unfortunately all superfast algorithms for LSR and LRA fail in the case of the worst case inputs, but we  prove that our superfast algorithms based on {\em sparse sampling} 
output accurate solutions of LRA and LSR for the average case inputs; this
 provides formal support, so far missing, for three well-known and challenging empirical observations.
\item%4  
In our study we unify various techniques for LRA and LSR, which  includes random sampling, proposed in  Computer Science, and  Cross-Approximation iterations, proposed in  Numerical Linear Algebra.\footnote{Hereafter we use the acronyms {\em CS},  {\em NLA},
and {\em C--A}.} We specify some examples of {\em synergy} of these techniques for the computation of LRA.
\item%5   
Our tests with real world inputs
 are in good accordance with our formal study of LSR and LRA and 
 its extension to the acceleration of the Fast Multiple  Method to the
{\em Superfast  Multiple  Method.} 
\item%6
Our progress should demonstrate the power of our novel insights and  techniques and should motivate new efforts towards superfast matrix computations.
\end{itemize}
\end{abstract}



% - - - - - - - - - - - - - - - - - - - - - - - - - - - - - - - - - - - - -

\section{Introduction}\label{sintro}
  
% - - - - - - - - - - - - - - - - - - - - - - - - - - - - - - - - - - - - -

\subsection{Superfast LRA, LSR, and Multipole Method }\label{sprbpr}
  
%------------------------------------------------------------------------------

Low-rank approximation (LRA) of a matrix
 is a hot research area of Numerical Linear Algebra
 (NLA) and Computer Science (CS), surveyed in \cite{HMT11},
 \cite{M11}, \cite{KS16}. This subject is highly
and increasingly popular because of a variety of applications to the most fundamental matrix computations
\cite{HMT11} and numerous problems of data mining and analysis ``ranging from term document data to DNA
SNP data" \cite{M11}.
% see
%\cite{GZT97}, \cite{GTZ97}, \cite{T00}, 
%\cite{FKV}, and \cite{DKM06},
%for sample early works.

In particular modern computations with Big Data involve matrices so huge that realistically  one can access and process only a very small fraction of their entries. So one seeks {\em superfast solution algorithms}, that is, the algorithms that use {\em sublinear time and memory space}, opposing to {\em fast algorithms}, which use linear time and space.

Quite typically these matrices  have {\em low numerical rank}, that is, can be closely approximated by matrices of low rank,
which one can handle superfast.

No superfast algorithm can output  accurate LRA of the worst case inputs, as one can readily prove by applying the adversary argument (see Appendix
  \ref{shrdin}), but empirically the {\em Cross-Approximation}  (C--A) 
iterations have been routinely computing accurate LRA for more than a decade worldwide.

Formal support for this empirical phenomenon has long remained a research challenge so far, and similarly for some successful empirical computations of LRA by means of random sampling, 
that is, by means of
 multiplication of an $m\times n$
input matrix by $n\times l$ random multipliers where the ratio $l/n$ is small.

Empirically the sampling algorithms with various families of random multipliers have consistently output accurate LRA, but formal support for this phenomenon has so far been limited to the
application of
 Gaussian, SRHT and SRFT multipliers\footnote{Here and hereafter ``Gaussian" means ``standard Gaussian  (normal) random " and
``SRHT and SRFT" are the acronyms for ``Subsample Random Hadamard and Fourier
transforms".} (cf. \cite{HMT11}), and with these multipliers the computations are not superfast.

Another challenge came  from the paper \cite{DMM08}, where sampling amounts to the extraction of random sets of rows and columns of an input matrix. In the tests by the authors with real world data their algorithms have computed
accurate LRA by sampling just  
 moderate numbers of rows and columns, but in  their proofs  the authors as well as all their successors had to assume using much larger samples.

We provide analytic support for all three listed challenging empirical observations 
by studying them as special cases of LRA by means of {\em dual sampling}, that is, sampling with a  fixed multiplier
and a random input matrix -- in contrast to the customary primal sampling, surveyed in \cite{HMT11} and \cite{M11}, where an  input matrix is fixed and a multiplier is  random.

We unify the study of LRA by means of
random sampling, proposed in Computer Science (CS), and
by means of the C--A iterations, proposed in 
Numerical Linear Algebra (NLA); in Remark \ref{resnrg} we demonstrate {\em synergy} achieved based on this unification.  

Dual sampling requires distinct analysis, but our estimates for the output errors are as strong as the primal estimates in \cite{HMT11} and  \cite{TYUC17} provided that out multipliers
are orthogonal or at least well-conditioned matrices of full rank. This includes sparse multipliers
with which sampling is superfast.

The real world inputs are not random, but 
by averaging over all the Gaussian input
parameters we extend our results to the average
matrices allowing their LRA.

Some authors argue that 
the average matrices are not necessarily
close to the real world inputs, for which 
our superfast algorithms can fail. We respond that 
the class of all such inputs that are hard for our algorithms is in a sense narrow and moreover  
that it becomes increasingly narrow when we seek superfast LRA of at least one of the matrices $M_i=MQ_i$ for a fixed input matrix $M$ and a sequence of square orthogonal matrices $Q_i$, $i=1,2,\dots$. Thus we expect to obtain accurate LRA in a small number of recursive applications of our superfast algorithms to these matrices $M_i$ (see more details in Sections 
\ref{sdllsr}, \ref{ssprsml}, and \ref{srcvr}). 
    
Our heuristic argument turned out to be in good accordance with the results of our extensive tests  of all these superfast
 LRA algorithms applied to real world inputs.

Motivated by the challenge of providing  formal support for some well-known empirical observations, we have actually gone farther. We proposed  new  recipes for choosing multipliers and using them recursively and arrived at superfast algorithms for LRA that accelerate the known ones, including the sampling and C--A algorithms.   

Our progress should encourage research
efforts towards superfast LRA. Our results should motivate bolder search for sparse multipliers
that support superfast LRA by means of
sampling. Our work should prompt new effort toward synergy of the study of LRA in the research communities of NLA and CS.

By-products of our study having independent interest include 
{\em superfast randomized refinement} of a crude but reasonably close LRA, superfast a posteriori error estimation and correctness verification for a candidate LRA of a matrix,\footnote{Here we must make some additional assumptions because one cannot verify correctness of LRA for the worst case input, as we can readily prove by applying the adversary argument.}, superfast
reduction of an LRA to its special CUR form, and a nontrivial proof that partial products of random bidiagonal and permutation matrices converge to a Gaussian matrix.

Our progress is readily extended to a variety of computations linked to LRA.
We specify an extension to  the acceleration of the Fast Multipole Method 
(which is one of the ten  Top Algorithms of the 20th Century \cite{C00}) to the {\em Superfast Multipole Method}.

Our duality approach is a natural tool in the study of various matrix computations 
besides LRA, as we have  showed earlier in
\cite{PQY15}, \cite{PZ17a}, and \cite{PZ17b}.
 
In Section \ref{sext} we propose  superfast approximate solution of the
celebrated and highly important problem of the {\em Least Squares Regression (LSR)}, extensively studied worldwide \cite{M11}, \cite{W14}. Then again we apply our duality approach
and prove accuracy of our superfast  solution for the average LSR input, and 
then again the results of our extensive tests with real world inputs are in good accordance with our formal study. 

%In the LSR case our proof applies to the customary average over %the random entries of  Gaussian (rather than perturbed factor-%\documentclass[10pt]{}Gaussian) input matrices.

%------------------------------------------------------------------------------

\subsection{Related Works}\label{sextrw} 
 
%-------------------------------------------------------------------------------

We  extend the study of the LSR problem in
\cite{M11} and \cite{W14}, and the references therein
by presenting the first superfast solution for that problem,
which we obtained  by applying our duality techniques. 

\cite{M11} and \cite{W14} also survey LRA and CUR LRA by means of random sampling, but this and other approaches
to LRA  are also  covered in
   \cite{HMT11}, 
  \cite{KS16},
\cite{PLSZ17}, \cite{BW17},
\cite{SWZ17}, \cite{OZ18}, and the references therein. 

Next we comment on three subjects most relevant
to our present work on LRA, that is, LRA via random sampling, LRA by means of C--A algorithms, and CUR LRA.

 The random sampling algorithms for LRA,
surveyed in \cite{HMT11}, \cite{M11},  and  \cite{KS16}, have been introduced and developed in 
\cite{FKV98/04}, \cite{DK03},  \cite{S06}, \cite{DKM06}, \cite{MRT06}, and \cite{DMM08} 
 and most recently advanced 
 in \cite{BW17}, \cite{SWZ17}, and
\cite{TYUC17}.
 
  C--A iterations are  a natural extension of  the Alternating Least Squares method of \cite{CC70} and 
 \cite{H70}. 
 Empirically they dramatically decreased  
quadratic 
memory  space and cubic arithmetic time
of the earlier LRA algorithms.
 The concept of C--A was implicit in \cite{T96} 
 and coined in \cite {T00};  
we credit \cite {B00},  \cite {BR03}, 
 \cite{MMD08}, \cite{MD09}, \cite {GOSTZ10}, \cite{OT10},  \cite {B11}, and \cite{KV16} for devising some efficient 
 C--A algorithms.
  
The application of this approach to LRA is closely linked to the introduction and the study of CUR (aka CGR and pseudo-skeleton) approximation. This important special class of LRA was  first studied as skeleton decomposition
in \cite{G59} and  QRP factorization
  in \cite{G65}  and \cite{BG65}; it was
   refined and redefined as rank-revealing factorization
  in \cite{C87}. Pointers to subsequent algorithms for CUR LRA can be found in \cite{OZ18} and \cite{PLSZ17}. The algorithms were largely directed towards volume maximization, but later random sampling provided an alternative 
  approach.\footnote{LRA has long been the domain and a major subject of NLA until Computer Scientists  proposed  random  sampling, but even earlier  D.E. Knuth, a foremost Computer Scientist, published his pioneering paper  \cite{K85}, which prompted the computation of CUR LRA  based  on volume maximization and eventually evolving into C--A algorithms for CUR LRA.}  
  
 In spite of this long and extensive study, our papers \cite{PLSZ16} and \cite{PLSZ17} and our presentations at the four international conferences in 2016 and 2017 were the first papers and the first presentations that provided formal support for superfast and accurate LRA computation. Later the advanced paper \cite{MW17} presented such a support in the special case of LRA  of positive semidefinite inputs.  
  
%------------------------------------------------------------------------------

\subsection{Our previous work and publications}\label{sprvwk}
 
%------------------------------------------
     
Our  paper  extends
 the study in the papers \cite{PQY15}, \cite{PZ16}, \cite{PLSZ16}, \cite{PZ17a}, and \cite{PZ17b},
devoted to
  
(i) the efficiency of heuristic sparse and structured multipliers for LRA (see \cite{PZ16}, \cite{PLSZ16}), 
 
(ii) the  approximation of  trailing singular spaces 
associated with the $\nu$ smallest singular values 
of a matrix having numerical nullity $\nu$ (see  \cite{PZ17b}), and 
 
(iii)  using
   random multipliers instead of  pivoting 
  in Gaussian elimination 
(see \cite{PQY15}, \cite{PZ17a}).\footnote{Pivoting, that is, row or column interchange, is intensive in data movement and is quite costly nowadays.} 

 The paper \cite{PZ16} still studied fast but not superfast LRA algorithms, but 
the reports  \cite{PLSZ16} and \cite{PLSZ17}  have already   covered most of
the main results of our present paper on superfast LRA algorithms and their extension to Superfast Multipole  Method. Currently we simplify, unify, and streamline these  presentations by linking various  
 LRA  results and techniques to our dual sampling.
   
%------------------------------------------------------------------------------
  
\subsection{Organization of the paper}\label{sorgp}

%------------------------------------------------------------------------------  
 
 At the end of this section we recall some basic definitions. 
 
 In the next section we cover  definitions
  and auxiliary results on matrix computations.
  
 We devote Section \ref{sext} to the LSR problem.,
 
 In  Section \ref{srmc} we define factor-Gaussian  and
  average matrices and recall the known bounds on the norms of Gaussian matrices and their pseudo inverses.
  
 We recall
 LRA problem and its fast solution by means of random sampling
in Section \ref{sbsalg}.  
 
 In  Section \ref{sbsalgsp} we study superfast LRA by means of  sampling and Cross-Approximation.
 
In Section \ref{scurlev} we study randomized computation of CUR LRA directed by leverage  scores.
 
 In Section \ref{sgnrml} we cover generation of multipliers.

 In Section \ref{sextpr} we extend our study of LRA to the acceleration of the Fast Multipole Method to the Superfast Multipole Method.
 
  In Section \ref{ststs} we present the results of our
numerical tests.

In Appendix \ref{sgssrht} we recall randomized error estimates for fast LRA from \cite{HMT11}.

In Appendix \ref{serrcs} we prove our error estimates
for superfast LRA by means of sampling.

In Appendix \ref{ssrcs} we recall the algorithms of \cite{DMM08} for sampling and re-scaling.
  
  In Appendix \ref{slratocur} we cover
  superfast transition from an LRA to a CUR LRA.
  
 In Appendix \ref{shrdin} we specify two small families of hard inputs for superfast LRA.
 
 In Appendix \ref{spstr} we recall the well-known  recipe for
 superfast a posteriori error  estimation for LRA in a special case.

\medskip
 
{\em Some definitions}

\begin{itemize}
 \item%1
Typically we use the concepts ``large", ``small", ``near", ``close", ``approximate", 
``ill-conditioned", and ``well-conditioned"  
quantified in the context, but we specify them quantitatively as needed. 
\item%2
``$\ll$" and ``$\gg$" mean ``much less than" and ``much greater than", respectively.
\item%3
{\em ``Flop"}  stands for ``floating point arithmetic  operation", ``{\em iid}"
for
``independent identically distributed". 
\end{itemize}

%------------------------------------------------------------------------------

\section{Matrix Computations: Definitions and Auxiliary Results}\label{ssdef}

%------------------------------------------------------------------------------
%------------------------------------------------------------------------------
 
\subsection{Basic definitions}  
  
%------------------------------------------------------------------------------

\begin{itemize}
  \item%1 
$I_s$ is the $s\times s$ identity matrix.  $O_{k,l}$ is the $k\times l$
matrix filled with zeros.
We drop the subscripts when they are clear from context or not needed.
\item%2
 $(B_1~|~B_2~|~\dots~|~B_h)$ 
denotes a $1\times h$ block matrix
 with $h$ blocks $B_1,B_2,\dots,B_h$.
\item%3
$\diag(B_1,B_2,\dots,B_h)$ 
denotes a $h\times h$ block diagonal matrix with  $h$
diagonal  blocks $B_1,B_2,\dots,B_h$.
\item%4
$\mathcal R(M)$ denotes the range, that is, the column span,
of a matrix $M$.
\item%5
$M^T$ and $M^*$
denote its transpose and Hermitian 
(aka complex conjugate) transpose, respectively.
\item%6
 An $m\times n$ matrix $M$ is called  {\em unitary}  (and also {\em orthogonal} if it is real)
if $M^*M=I_n$ or  $MM^*=I_m$. 
\item%7
For a matrix $M=(m_{i,j})_{i,j=1}^{m,n}$ and two sets $\mathcal I\subseteq\{1,\dots,m\}$  
and $\mathcal J\subseteq\{1,\dots,n\}$,   define
the submatrices
\begin{equation}\label{eq111}
M_{\mathcal I,:}:=(m_{i,j})_{i\in \mathcal I; j=1,\dots, n},  
M_{:,\mathcal J}:=(m_{i,j})_{i=1,\dots, m;j\in \mathcal J},~{\rm and}~ 
M_{\mathcal I,\mathcal J}:=(m_{i,j})_{i\in \mathcal I;j\in \mathcal J}.
\end{equation}
\item%8

A vector ${\bf u}$ is said to be {\em unit} if $||{\bf u}||=1$.
%------------------------------------------------------------------------------
\item%9
%($||UW||=||W||$ and $||WU||=||W||$ if the %matrix $U$ is unitary.)
$Q(M)$ and  $R(M)$ denote the $m\times n$ unitary factor and the $n\times n$ upper triangular factor in the thin QR factorization of an  
$m\times n$ matrix $M$, respectively  
(see \cite[Theorem 5.2.3]{GL13}).
\item%10
$||\cdot||$ and $||\cdot||_F$ denote the spectral and Frobenius  matrix norms, respectively;  $|\cdot|$ can denote either of them.
\item%11
 $\rank(M)$  denotes the  {\em rank} 
of a matrix $M$.
\item%12 
$\epsilon$-$\rank(M)$ is argmin$_{|E|\le\epsilon|M||}\rank(M+E)$, called
 {\em numerical rank}, $\nrank(M)$, if
  $\epsilon$  is small in context.

%------------------------------------------------------------------------------
 
\item%13
A matrix is {\em  Gaussian} 
if all its entries are  iid Gaussian
(aka standard normal)  variables. 
\medskip

\item%14
$\mathbb R^{p\times q}$ and 
$\mathcal G^{p\times q}$ are the classes of $p\times q$ real and
Gaussian  matrices,
respectively. 
\item%15
 $G_{p,q}$ denotes a Gaussian matrix
in $\mathcal G^{p\times q}$. 
\end{itemize}

For simplicity  we assume  dealing
with real matrices throughout, except for Section \ref{sgnrml}, but our study can be quite readily extended to the complex case; in particular 
 see \cite{D88}, \cite{E88}, \cite{CD05}, 
\cite{ES05}, and \cite{TYUC17} for some relevant results about complex Gaussian matrices.
  
%------------------------------------------------------------------------------
 
\subsection{SVD, pseudo inverse, and conditioning}  
  
%------------------------------------------------------------------------------

\begin{itemize}
\item%9
$M=S_{M}\Sigma_{M}T^*_{M}$ is 
 {\em  Compact Singular Value Decomposition (SVD)}
of a matrix $M$ of rank
$r$ where  
$S_{M,r}$ and $T_{M,r}$
are the  unitary
matrices of its singular vectors 
 and 
 $\Sigma_{M,r}=\diag(\sigma_j(M))_{j=1}^{r}$ is the
diagonal matrix of its singular values,
$\sigma_1(M)\ge \sigma_2(M)\ge \dots\ge 
\sigma_{r}(M)>0$.  
\item%
$M_{\rho}$, its {\em rank-$\rho$ truncation}, is obtained by setting 
$\sigma_j(M)=0$ for $j>\rho$.
\item%9
 SVD of $M_{\rho}$ is said to be its {\em top rank-$\rho$ SVD}  (see Figure \ref{fig5})
\begin{figure}
[ht]
\centering
\includegraphics[scale=0.25]{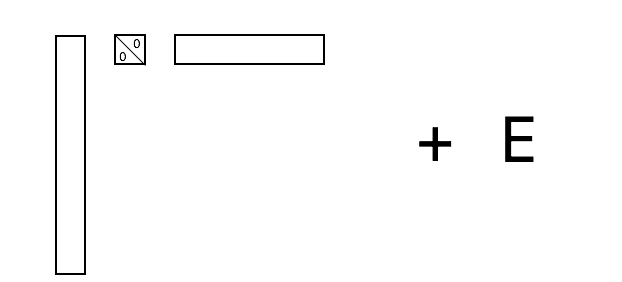}
\caption{The top SVD of a matrix}
\label{fig5}
\end{figure}   
\item%9
 $M^+=T_{M}\Sigma_{M}^{-1}S^*_{M}$ 
is its  {\em Moore--Penrose pseudo
inverse}. 
\item%10a
  $M^+_r$ denotes the pseudo inverse $(M_r)^+$ of the rank-$r$ truncation of
  a matrix  $M$.
\item%10
 $\kappa(M)=||M||~||M^+||$
 is the spectral  {\em condition number}
 of $M$.
%\item%11
%The  $\xi$-rank of a matrix, for a fixed positive  $\xi$,
%is the minimum rank of its approximations within 
%the norm bound  $\xi$. 
%The   {\em numerical rank} of a matrix is its  $\xi$-rank
%for $\xi$ being small in context.
\item%12
A matrix $M$ 
 is called 
{\em ill-conditioned} if 
its condition number 
 $\kappa(M)$ is large in context, and it
%or equivalently if its rank exceeds its numerical rank.  
is  called {\em well-conditioned}
if this  number is reasonably bounded.
[A matrix is ill-conditioned
if and only if it has a matrix of a smaller rank nearby
or equivalently
 if and only if its rank exceeds 
its numerical rank; thus
a matrix of full rank is well-conditioned 
if and only if it has {\em full numerical rank}.]
 \end{itemize} 
%(The ratio of the output and input error %norms of Gaussian elimination 
%is roughly the condition number of an %input matrix,
%cf.  \cite{GL13}.] $\nrank (M)=r$ if and %only if a matrix $M-E$ is well-%conditioned  

Recall the following well-known properties:
\begin{equation}\label{eqnrmcnd}
\sigma_1(M)=||M||,~\sigma_{r}(M)~||M^+||=1~{\rm if}~
\rank(M)=r,~{\rm and~so}~
\kappa(M)=
\frac{\sigma_1(M)}{\sigma_{r}(M)}\ge 1;
\end{equation}
furthermore 
$\kappa(M)=1$ if and only if  $M$ is a unitary  matrix.

%------------------------------------------------------------------------------
 
\subsection{Auxiliary results}  
  
%------------------------------------------------------------------------------
%------------------------------------------------------------------------------

\begin{lemma}\label{lepr3} {\rm [Orthogonal invariance of a Gaussian matrix.]}

Suppose that $k$, $m$, and $n$  are three  positive integers, $k\le \min\{m,n\},$
$G_{m,n}\in \mathcal G^{m\times n}$, $S\in\mathbb R^{k\times m}$, 
 $T\in\mathbb R^{n\times k}$, and
$S$ and $T$ are orthogonal matrices.
Then $SG$ and $GT$ are Gaussian matrices.
\end{lemma}

%------------------------------------------------------
 
 \begin{lemma}\label{lehg} {\rm (The norm of the pseudo inverse of a matrix product. See, e.g., \cite{PLSZ17}.)}
 
Let $A\in\mathbb R^{k\times r}$, 
$B\in\mathbb R^{r\times r}$, and
$C\in\mathbb R^{r\times l}$
and let  the matrices $A$, $B$, and 
$C$ have full rank 
$r\le \min\{k,l\}$.
Then 
$||(ABC)^+||
\le ||A^+||~||B^+||~||C^+||$. 
\end{lemma}

\begin{lemma}\label{letrnc} {\rm (The minimal error of an LRA. \cite[Theorem 2.4.8]{GL13}.)}

For a matrix $M$ 
and a positive integer $\rho$,  the 
 rank-$\rho$ truncation $M_{\rho}$ is a closest rank-$\rho$
approximation of $M$
under both spectral and Frobenius norms,   
$$||M_{\rho}-M||=\sigma_{\rho+1}(M),~
{\rm and}~\tau_{\rho+1}^2(M):=
||M_{\rho}-M||_F^2=\sum_{j\ge \rho}\sigma_j^2(M)$$
%\end{equation}
or in a unified way
$$\tilde \sigma_{r+1}(M):=|M_{\rho}-M|=
\min_{N:~\rank(N)=r} |M-N|.$$
\end{lemma}

\begin{lemma}\label{leprtinv} 
{\rm (The norm of the pseudo inverse of a perturbed matrix. \cite[Theorem 2.2.4]{B15}.)}

If $\rank(M+E)=\rank(M)=r$ and $\eta=||M^+||~||E||<1$, then
$$\frac{1}{\sqrt r}||(M+E)^+||_F\le||(M+E)^+||\le\frac{1}{1-\eta} ||M^+||.$$
\end{lemma}
\begin{lemma}\label{leprtpsdinv} {\rm (The impact of a perturbation of a matrix on its pseudo inverse. \cite[Theorem 2.2.5]{B15}.)}

If $\rank(M+E)=\rank(M)=r$ for $M\in \mathbb R^{k\times l}$, then
$$|M^+-(M+E)^+|\le 
\mu ~|M^+|~~|(M+E)^+|~~|E|.$$
Here  
$\mu=1$ where $|\cdot|=||\cdot||_F$, 

$\mu=(1+\sqrt 5)/2$ where $|\cdot|=||\cdot||$
and
$r<\min\{k,l\}$, 

 $\mu=\sqrt 2$  where $|\cdot|=||\cdot||$
and $r=\min\{k,l\}$.
\end{lemma}

\begin{corollary}\label{coprtinv} 
For matrices $M$ and $E$ and scalars $\eta$ and $\mu$  of Lemmas \ref{leprtinv} and 
\ref{leprtpsdinv} it holds that
 \begin{equation}\label{eqmueta}
\frac{1}{\sqrt r}||M^+-(M+E)^+||_F\le||M^+-(M+E)^+||\le 
\frac{\mu}{1-\eta} ~||M^+||^2~||E||.
\end{equation}
\end{corollary}

%------------------------------------------------------------------------------

\begin{lemma}\label{lepert1} {\rm (The impact of a perturbation of a matrix on its Q factor. \cite[Theorem 5.1]{S95}.)}

Let $M$ and $M+E$ be a pair of
 $m\times n$ matrices  
%respectively,
and let the norm $||E||$ be small. Then 
$$||Q(M+E)-Q(M)||\le \sqrt 2 ||M^+||~||E||_F+O(||E||_F^2).$$
\end{lemma}

%------------------------------------------------------------------------------

 \begin{lemma}\label{lesngr} {\rm (The impact of a perturbation of a matrix on its singular values. \cite[Corollary 8.6.2]{GL13}.)}
 
 For $m\ge n$ and a pair of ${m\times n}$ matrices $M$ and $M+E$ it holds that
 $$|\sigma_j(M+E)-\sigma_j(M)|\le||E||~{\rm for}~j=1,\dots,n. $$
  \end{lemma} 
  
%------------------------------------------------------------------------------

         Hereafter  write 
\begin{equation}\label{eqtqsh}
t_{q,s,h}:=((q-s)sh^2+1)^{1/2}
\end{equation}
for two integers $q$ and $s$ such that $q\ge s$ and a real $h\ge 1$. 
 \begin{theorem}\label{thcndprt} {\rm [Rank-revealing  QR   and LU factorizations of a matrix.]}
 
Given a real $h\ge 1$, three  positive integers $m$, $n$, and $r$ such that $r\le \min\{m, n\}$,  and an $m\times n$ orthogonal matrix $M$,
 the algorithms of both papers \cite{GE96} and \cite{P00} use  $O(mn\min\{m,n\})$ flops in order to
compute an
 $r\times r$ submatrix  $M_{r,r}$ of $M$
such that
 the ratio $\sigma_r(M)/\sigma_r(M_{r,r})$
is at most $t_{m,r,h}t_{n,r,h}$ in \cite{GE96} and at most $t_{m,r,h}^2t_{n,r,h}^2$ in  \cite{P00}.\footnote{Numerical stability of the  algorithms of
 \cite{GE96} and \cite{P00} is ensured under the choice of the 
parameter  $h$ a little exceeding 1.
 The  algorithms of
 \cite{GE96} and  \cite{P00} rely on computing  strong rank-revealing  
QR and LU factorization of $M$, respectively (cf. \cite{GE96}, \cite[Section 5.4]{GL13}, \cite[Chapter 5]{S98}).}
\end{theorem}

\begin{remark}\label{resngv}  {\rm [The Impact of Sampling and Re-scaling on the Singular Values
of a Matrix.]}

For a fixed 
positive $\delta\le 1$ and an $n\times r$
orthogonal matrix $M$,
Theorems \ref{thnrmsmpl} and  \ref{thsngvsmpl} define a randomized 
alternative algorithm  that  
computes an $n\times l$ sampling matrix $S$ and diagonal re-scaling matrix $D$ such that 
\begin{equation}\label{eqsngvsmpl}
 1-\epsilon_{r,l,\delta}\le
\sigma_i^2(M^TSD)\le
1+\epsilon_{r,l,\delta}~{\rm for}~
\epsilon_{r,l,\delta}=\sqrt{4r\ln (2r/\delta)/l}~{\rm and}~
i = 1,\dots, r,
\end{equation}
 with a probability at least $1-\delta$. For $l$ substantially  
exceeding $4r\ln(2r)$, such a randomized algorithm dramatically decreases the
factors  $t_{m,r,h}t_{n,r,h}$
and  $t^2_{m,r,h}t^2_{n,r,h}$ of Theorem
\ref{lesngr}. The cost of performing it is dominated by the cost  $O(nr)$ of computing the norms of the $n$ column vectors of the matrix $V$.
\end{remark}
  
    The following result implies that the top rank-$r$ SVD of a matrix $M$ is stable in its perturbation within a fraction of  
 $\sigma_{r}(M)-\sigma_{r+1}(M)$. 
    
\begin{theorem}\label{thsngspc} {\rm (The impact of a perturbation of a matrix on its top  singular vectors. \cite[Theorem 8.6.5]{GL13}.)} 
  Suppose that   
    $$g=:\sigma_{r}(M)-\sigma_{r+1}(M)>0~ {\rm and}~||E||_F\le 0.2g.$$
   Then,                                                                                                                                                                                                                                                                                                                                                                                                                                                                                                                                                                                                                                                                                                                                                                                                                                                                                                                                                                                                                                                                                                                                                                                                                                                                                                                                                                                                                                                                                                                                                                                                                                                                                                                                                                                                                                                                                                                                                                                                                                                                                                                                                                                                                                                                                                                                    
   for the left and right singular spaces associated with the $r$ largest singular values of the matrices 
   $M$ and $M+E$, 
   there  exist orthogonal matrix bases 
   $ B_{r,\rm left}(M)$,
    $B_{r,\rm right}(M)$,
    $ B_{r,\rm left}(M+E)$, 
    and
   $B_{r,\rm right}(M+E)$
   such that
$$\max\{||B_{r,\rm left}(M+E)-
B_{r,\rm left}(M)||_F,||B_{r,\rm right}(M+E)-B_{r,\rm right}(M)||_F\}\le 
4\frac{||E||_F}{g}.$$
      \end{theorem}
For example, if 
$\sigma_{r}(M)\ge 2\sigma_{r+1}(M)$, which
implies that $g\ge 0.5~\sigma_{r}(M)$,
and if 
$||E||_F\le 0.1~ \sigma_{r}(M)$, then
the upper bound on the right-hand side is  
approximately
$8||E||_F/\sigma_r(M)$.

%------------------------------------------------------------------------------
%------------------------------------------------------------------------------
 
\section{Least Squares Regression}\label{sext} 
  
%------------------------------------------------------------------------------
 
\subsection{The LSR problem and its fast
exact and approximate solution}  
  
%-----------------------------------------------------------------------------------------------------

\begin{problem}\label{pr1} {\em [Least Squares Solution of an Overdetermined Linear System of Equations or Least Squares Regression (LSR).]}
Given two integers $m$ and $d$ such that $1\le d<m$,
a matrix $A\in \mathbb R^{m\times d}$, and a vector ${\bf b}\in \mathbb R^{m}$,
compute and output a vector 
${\bf x}\in \mathbb R^{d}$ that minimizes the norm $||A{\bf x}-{\bf b}||$
or equivalently  outputs the subvector
${\bf x}=(y_i)_{i=0}^{d-1}$ of the vector 
\begin{equation}\label{eqlsrmy}
{\bf y}=(y_i)_{i=0}^d={\rm argmin}
||M{\bf y}||~{\rm where}~M=(A~|~{\bf b})~
{\rm and}~
\begin{pmatrix}  {\bf x} \\
 -1
\end{pmatrix}. 
\end{equation} 
\end{problem}

The minimum norm solution to this problem is given by the vector ${\bf x}=A^+{\bf b}$. 
The  solution is unique and
is equal to $(A^TA)^{-1}A^T{\bf b}$ if a matrix $A$ has full rank $d$.

%------------------------------------------------------------------------------
 
  Sarl\'os in \cite{S06} proposed  the following  randomized algorithm
  for approximate LSR.

%------------------------------------------------------------------------------
 
\begin{algorithm}\label{algapprls} 
{\rm  [Randomized Approximate LSR.]}

%------------------------------------------------------------------------------
 
\begin{description}

%------------------------------------------------------------------------------
 
\item[{\sc Input:}] 
An $m\times (d+1)$  matrix $W$.
%------------------------------------------------------------------------------

\item[{\sc Output:}]  
A vector ${\bf x}\in \mathbb R^{d}$
%or its subvector ${\bf x}$ of %(\ref{eqlsrmy})
 approximating a solution 
of Problem \ref{pr1}  for $M=W$.

%------------------------------------------------------------------------------

\item[{\sc Initialization:}] 
 Fix an 
%oversampling
 integer $k$ such that $1\le k\ll m$. 

%------------------------------------------------------------------------------ 

\item[{\sc Computations:}]
%\item (cf. items 1.4.6 and 1.4.10): $~$

\begin{enumerate}
\item %1
Generate a matrix  $F$ such that
%\begin{equation}\label{eqlg}
${\sqrt k}~F \in \mathcal G^{k\times m}$.
%\end{equation}
\item %2 
Compute and output a solution ${\bf x}$ of Problem \ref{pr1}
for the $k\times (d+1)$ matrix $M=FW$. 
\end{enumerate}

%------------------------------------------------------------------------------

\end{description}

%------------------------------------------------------------------------------

\end{algorithm}

%------------------------------------------------------------------------------
 
%Notice that ${\sqrt k}~M\in \mathcal G^{k\times (d+1)}$ 
% by virtue of Lemma \ref{lepr3}.
 
%  Now let $k\gg d$. Then $G=G_{k,d+1}$ is %a tall-skinny Gaussian matrix and whp is %nearly orthogonal
% by virtue of the theorem below. %Therefore   whp the norm 
%$||G{\bf y}||\approx 1$,  implying that
% Algorithm  \ref{algapprls} outputs an %approximate solution to Problem \ref{pr1} %for an orthogonal input $M=Q$. 

%\begin{theorem}\label{randomprojection}
%(See \cite[Lemma 2]{AV06}.) 
%Suppose that 
%$G\in \mathcal G^{s\times r}$,
%${\bf u}\in \mathbb R^r$, $||{\bf u}|=1$,
%${\bf v} = \frac{1}{\sqrt{s}}G{\bf u}$,
%and $r\le n$. Fix $\bar \epsilon > 0$.
% Then $$\mathrm{Probability}\{
%(1-\bar \epsilon) \le ||{\bf v}||^2 \le %(1+\bar \epsilon)\}
%\ge 1 - 2e^{-(\bar \epsilon^2-
%\bar\epsilon^3)\frac{n}{4}}.$$ 
%\end{theorem}
% The following result (cf. \cite[Theorem %2.3]{W14}) extends the bounds of this %theorem to all unit vectors ${\bf u}$.
 
%------------------------------------------------------------------------------

The following theorem shows that the algorithm outputs approximate solution to Problem \ref{pr1} whp.\footnote{Such approximate solutions 
serve as pre-processors for practical implementation of
numerical linear algebra algorithms 
for Problem \ref{pr1} of least squares computation \cite[Section 4.5]{M11}, \cite{RT08}, \cite{AMT10}.}

%------------------------------------------------------------------------------
 
\begin{theorem}\label{thlsrp}{\rm (Error Bound for Algorithm \ref{algapprls}. See \cite[Theorem 2.3]{W14}.)}   
Let us be given 
two integers $s$ and  $d$
 such that $0<d<s$, a 
 matrix $G=G_{s\times (d+1)}                                                                           \in \mathcal G^{s\times (d+1)}$,
 and  two tolerance values $\gamma_s$ and  
$\xi_s$  such that  
\begin{equation}\label{eqxigm}
0<\gamma_s<1,~ 
0<\xi_s<1,~{\rm and}~s=((d+\log(1/\gamma_s)~\xi_s^{-2})~\eta
\end{equation}
 for a constant $\eta$. 
Then  
\begin{equation}\label{eqlsrp}
{\rm Probability}\Big\{1-\xi_s\le 
\frac{1}{\sqrt s}~\frac{||G{\bf y}||}{||{\bf y}||}\le 1+\xi_s~{\rm for~all~vectors}~
  {\bf y}\neq {\bf 0}\Big\}\ge 1-\gamma_s.
\end{equation}
\end{theorem}
\begin{corollary}\label{colsrp}
Bound (\ref{eqlsrp}) holds provided that $s=k$,
$G=\frac{1}{\sqrt k}M$,  the
 values 
$\gamma_k$ and $\xi_k$ 
satisfy (\ref{eqxigm}) for $s=k$, and 
$M=FW$ is the matrix 
 of Algorithm \ref{algapprls}. 
\end{corollary}

%------------------------------------------------------------------------------
 
 For $m\gg k$ the transition to  the matrix $FW$  substantially decreases  the size  of Problem \ref{pr1};  the computation of the matrix $FM$, however,
 involves order of $dkm> d^2m$ flops, and
this dominates  the overall arithmetic computational cost  of the solution.  The current record upper estimate for this cost is  $O(d^2m)$
 (see \cite{CW13}, \cite[Section 2.1]{W14}), while 
 the record lower bound of \cite{CW09} 
 has  order $(k/\epsilon) (m + d) \log(md)$ provided that the relative output error norm is within a
factor  of $1+\epsilon$ from its minimal value. 

%------------------------------------------------------------------------------
 
\subsection{Superfast dual LSR}  
  
%------------------------------------------------------------------------------

We can accelerate Algorithm \ref{algapprls} to  superfast level  by
 choosing various sparse  multipliers. 
 For example,  
we need no  flops in order to 
 pre-multiply a 
matrix $M$ by  
the orthogonal matrix 
 $\sqrt{\frac {m}{k}}~F=(I_{k}~|~O_{k,m-k})~P_m$ where $P_m$ is a fixed or random 
$m\times m$ permutation matrix. 
For such and any other superfast algorithm for LSR Problem \ref{pr1}, however, one can apply the adversary argument and readily 
specify some inputs for which 
the solution output by the algorithm is far from optimal. 

Nevertheless the following result implies that
except  for a narrow class of hard inputs
Algorithm \ref{algapprls} outputs accurate solution to Problem \ref{pr1} as long as it is applied with any orthogonal multiplier $F$,
 including sparse orthogonal  multipliers with which the algorithm  is superfast. 

%------------------------------------------------------------------------------
 
\begin{theorem}\label{thlsrdd} {\em [Error Bounds for Superfast LSR.]}
Suppose that we are given three  integers $k$, $m$, and  $d$ such that $0<d<k<m$, and
four tolerance values $\gamma_k$, $\gamma_m$, $\xi_k$, and  
$\xi_m$  satisfying (\ref{eqxigm}) 
for $s=k$ and $s=m$.
Define  an orthogonal matrix 
 $Q_{k,m}\in  \mathbb R^{k\times m}$
 and
a matrix $G_{m,d+1}\in \mathcal G^{m\times (d+1)}$ and write
\begin{equation}\label{eqfm}
 F:=a~Q_{k,m}~{\rm and}~M:=b~G_{m,d+1}
 \end{equation}
 for two scalars $a$ and $b$ such that
 $ab\sqrt k=1$.
Then
$${\rm Probability}\Big\{\frac{1-\xi_k}{1+\xi_m}~\le \frac{||FM{\bf z}||}{||{\bf z}||}\le \frac{1+\xi_k}{1-\xi_m}~{\rm for~all~vectors}~
  {\bf z}\neq {\bf 0}\Big\}\ge 1-\gamma_k-\gamma_m.$$
\end{theorem}
\begin{proof} 
Apply Theorem \ref{thlsrp} twice -- for $s=m$ to the $m\times (d+1)$  matrix 
$\frac{1}{b}M$, which is Gaussian by virtue of (\ref{eqfm}), and for
$s=k$ to the $k\times (d+1)$ matrix 
$\frac{1}{ab}FM$, which is
 Gaussian by virtue of Lemma \ref{lepr3}.
  Combine the implied bounds on the norms 
 $||M{\bf z}||$ and $||FM{\bf z}||$.
\end{proof}
Unlike  Algorithm \ref{algapprls} for a fixed input $M$ and  scaled Gaussian multiplier $F$ we assume that $F$ is a fixed orthogonal multiplier and $M$ is a Gaussian matrix. In this case we call the LSR problem {\em dual}.

The real world inputs for LSR are not Gaussian, but the theorem also characterizes 
the average input of
Problem \ref{pr1} defined  by means of averaging over Gaussian inputs, 
and thus the
application of Algorithm \ref{algapprls}
to the average matrix $M$ with using any fixed orthogonal multiplier $F$ outputs
reasonably close approximate  solution to 
Problem \ref{pr1}.

Thus we can argue that the algorithm fails only on a  narrow class of  pairs of $F$ and $M$.

All of the heuristics above and in the next subsection are  in very good accordance with the results of our tests in Section \ref{ststs}.

%------------------------------------------------------------------------------
 
\subsection{Dual implicit pre-processing of LSR}\label{sdllsr}  
  
%--------------------------------------------------------------------------

For any fixed orthogonal multiplier
$F\in \mathbb R^{k\times m}$  we can readily choose 
a hard input  matrix $M$ for which our superfast LSR algorithm 
fails, but Theorem \ref{thlsrdd} shows
that the class of such hard inputs is rather narrow.
 
In Section \ref{sgnrml} we cover some relevant  classes of multipliers. 
 The  test results with multipliers of these classes and real world inputs are in rather good accordance with our analysis, testifying in favor of our approach.
  
Let us comment on two further recipes for choosing multipliers.

(i)  Choose submatrices of the identity matrix, hereafter said to be {\em sub-identity  matrices}. Their  multiplication by any matrix involves no flops. This is still quite a general
class of multipliers: any orthogonal multiplier $F\in \mathbb R^{k\times m}$
can be represented as the product $F=(I_k~|~O_{k,m-k})Q$
 for a square orthogonal matrix $ Q\in \mathbb R^{m\times m}$ and a sub-identity multiplier 
 $(I_k~|~O_{k,m-k})$.
Recall that by virtue of Lemma \ref{lepr3}  $QM$ is  a 
Gaussian matrix if so is $M$.

(ii)  Intuitively, the small chances for
 the failure of our superfast LSR algorithm should
decrease fast as we recursively apply Theorem \ref{thlsrdd} to a fixed  multiplier $F$ and a sequence of 
  input matrices
  $M_i=Q_iM$ for fixed or random
 $m\times m$ orthogonal  matrices $Q_i$ and
 $i=1,2,\dots$.
 
 We perform this pre-processing $M_i=Q_iM\rightarrow M_iF$ implicitly --  by
 applying Algorithm \ref{algapprls} to the matrix 
$M$ and sparse orthogonal multipliers $F_i=FQ_i$, so that
  $FM_i=F_iM$.
  
The computations are superfast if the matrices $F_i$ are sufficiently sparse and if the solution 
  succeeds in a bounded number of applications of the theorem and its supporting	algorithm to the pairs of $M$ and $F_i$.

  At the first success of such an application  we would  solve the original LSR problem by virtue of Theorem 
  \ref{thlsrdd} applied to the multiplier $F_i$ replacing $F$.  
  
  We call the above policy {\em dual implicit pre-processing of LSR}
  and extend it to LRA  in Section \ref{ssprsml}.

%------------------------------------------------------------------------------
 
\begin{remark}\label{renonorth} {\em [LSR with Nonorthogonal Multipliers.]}
We can represent any $k\times m$  multiplier as $LF$ for a $k\times k$ lower triangular $L$ and a $k\times m$ orthogonal matrix $F$. In this case the ratio 
$\frac{||FM{\bf z}||}{||{\bf z}||}$ still satisfies Theorem \ref{thlsrdd} and differs from the ratio $\frac{||LFM{\bf z}||}{||{\bf z}||}$ by factors in the range
from $||L||$ to $||L^{-1}||$.
\end{remark}

In view of this remark  well-conditioned multiplier can be  nearly as efficient as orthogonal ones. In particular a 
properly scaled $k\times m$ Gaussian matrix is nearly orthogonal for $k\ll m$; its approximations by partial products of Section \ref{sgpbd} can be considered among our candidate multipliers.

%------------------------------------------------------------------------------
 
\section{Random and Average Matrices}\label{srmc}
 
Next we recall some definitions and auxiliary results for  computations with
 random matrices (cf. \cite{HMT11}). 
   
%------------------------------------------------------------------------------
 
\subsection{Random and average matrices
of low rank and low numerical rank}

%------------------------------------------------------------------------------

\begin{theorem}\label{thrnd} {\em [Nondegeneration of Gaussian Matrices
with Probability 1.]}

Let $F\in\mathcal G^{r\times m}$, 
$H\in\mathcal G^{n\times r}$, $M\in  
\mathbb R^{m\times n}$, and  $r\le\rho:=\rank(M)=\min\{m,n\}$. 
Then 
the matrices $F$, $H$, $FM$,  and $MH$  
have full rank $r$ 
with probability 1.
\end{theorem}

%------------------------------------------------------------------------------
 
\begin{proof}
Fix any of the matrices 
 $F$, $H$,  $FM$, and $MH$ and its $r\times r$ submatrix $B$. Then 
 the equation $\det(B)=0$ 
  defines an algebraic variety of a lower
dimension in the linear space of the entries of the matrix because in this case $\det(B)$
is a polynomial of degree $r$ in the entries of the matrix $F$ or $H$
(cf. \cite[Proposition 1]{BV88}). 
Clearly, such a variety has Lebesgue  and
Gaussian measures 0, both being absolutely continuous
with respect to one another. This implies  the theorem.
\end{proof}

%------------------------------------------------------------------------------

\begin{assumption}\label{assmp1}  {\em [Nondegeneration of Gaussian Matrices.]}
Hereafter dealing with Gaussian matrices we ignore the probability 0 of their degeneration and assume that they have full rank. 
\end{assumption}

%------------------------------------------------------------------------------

\begin{definition}\label{deffctrg}  {\em [Factor-Gaussian Matrices.]} 
Let $r\le \min\{m,n\}$ and let $\mathcal G_{r,B}^{m\times n}$,
$\mathcal G_{A,r}^{m\times n}$, and
$\mathcal G_{r,C}^{m\times n}$
denote the classes of matrices 
$G_{m,r}B$, $AG_{r,n}$, and 
$G_{m,r} C G_{r,n}$, respectively, which we call
 {\em left},
 {\em right}, and 
 {\em two-sided factor-Gaussian matrices
 of  rank} $r$, 
 respectively,\footnote{Figure \ref{fig5})  actually represents both the top SVD and
 a two-sided factor-Gaussian matrix of low rank.}
 provided that $G_{p,q}$  denotes a 
 $p\times q$
 Gaussian matrix and that
 $A\in \mathbb R^{m\times r}$, 
 $B\in \mathbb R^{r\times n}$, and
 $C\in \mathbb R^{r\times r}$
 are well-conditioned matrices of full rank $r$.
 \end{definition}
 
 \begin{theorem}\label{thfctrg}
 The class $\mathcal G_{r,C}^{m\times n}$
 of $m\times n$ two-sided factor-Gaussian matrices $G_{m,r} C G_{r,n}$
 does not change when we define it replacing the factor $C$ by the diagonal matrix $\Sigma_C$ of its singular values. 
 \end{theorem}
 \begin{proof}
Let $C=S_C\Sigma_C T_C^*$ be SVD.
Then $A=G_{m,r}S_C\in\mathcal G^{m\times r}$ and $B=T_C^*G_{r,n}\in\mathcal G^{r\times n}$
by virtue of Lemma \ref{lepr3}, and so
 $G_{m,r} C G_{r,n}
=A\Sigma_C B$ for $A                                                                                                                                                                                                                                                                                                                                                                                                                                                                                                     \in\mathcal G^{m\times r}$ and 
$B\in\mathcal G^{r\times n}$.
\end{proof} 
 
%------------------------------------------------------------------------------

\begin{definition}\label{defrelnrm} 
{\rm The relative norm of a perturbation of  a
  Gaussian matrix} is the ratio of the perturbation norm and the expected value of the norm of the matrix (estimated in Theorem \ref{thsignorm}).  
\end{definition}
  
%------------------------------------------------------------------------------
  
We refer to all the three matrix classes above as {\em factor-Gaussian matrices
 of  rank} $r$, to their perturbations within
 a relative norm bound $\epsilon$ as {\em factor-Gaussian matrices
 of $\epsilon$-rank} $r$, and to their small-norm 
  perturbations as  {\em factor-Gaussian matrices
 of numerical rank} $r$. 
  
  Hereafter  ``perturbation of a factor-Gaussian matrix" means a {\em perturbation having a small relative norm.}
 
 Clearly $||(A\Sigma)^+||\le ||\Sigma^{-1}||~||A^+||$ and $||(\Sigma B)^+||\le ||\Sigma^{-1}||~||B^+||$ for
a factor-Gaussian matrix 
$M=A\Sigma B$ of rank $r$ of Definition \ref{deffctrg}, and so
whp such a matrix is both left and right 
factor-Gaussian of rank $r$.

%------------------------------------------------------------------------------

\begin{definition}\label{defavrg}  {\em [Average Matrices of a Fixed Rank and Fixed Numerical Rank.]}
 Define the {\em average $m\times n$ matrices of rank, $\epsilon$-rank, and numerical rank} $r$ by averaging over all Gaussian entries of the matrices 
 $G_{m,r}\in \mathcal G^{m\times r}$ and
$G_{r,n}\in\mathcal G^{r\times n}$ 
of Definition \ref{deffctrg}.
\end{definition}

%------------------------------------------------------------------------------

%\begin{definition}\label{defsmooth}
%An  $m\times n$ matrix is $\epsilon$-{\em %Gaussian}
% of rank $r$ if it is
%  the sums of a rank-$r$ matrix 
% and a matrix $xG_{m,n}$ where 
 % $G_{m,n}$ is an
% $m\times n$ Gaussian matrix and $x$ is 
 % a positive scalar not exceeding 
  %$\epsilon$. Call such a  matrix 
%  {\em locally Gaussian} of rank $r$ if 
%$x$ is a small positive
%  scalar. By averaging over the Gaussian %entries of the matrices
% $G_{m,n}$ define the {\em average 
 %$\epsilon$-Gaussian} and the  {\em %average locally Gaussian} $m\times n$ %matrices of rank $r$.
% \end{definition}

%------------------------------------------------------------------------------
 
\subsection{Norms of Gaussian  matrices
and of their pseudo inverses}\label{snrmg}

Hereafter $\Gamma(x)=
\int_0^{\infty}\exp(-t)t^{x-1}dt$
denotes the Gamma function, $e:=2.71828\dots$, and $\mathbb E(v)$ denotes the expected value of 
a random variable $v$. We write  
$\mathbb E||M||$ if $v=||M||$ and 
$\mathbb E||M||_F^2$
if $v=||M||_F^2.$ 

\begin{definition}\label{defnrm} {\rm [The Norms of Gaussian Matrices and of Their Pseudo Inverses.]}
Write 
$\nu_{m,n}=|G|$, $\nu_{{\rm sp},m,n}=||G||$, $\nu_{F,m,n}=||G||_F$,
$\nu_{m,n}^+=|G^+|$,  $\nu_{{\rm sp},m,n}^+=||G^+||$, and
$\nu_{F,m,n}^+=||G^+||_F$,
for  a  Gaussian $m\times n$ matrix  $G$.
($\nu_{m,n}=\nu_{n,m}$ and
$\nu_{m,n}^+=\nu_{n,m}^+$
for all pairs of $m$ and $n$.) 
\end{definition} 

%-----------------------------------------------------------------------------

\begin{theorem}\label{thsignorm} {\rm [Estimates for the Norms of Gaussian Matrices.]}

%------------------------------------------------------------------------------ 

(i) $\nu_{F,m,n}^2=\chi^2(mn)$ is the $\chi^2$-function of order $mn$, having value strongly concentrated about its
expected value $mn$ and
having the probability density function 
$\frac{x^{0.5mn-1}\exp(-0.5x)}{2^{0.5 mn}\Gamma(0.5 mn)}$
and such that
\begin{equation}\label{eqchi} 
{\rm Probability}\{\chi^2(r)-r \ge 2\sqrt{rx} + 2rx\}
\le \exp(-x)~{\rm  for~any}~x>0
\end{equation}
(cf. \cite[Lemma 1]{LM00}).
(ii)
{\rm Probability}$\{\nu_{{\rm sp},m,n}>t+\sqrt m+\sqrt n\}\le
\exp(-t^2/2)$ for all $t\ge 0$ and 
 $\mathbb E(\nu_{{\rm sp},m,n})\le \sqrt m+\sqrt n$  (cf. \cite[Theorem II.7]{DS01}).
\end{theorem}

%------------------------------------------------------------------------------

\begin{theorem}\label{thsiguna} 
 {\rm [Estimates for the Norms of Pseudo Inverses of Gaussian Matrices.
See Assumption \ref{assmp1}.]} 
%{\rm and}~\zeta(t)= 
%\frac{\sqrt{2m}}{\Gamma(m/2)}(t\sqrt{m/2})^{m-1}\exp(-mt^2/2)=
%2~t^{m-1}(\frac{m}{2})^{m/2}\exp(-\frac{m}{2}t^2)/\Gamma(\frac{m}{2}).$$ 

(i)  {\rm Probability} $\{\nu_{{\rm sp},m,n}^+\ge m/x^2\}<\frac{x^{m-n+1}}{\Gamma(m-n+2)}$
for $m\ge n\ge 2$ and all positive $x$,

(ii) {\rm Probability} $\{\nu_{{\rm sp},n,n}^+\ge x\}\le \frac{2.35\sqrt n}{x}$ 
for $n\ge 2$  and all positive $x$,

(iii)  {\rm Probability} $\{\nu_{F,m,n}^+\ge t\sqrt{\frac{3n}{m-n+1}}\}\le t^{n-m}$ and 
 {\rm Probability} $\{\nu_{{\rm sp},m,n}^+\ge  	 
  t\frac{e\sqrt{m}}{m-n+1}\}\le t^{n-m}$
  for all $t\ge 1$ provided that $m\ge 4$, and
  
(iv)  $\mathbb E((\nu^+_{F,m,n})^2)=\frac{n}{m-n-1}$ and
$\mathbb E(\nu^+_{{\rm sp},m,n})\le \frac{e\sqrt{m}}{m-n}$ 
provided that $m\ge n+2\ge 4$. 
\end{theorem}

%------------------------------------------------------------------------------

\begin{proof}
 See \cite[Proof of Lemma 4.1]{CD05} for claim (i), see \cite[Theorem 3.3]{SST06}) for claim (ii), and  see
\cite[Proposition 10.4 and equations (10.3) and (10.4)]{HMT11} for claims (iii)
and (iv).
\end{proof}
 
%------------------------------------------------------------------------------
%------------------------------------------------------------------------------

Theorems \ref{thsignorm} and \ref{thsiguna}
combined imply that an $m\times n$ Gaussian
matrix is well-conditioned  whp
even where the integer $|m-n|$ is close to 0, and whp the upper bounds of Theorem \ref{thsiguna}
on the norm $\nu^+_{m,n}$ decrease very fast as the difference $|m-n|$ grows from 1.

We conclude with the following equations from \cite[Prposition 10.1]{HMT11}.

%------------------------------------------------------------------------------

\begin{theorem}\label{thgssnpr} 
Let $SGT$ be the product of three matrices  $S$,  $G$, and  $T$ for a %Gaussian matrix $G$. Then  
$$\mathbb E||SGT||_F^2=
||S||_F^2||T||_F^2;~\mathbb E||SGT||=
||S||~||T||_F+||S||_F~||T||.$$
 \end{theorem}
 
%------------------------------------------------------------------------------ 
  
\section{Fast LRA by Means of Random Sampling}\label{sbsalg}
 
%------------------------------------------------------------------------------

\subsection{LRA problem and its solution by means of Range Finder}
 
An $m\times n$ matrix $M$ can be  represented (respectively, approximated) 
by a product $UV$ of two matrices $U\in \mathbb R^{m\times l}$  
and $V\in \mathbb R^{l\times n}$ if and only if $l\ge r=\rank(M)$ 
(respectively, $l\ge \nrank(M)$),
and next we study the
 computation of such
a representation or approximation  
(see Figure \ref{fig1}).
      
%------------------------------------------------------------------------------

\begin{figure}[ht] 
\centering
\includegraphics[scale=0.25]{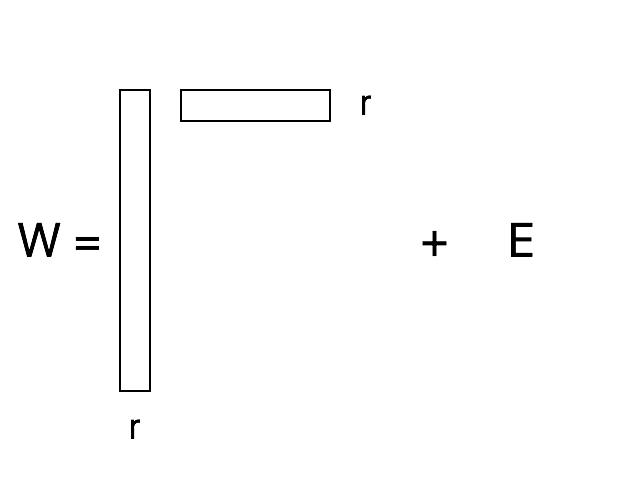}
\caption{Rank-$r$ approximation of a matrix}\label{fig1}
\end{figure} 

By applying SVD one can compute  optimal LRA with 
$|E|=\tilde \sigma(M)$ for $E:=M-UV$
and $\tilde \sigma(M)$ of Lemma \ref{letrnc}, but next we recall faster randomized algorithms that compute nearly optimal LRA whp  (see \cite[Algorithms 4.1--4.5]{HMT11} and Figure 3). 

\begin{algorithm}\label{alg1} {\rm [LRA via Range Finder. See Figures 
\ref{fig01}
and 
\ref{Fig0101} 
and Remark  \ref{retrnc}.]}

%------------------------------------------------------------------------------

\begin{description}

%------------------------------------------------------------------------------

\item[{\sc Input:}] 
An $m\times n$ matrix  $M$ and an integer
  $r$ (the target  rank),  
$1\le r \le \min\{m,n\}$. 

%------------------------------------------------------------------------------

%\item[{\sc Output:}] 
%A  rank-$r$  approximation matrix $\tilde M$ 
%and the relative error $||\tilde M-M||/||M||$.  

%------------------------------------------------------------------------------

\item[{\sc Initialization:}] 
 Fix 
 %a matrix norm $|\cdot|$ (spectral or %Frobenius), 
% a nonnegative 
%tolerance $\delta$,
 an 
%oversampling
 integer $l$, $r\le l\le n$, and
 an $n\times l$  matrix $H$ of full numerical rank $l$.

%------------------------------------------------------------------------------ 
 
\item[{\sc Computations:}]
%\item (cf. items 1.4.6 and 1.4.10): $~$

\begin{enumerate}
\item %1
Compute the  $m\times l$ matrix $MH$.
If $\nrank(MH)<r$, output 
FAILURE.\footnote{If the algorithm fails we  can reapply it generating a new
matrix $H$. If failure persists, we can increase the target rank $r$.}
\item %2
Otherwise compute and output
 an $m\times l'$ matrix  $U$ of rank $r$ for $r\le l'\le l$ 
 such that
\begin{equation}\label{eqincl}
 \mathcal R(MH)\supseteq\mathcal R(U)\supseteq \mathcal R((MH)_r). 
  \end{equation}
% (see the next subsection). 
%Remove its columns that have small norms.
%Orthogonalize its 
%remaining columns (cf. \cite[Theorem %5.2.3]{GL13}).
\item %3  
Compute and output  the $l'\times n$ matrix 
\begin{equation}\label{eqv}
V:=U^+M= {\rm argmin} |UV-M|.
 \end{equation}
% Output SUCCESS if $|\tilde M-M|\le %\delta$;                                                                                                   %otherwise output
% FAILURE.
\end{enumerate}

%------------------------------------------------------------------------------

\end{description}

%------------------------------------------------------------------------------
 
\end{algorithm}

%------------------------------------------------------------------------------

\begin{figure}[htb] 
\centering
\includegraphics[scale=0.25] {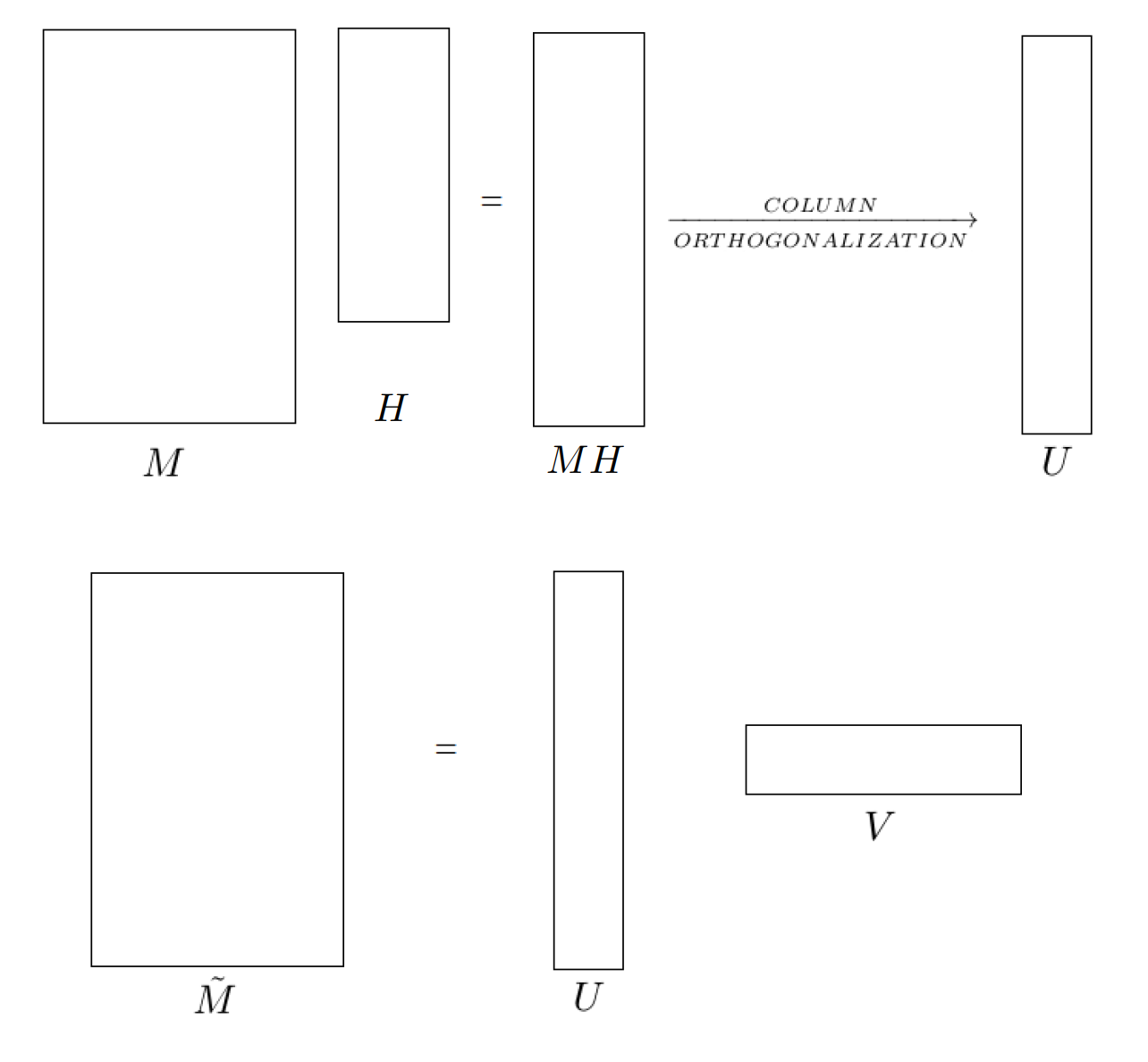}
\caption{The matrices of Algorithm \ref{alg1}}\label{fig01}
\end{figure}
 
%------------------------------------------------------------------------------ 
%------------------------------------------------------------------------------ 
  
We define {\bf Algorithm \ref{alg1}a} 
 and {\bf Algorithm \ref{alg1}b}, respectively, by specifying stage 2
 as  follows:
\medskip

(a) $U:=\diag(I_r,O)Q$ is the matrix  made up of the $l'$ leftmost columns of the  factor $Q$ in a rank-revealing QR factorization of the matrix $MH$ (cf. \cite{GE96}, \cite[Section 5.4]{GL13}, \cite[Chapter 5]{S98});
%\medskip

(b) $U=(MH)_r$.
    \medskip
   
Under both recipes,  (\ref{eqincl}) holds and  one can readily perform both stages 2  and 3 superfast for $l\ll n$.

\begin{remark}\label{retrnc}  {\rm [Decreasing the Rank of an LRA.]}
In Algorithm \ref{alg1} we are given a target rank $r$
and seek a rank-$r$ approximation  $UV$
 of a matrix $M$,\footnote{For the impact of the smaller singular values $\sigma_j(M)$ for $j>r$ of a matrix $M$ on its rank-$r$ approximation, see \cite{TYUC17} and \cite{YGL18}.} although we
 can achieve this in two stages: first compute  a close approximation $UV$ of a larger rank 
$l'\ll \min\{m,n\}$ and then compute  superfast its rank-$r$ truncation 
$(UV)_r$, which still closely approximates the matrix $M$. Namely
Tropp et al. prove in \cite[Section 6.2]{TYUC17}) that
\begin{equation}\label{eqpr6.1}
||(UV)_r-M||_F\le \tau_{r+1}(M)+
2 ||UV-M||_F
\end{equation}
 for $\tau_{r+1}(M)$  denoting the optimal error bound under the Frobenius  norm
 (cf. Lemma \ref{letrnc}). One may prefer to choose $l'>r$ where our available lower bound on the ratio 
 $\sigma_r(M)/\sigma_{r+1}(M)$ is close to 1.
 \end{remark}
 \begin{assumption}\label{assmp2}
Hereafter we simplify our exposition by assuming that $l'=l$.
  \end{assumption}
 
%------------------------------------------------------------------------------
 
\subsection{LRA via Range Finder with pre-processing}   
  
%------------------------------------------------------------------------------
 
%\subsection{Two variations}
The following algorithm is a variation of 
the algorithm by  Tropp et al.
 from \cite{TYUC17}, where  \cite[Theorems 4.7 and 4.8]{CW09} are cited
as the source:\footnote{The algorithms and estimates of \cite[Theorems 4.7 and 4.8]{CW09} use Rademacher (rather than Gaussian) matrices.} 

\begin{algorithm}\label{alg2} {\rm [LRA via Range Finder with pre-multiplication. See Figure 
 \ref{Fig0101}.]}

%------------------------------------------------------------------------------

\begin{description}

%------------------------------------------------------------------------------

\item[{\sc Input:}] 
As in Algorithm \ref{alg1}.

%------------------------------------------------------------------------------

%\item[{\sc Output:}] 
%A  rank-$r$  approximation matrix $\tilde M$ 
%and the relative error $||\tilde M-M||/||M||$. 
 
%------------------------------------------------------------------------------

\item[{\sc Initialization:}] 
 Fix 
 two integers $k$ and
 $l$, $r\le k\le m,~$$r\le l\le n$, and
 two  matrices $F\in\mathbb R^{k\times m}$ and $H\in\mathbb R^{n\times l}$ of full numerical ranks $k$ and $l$, respectively.

%------------------------------------------------------------------------------ 
 
\item[{\sc Computations:}]
%\item (cf. items 1.4.6 and 1.4.10): $~$

\begin{enumerate}
\item %1
As in Algorithm \ref{alg1}.
\item %2
 As in Algorithm \ref{alg1}.
% (see the next subsection). 
\item %3 
Compute the matrices  
 $FU\in\mathbb R^{k\times l}$
 and $FM\in\mathbb R^{k\times n}$.
\item %4 
 If $\nrank(FU)<r$ output 
 FAILURE.\footnote{If the algorithm fails we  can reapply it with new
matrices $F$ and $H$. If failure persists, we can increase the target rank $r$.}
%Orthogonalize its 
%remaining columns (cf. \cite[Theorem %5.2.3]{GL13}).
Otherwise compute the $l\times m$ matrix
$(FU)^+$.  
\item %5
Compute and output the $l\times n$ matrix 
\begin{equation}\label{eqv1}
V:=(FU)^+FM={\rm argmin} |(FU)V-FM|.
\end{equation}
\end{enumerate}

%------------------------------------------------------------------------------

\end{description}

%------------------------------------------------------------------------------
 
\end{algorithm}

%------------------------------------------------------------------------------
         
    By applying   
    recipes (a) and (b) of the previous section at stage 2 arrive at
{\bf Algorithms \ref{alg2}a}
    and {\bf \ref{alg2}b}.

%------------------------------------------------------------------------------

\begin{figure}[htb]
\centering
\includegraphics[scale=0.35] {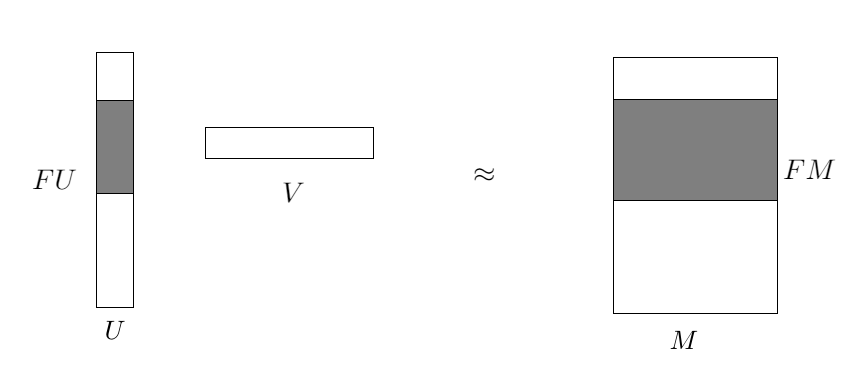}
\caption{The matrices of Algorithms \ref{alg1} and \ref{alg2} (shown in dark color).}\label{Fig0101}
\end{figure}
 
%------------------------------------------------------------------------------ 
%------------------------------------------------------------------------------
 
\subsection{The known error bounds}  
  
%------------------------------------------------------------------------------
 
%   Hereafter we write 
%\begin{equation}\label{eqmtilde}
%\tilde M:=UV.
%\end{equation}

    $\tilde \sigma_{r+1}(M)$ is a lower bound on $|M-UV|$ by virtue of Lemma \ref{lesngr}.
   
Algorithm \ref{alg1}a nearly reaches this lower
estimate, both in expectation and probability, when it is applied  with
  a Gaussian multiplier $H$ (see \cite[pages 274--275]{HMT11}).  
   It supports a 
   %little 
   weaker but still quite good randomized upper bounds with SRHT and SRFT multipliers $H$ (see \cite{T11} and \cite[Theorem 11.2]{HMT11}). We recall these upper bounds in Appendix \ref{sgssrht}.
    
         Clarkson and Woodruff   
  prove in  \cite{CW09}  that their algorithm reaches the lower bound within a factor of $1+\epsilon$ 
  whp if the multipliers 
  $F\in \mathcal G^{k\times m}$ and 
  $H\in \mathcal G^{n\times l}$ are iid Gaussian and if $k$ and $l$ are sufficiently large, having order of  
  $r/\epsilon$ and $r/\epsilon^2$ for small 
  $\epsilon$, respectively.

 Tropp et al. in \cite[Section 1.7.3]{TYUC17} point out practical benefits of the computation of LRA
for smaller integers $k$ and $l$  and  prove
 in \cite[Theorem 4.3]{TYUC17} that 
 as long as 
the ratios $k/l=l/r$ noticeably exceed 1 the output of Algorithm \ref{alg2}a 
for Gaussian factors $F$ and $H$ 
 is accurate in expectation,  namely, 
 \begin{equation}\label{eqtyuc} 
 \mathbb E||M-UV||_F^2\le 
 \frac{kl}{(k-l)(l-r)}\tau_{r+1}^2(M).
 \end{equation} 
For example, $\mathbb E||M-UV||_F^2$
 is within a factor of 4 from its lower bound 
$\tau_{r+1}^2(M)$ for 
$k=2l=4r$.
    
%------------------------------------------------------------------------------ 
  
\section{Superfast LRA by Means of  Sampling and Cross-Approximation}\label{sbsalgsp}
 
%------------------------------------------------------------------------------

\subsection{Overview}\label{sover}
For $l\ll n$
 Algorithms \ref{alg1} and \ref{alg2}
are superfast  at stage 2 because 
$MH\in \mathbb R^{m\times l}$.
Likewise for $k\ll m$
 Algorithm \ref{alg2}
is superfast  at stage 4 because 
$FM\in \mathbb R^{k\times n}$
and $FU\in \mathbb R^{k\times l}$.

We can perform the other stages superfast  as well if we choose 
 sufficiently sparse multipliers $F$ and $H$.
 
Then the output LRAs as well as the output of any superfast LRA algorithm cannot be accurate  for the worst case input and even for a small input family of Appendix \ref{shrdin},
% Clarkson and Woodruff in \cite{CW09}  %proved that accessing $mn/\epsilon$ %memory cells is necessary in order to %output 
%  randomized solution with the worst case error norm bound   within the factor $1+\epsilon$ from the optimum.
but our  Corollaries \ref{coerrfctr} and \ref{commh+}, Theorem \ref{thst3}, and 
estimates (\ref{equfu+fappr})  and (\ref{eqfprob}) together show
that {\em whp 
 Algorithms \ref{alg1} and  \ref{alg2} 
 output accurate LRA of a perturbed factor-Gaussian matrix} when we choose {\em any fixed pair of
% orthogonal or even well-conditioned
 multipliers} $F$ and $H$ of full numerical ranks $k$ and $l$, respectively.

In particular this holds where the {\em multipliers are sufficiently sparse so that
 the algorithms are superfast}.

 We call such LRA algorithms {\em dual sampling} because 
 we randomize an input for fixed multipliers
 versus customary randomization of multipliers for a fixed input, e.g., in \cite{HMT11}.
 
 We present our results in the following two subsections 
 but move some proofs into Appendix \ref{serrcs}.
 
  As a by-product we deduce a formula for {\em superfast a posteriori error estimation and correctness verification} for a candidate rank-$r$ approximation 
   of a matrix $M$ given with some upper bounds on the norms
  $||M||$ and $||M_r^+||=1/\sigma_r(M)$
  (see Remark \ref{reposter}).
  Then, in Section \ref{ssprsml}, we propose some recipes for choosing multipliers and in 
  Section \ref{sprimcur} cover superfast computation of LRA in a special CUR form by means of sampling with sub-identity multipliers. Here we include the celebrated Cross-Approximation  algorithm, which we link to recursive random sampling.
  
%------------------------------------------------------------------------------
 
\subsection{The impact of post-multiplication}\label{simppost}

%We  first prove that output of Algorithm
%\ref{alg1} is the exact rank-$r$ %representation with probability 1
%if its input is a right factor-Gaussian %matrix of rank $r$.
\begin{theorem}\label{thmmh}  
Let Algorithm
\ref{alg1}  be applied 
 with an $n\times l$ multiplier $H$ of rank  
$l\ge r$
to a right factor-Gaussian matrix
$M$ of rank $r$, so that $M=AG_{r,n}\in\mathbb R^{m\times n}$, 
%where  
%\begin{equation}\label{eqh}
for $G_{r,n}\in \mathcal G^{r\times n}$ and  
$\rank(A)=r$. Then the
algorithm outputs
%\end{equation}
 rank-$r$ decomposition  
 $M=UV$ with probability 1.
\end{theorem}
\begin{proof}
By virtue of Theorem \ref{thrnd}
   $\rank(M)=\rank(MH)=\rank(AG_{r,n}H)=r$ with probability 1, because $\rank(A)=r$ and 
  $\rank(H)=l\ge r$ by assumption.
  Clearly $\mathcal R(M)\supseteq \mathcal R(MH)$, and it follows that
  $$\mathcal R(MH)=\mathcal R(M).$$

 Consequently equations  (\ref{eqincl}) and (\ref{eqv}) combined imply  the theorem. 
 		
\end{proof}
    
Our next three theorems rely on
Corollary \ref{coprtinv} and 
     Lemma \ref{lepert1} and  involve the random variables $\nu_{{\rm sp},r,l}$,  $\nu_{{\rm sp},r,n}$, and $\nu_{{\rm sp},r,l}^+$ of Definition \ref{defnrm},  the norms of the input, auxiliary and error matrices, and the terms 
      $O(|E|^2)$ that combine the values of  order $|E|^2$ 
      where $E$ is the matrix of 
      perturbation of a rank-$r$ matrix 
      $\tilde M$ into an input matrix $M$. 
      
      Claims (i) of Theorems \ref{thcrerr2}  and \ref{thcrerr1} apply to any rank-$r$ approximation 
      $\tilde M=M-E$ of  $M$. The norm 
      $|E|$  reaches its 
      minimal value
       $\tilde \sigma_{r+1}(M)$ 
       for $\tilde M=M_r$,
       but in claims (ii) of Theorems \ref{thcrerr2}  and \ref{thcrerr1} and in 
        Theorem \ref{therrfctr}  we %estimate the error norm $|M-UV|$ provided
         let $M$ be a perturbation  of an $m\times n$ right factor-Gaussian matrix 
          $\tilde M=M-E$ of rank $r$ such that 
     \begin{equation}\label{eqfctrg}    
        \tilde M =AG_{r,n},~A\in \mathbb R^{m\times r},
        G_{r,n}\in \mathcal G^{r\times n},~{\rm and}~\rank(A)=r,~
       |\tilde M|\le|A|~\nu_{r,n}.
       \end{equation}
       Then $\tilde M H=AG_{r,n}S_H\Sigma_HT_H^*=AG_{r,n}\Sigma_HT_H^*$ by virtue of Lemma \ref{lepr3}, and we deduce from Lemma \ref{lehg} that
      \begin{equation}\label{eqfctrg1}  
|(\tilde MH)_r|\le|A|~\nu_{{\rm sp},r,l}~|H_r|,~~~
%by virtue of Lemma \ref{lepr3},
(|\tilde MH)_r^+|\le|\tilde M^+|~|\nu_{{\rm sp}r,l}^+|H_r|.
\end{equation}
 \begin{theorem}\label{thcrerr2} {\rm [Estimates for the
  Output Error  Norm of Algorithm \ref{alg1}a.]}
 
Apply Algorithm \ref{alg1}a to a 
perturbation $M=\tilde M+E$ of a rank-$r$
matrix $\tilde M$.  
  
(i) Then  
%\begin{equation}\label{eqqqr}
$$||M-UV||\le 2||E||+\Big(||HE||+
\sqrt 2 ~||(\tilde M H)^+||~||HE||_F\Big)||(\tilde M)||+
O(||E||^2_F)\le$$
$$2||E||+\Big(1+\sqrt{2r}~||(\tilde M H)^+||\Big)||(\tilde M)||~||HE|||+
O(||E||^2_F).$$
%\end{equation} 

(ii) Let $\tilde M$ be a rank-$r$
right factor-Gaussian matrix 
of (\ref{eqfctrg})
and  write $\kappa:=\kappa(A)\kappa(H_r)$.
Then
% \begin{equation}\label{eqkpp}
$$||M-UV||\le \Big(2+\Big(||A||~||H||+
\sqrt{2r}~\nu_{{\rm sp},r,l}^+\kappa\Big)
~\nu_{{\rm sp},r,n}\Big)||E|| +O(||E||_F^2).$$
%\end{equation}
 \end{theorem}

  \begin{theorem}\label{thcrerr1}  {\rm [Estimates for the Output Error  Norm of Algorithm \ref{alg1}b.]}
 
Apply Algorithm \ref{alg1}b where 
 $M=\tilde M+E$,
 $\rank(\tilde M)=r$, and 
 $\eta:=||(MH)_r^+|| ||E||<1$. 
 
(i) Then  
%\begin{equation}\label{eqt123nrms}
$$||M-U V||\le \Big (1+\kappa(\tilde MH)+\kappa(\tilde M,H_r)
+\frac{\mu}{1-\eta}\kappa(\tilde MH)\kappa(\tilde M,H_r)\Big)~||E||+O(||E||^2) $$
%\end{equation}
for $\mu\le(1+\sqrt 5)/2$ 
 of  Lemma \ref{leprtpsdinv} and 
 $$\kappa(\tilde MH)\le \kappa(\tilde M,H):  =||\tilde M||~||H||~||(\tilde MH)^+||\le \kappa(\tilde M)\kappa(H_r).$$ 
  
(ii) Let $\tilde M$ be a rank-$r$
right factor-Gaussian matrix 
 of (\ref{eqfctrg})
and keep writing  $\kappa:=\kappa(A)\kappa(H_r)$.
Then
%$\eta\le 
%||A^+||~||H_r^+||\nu_{{\rm sp},r,l}^+||%EH_r||$ and
%\begin{equation}\label{eqt123gs} 
% \begin{equation}\label{eqabt123}
$$ ||M-UV||\le \Big(1+\Big(\kappa(G_{r,l})+\nu_{{\rm sp},r,n}\nu_{{\rm sp},r,l}^+\Big(1+
 \frac{\mu}{1-\eta}\kappa(G_{r,l})\kappa\Big)\Big)\kappa\Big) ||E||
 +O(||E||^2)~{\rm for}~G_{r,l}\in \mathcal G^{r\times l}.$$
%\end{equation}
 \end{theorem}
 
By including the term $O(||E||^2)$ we simplified the estimates of Theorem 
\ref{thcrerr1}; by extending its proof in Appendix \ref{serrcs} we can specify that term
(cf. (\ref{eqmh+mh}) and (\ref{eq|u+|})).
We extend  this theorem 
in Remark \ref{reposter} to 
superfast a posteriori error estimation
for candidate LRA; we cannot obtain such estimation from our other theorems. 

In our next theorem  we assume that 
  $\tilde M$ is a right  factor-Gaussian matrix of rank $r$, bound the ratio $||E||_F/(\sigma_r(M)-\sigma_{r+1}(M))$, and then avoid using 
  the term  $O(||E||^2)$.
 Moreover the theorem refines the estimates of the   latter two theorems and make them depend only on $|E|$, $\kappa(H_r)$, $l$, $n$, and $r$. 
   \begin{theorem}\label{therrfctr} {\rm [Estimates in the Case of a Perturbed Factor-Gaussian Input.]}
    Apply Algorithm \ref{alg1}   to a perturbation $M$ of a rank-$r$  right 
    $m\times n$ factor-Gaussian matrix 
    $\tilde M$
      of (\ref{eqfctrg}) and 
suppose that  
\begin{equation}\label{eqalphsg}   
  \alpha:=||E||_F/(\sigma_r(M)-\sigma_{r+1}(M))\le 0.2~{\rm and}~ 
   \xi:=4\alpha\nu_{{\rm sp},r,n}\nu_{{\rm sp},r,l}^+\kappa(H_r)<1.
\end{equation} 
Then 
$$||M-UV||^2\le 
 \Big(1+\frac{1}{(1-\xi)^2}\kappa^2(H_r) \nu_{{\rm sp},r,n}^2(\nu_{{\rm sp},r,l}^+)^2\Big)||E||^2. 
$$  
  \end{theorem}

Theorems \ref{thcrerr2}--\ref{therrfctr} involve the norm $|E|$ of a perturbation matrix $E$.
In our next estimates for  the expected value $\mathbb E||M-UV||$ we use the 
relative spectral norm of that matrix
(cf. Definition \ref{defrelnrm}):
\begin{equation}\label{eqdlte} 
\Delta_{E,r,n}=\delta_{r,n}||E||~{\rm where}~\delta_{r,n}=
\frac{1}{\mathbb E(\nu_{{\rm sp},r,n})}
  \end{equation}
  and  (see Theorem \ref{thsignorm} or \cite{L07}) where
  \begin{equation}\label{eqdlteasym}
\sqrt n~  \delta_{r,n}=O(1)~{\rm as}~n\rightarrow \infty.
  \end{equation}

Now let $l\gg r$, let $M$ be a
perturbed right
rank-$r$ factor-Gaussian matrix, combine Theorems \ref{thsignorm}, \ref{thsiguna}, and \ref{thcrerr2}--\ref{therrfctr},   and obtain that
the norm $\mathbb E||M-UV||$ is
 {\em  strongly concentrated about its expected value} $\mathbb E||M-UV||$.
 
 These theorems together also imply the following estimates for this value.

\begin{corollary}\label{coerrfctr} {\rm [Expected Error Norm
of an LRA of a Perturbed Factor-Gaussian Matrix.]}

Under the assumptions of Theorems
 \ref{thcrerr2}--\ref{therrfctr} define 
 $\delta_{r,n}$ and $\Delta_{E,r,n}$ by (\ref{eqdlte}),
  write $e:=2.71828182\dots$
and $\kappa:=\kappa(A)\kappa(H_r)$
(cf. Theorems
 \ref{thcrerr2} and \ref{thcrerr1}),
 and use the constants $\mu$ and $\eta$ of~Theorem~ \ref{thcrerr1} and 
 $\xi$ of~Theorem~\ref{therrfctr}.
 Then
$${\rm (i)}~~~~~~\mathbb E||M-UV||\le \Big(2\delta_{r,n}+||A||~||H||+\frac{e\sqrt{2rl}}{l-r}~~\kappa\Big)\Delta_{E,r,l} +O(\Delta_{E,r,l}^2);~~~~~~~~~$$

 $${\rm (ii)}~~~~~\mathbb E||M-UV||\le ((1+q_1\kappa)\delta_{r,n}+q_2\kappa^2)\Delta_{E,r,l}+O(\Delta_{E,r,l}^2)~~~~~~~~~~~~~~~~~~~~~~~~$$ 
$$~{\rm for}~q_1=e~\frac{l+\sqrt {lr}}{l-r}~{\rm and}~q_2=
\Big(1+\frac{e\mu}{1-\eta}~\frac{l+\sqrt {lr})}{l-r}\Big)
\frac{e\sqrt l}{l-r}$$
such that 
$$q_1 \approx e~{\rm if}~r\ll l~{\rm and}~q_2  \approx\Big(1+\frac{e\mu}{1-\eta}\Big)\frac{1}{\sqrt l}~{\rm if}~r\ll l;$$
$$~~~~~~~~~{\rm (iii)}~~~~~\mathbb E||M-UV||^2\le 
(\delta_{r,n}^2+
s_1^2\kappa^2(H_r))||E||^2~{\rm for}~s_1=\frac{e}{1-\xi}~\frac{\sqrt {l}}{l-r}~~~~~~~~~~~~~~~~~~~~~~~~~~~$$                                                                                                                                                                                                                               such that
 $$s_1\approx\frac{e}{(1-\xi)\sqrt l}~
 {\rm if}~r\ll l.$$ 
\end{corollary}
\begin{remark}\label{reerrwnrmz} {\rm [Comparison of Estimates for the  Expected Error Norms for the Outputs of Algorithms \ref{alg1}a and \ref{alg1}b.]}
Recall that  $\kappa(A)$ is not large by assumption. Furthermore 
$\kappa(H_r)=1$ in the case of orthogonal multipliers $H$, and then both claims (ii) and (iii)
of the corollary imply that
$\mathbb E||M-UV||=O(1/\sqrt l)$
if $r\ll l$ and if $\delta_{r,n}$ is small (this follows from (\ref{eqdlteasym})
for large $n$); the upper bound of claim (i) is overly pessimistic because 
$||U^+||=1$ in Algorithm \ref{alg1}a
while $||U^+||=O(\sqrt l/(l-r))$ in Algorithm \ref{alg1}b applied to a
 factor-Gaussian matrix.
 \end{remark}
  
%------------------------------------------------------------------------------

\subsection{The impact of pre-multiplication}\label{simppre}
 
\begin{theorem}\label{thst3}  {\rm [Impact of Pre-Multiplication
on the Output Error Norm.]}

Suppose that Algorithm \ref{alg2} outputs  a matrix $UV$ for 
$V=(FU)^+FM$ and that $m\ge k\ge l=\rank(U)$. Then
\begin{equation}\label{equvmpre}
 M-UV=(I_m-U(FU)^+F)(M-UU^+M)
\end{equation}
and hence
\begin{equation}\label{eqf}
 |M-UV|\le 
f|M-UU^+M|~{\rm for}~f:=|I_m-U(FU)^+F|\le
1+|U|~|(FU)^+|F|.
\end{equation} 
 \end{theorem}
 \begin{proof} 
 Recall that $V=(FU)^+FM$ and notice that
 $(FU)^+FU=I_m$
 for $k\ge l$. 
 
 Therefore
 $V=U^+M+(FU)^+F(M-UU^+M)$. Consequently (\ref{equvmpre}) and
(\ref{eqf}) hold.
 \end{proof}

 We estimated the norm $||UU^+M-M||$
in the previous subsection;  next we estimate the factor $f$. 

 Recall that
$(FU)^+=U^+F_r^+$ and 
$\rank(U(FU)^+F)=r$; therefore $f\ge 1$ for $r<m$.

Next choose $k=l=r$, apply Algorithm \ref{alg2}b, which computes an orthogonal 
 $m\times r$ matrix  $U=Q(MH)$,    and   
  apply to it the algorithms of  \cite{GE96}. They 
  output an $r\times m$ 
 matrix $F$ being a submatrix  of an $m\times m$ permutation matrix and such that 
 \begin{equation}\label{eqfuge}
||(FU)^+||^2\le t_{m,r,h}^2=
(m-r)rh^2+1~{\rm  ~for~any~real}~h\ge 1
\end{equation}
   (see  Theorem \ref{thcndprt}). Choose $h$ near 1.
  
% \begin{remark}\label{retsknn}
%For $l\ll m$  
%the matrix $U$ is tall-skinny,  
%and then we can alternatively apply to it  %the randomized algorithm of  Theorem %\ref{thcurgge}. Then 
%with a probability at least  $1-\gamma$
%this decreases the latter upper bound on %the norm $||(FU)^+||$  to $1+\epsilon$ %for any fixed positive $\epsilon$, %although in this case 
%the number  of required rows of the %factor $F$
%grows fast as $\epsilon$ and $\gamma$ %decrease to 0.
%\end{remark} 
 QR factorization of the matrix $MH$ and the algorithms of \cite{GE96} together
use $O(mr^2)$ flops, and so they are 
 superfast
for $r\ll m$; pre-multiplication of both matrices $U$ and $M$ by $F$ involves no flops.

Recall that $h\approx 1$, $||U||=
||F||=1$ and
if  $|||H||\approx 1$ as well,
then 
 (\ref{eqfuge}) implies that
\begin{equation}\label{equfu+fappr}
f\le f'\approx \sqrt {mr}~~~
{\rm for}~r\ll m.
\end{equation}   

By alternatively combining 
Algorithm \ref{alg2} with randomized  algorithm of Remark \ref{resngv} 
(rather than with deterministic algorithms of \cite{GE96}), we decrease the factor
$f$  whp to nearly  1
provided that $m\ge k\ge l$
and that $l$ noticeably exceeds $4r\ln(2r)$. Namely in this case for 
 a fixed positive 
$\delta\le 1$  we deduce from Theorem \ref{thsngvsmpl} that 
\begin{equation}\label{eqfprob}
f\le 
1+\sqrt{4r\ln (2r/\delta)/l}~
~{\rm with~a~probability~at~least}~1-\delta.
\end{equation}

Our bounds (\ref{equfu+fappr}) and (\ref{eqfprob}) hold where we apply Algorithm \ref{alg2} to any  $m\times n$ input matrix $M$.

 In our
  next theorem and corollary $M$ is restricted to be a two-sided factor-Gaussian matrix  $M$ of rank $r$ for any choice of integers $r$, $l$, $k$, $m$, and $n$ such that $l\le n$ and 
  $1\le r\le l\le k\le m$ and for any choice of multipliers $F$ and $H$, and then we only approximate the expected value  of $f$ for such a random input.   
  
  Unlike (\ref{equfu+fappr})  and (\ref{eqfprob}), however, these  more limited results cover the choice of $F=H^*$, which supports the computation of
 Hermitian and Hermitian semidefinite LRA of Hermitian and Hermitian semidefinite inputs, respectively, by means of
 the algorithms of \cite[Sections 5 and 6]{TYUC17}. 
  
 \begin{theorem}\label{commhfg} {\rm [Estimates for the Impact 
 Factor $f$ in the Case of a Perturbed Factor-Gaussian Input.]}
 
 Suppose that Algorithm \ref{alg2} outputs  a matrix $UV$ for 
$V=(FU)^+FM$ where $m\ge k\ge l\ge \rank(U)\ge r\ge 1$  and $M$ is
  an $m\times n$ two-sided factor-Gaussian matrix  of rank $r$,  that is,  $M=G_{m,r}\Sigma G_{r,n}$, $G_{m,r}\in\mathcal G^{m\times r}$, $G_{r,n}\in\mathcal G^{r\times n}$, and
 $\Sigma$ is a well-conditioned nonsingular diagonal matrix.
 Then 
 
 (i) $UV=M$ with probability 1,
 
 (ii) $f-1\le |U(FU)^+F|\le \kappa(F)\kappa(H_r)\kappa(\Sigma)~\kappa(G_{r,l})\nu^+_{k,r}\nu_{r,m}$ for $G_{r,l}\in \mathcal G^{r\times l}$, and
 
 (iii) a perturbation of the matrix $M$ by a matrix $E$ increases the bound on $f$  
 within $O(|E|)$.
\end{theorem}
\begin{proof} 
Theorems \ref{thmmh}  
 and \ref{thst3} together imply 
 claim (i) as long as $\rank(FU)=r$ with probability 1.
 Now recall that by assumption  
$\rank(M)=r\ge \rank(MH)$, and so $U=(MH)_r=MH$ and $FU=FMH$.

Let $F=S_F\Sigma_FT^*_F$ and 
 $H=S_H\Sigma_HT^*_H$ be SVDs.
By  twice applying Lemma \ref{lepr3}  deduce that 
\begin{equation}\label{eqfmh}
FU=FMH=S_F\Sigma_FG_{k,r}\Sigma G_{r,l}\Sigma_HT^*_H.
 \end{equation}
Hence $\rank(FU)=r$
  with probability 1, and  
 claim (i) of  Theorem \ref{commhfg} follows.

Furthermore (\ref{eqfmh}) implies that
$(FU)^+=
T_H\Sigma_H^+G_{r,l}^+\Sigma^+G_{k,r}^+\Sigma^+_FS_F^*$.
Recall that $S_F$ and $T_H$ are orthogonal matrices, apply Lemma \ref{lehg}, and obtain that $$|(FU)^+|=
|\Sigma_H^+G_{r,l}^+\Sigma^+G_{k,r}^+\Sigma^+_F|\le
|\Sigma_H^+|~|G_{r,l}^+|~|\Sigma^+|~|G_{k,r}^+|~|\Sigma^+_F|.$$  
Substitute $G_{r,l}^+=\nu^+_{r,l}$,
 $G_{k,r}^+=\nu^+_{k,r}$, 
 $|\Sigma^+_F|=|F^+|$, and
 $|\Sigma^+_H|=|H_r^+|$ and obtain
% \begin{equation}\label{eqfh}
 $$|(FU)^+|\le 
|F^+|~\nu^+_{r,l}|\Sigma^+|\nu^+_{k,r}|H_r^+|.$$
 %\end{equation}
Together with the bound $|U|=|MH|\le\nu_{r,m}|\Sigma|\nu_{r,l}|H|$ this  
   implies  claim (ii). 
   
    Notice that under a 
    perturbation of $M$ by $E$,  
    $U$ changes by $EH$ and $FU$ by $FEH$.  Deduce from Corollary \ref{coprtinv} that   
  the implied impact on $(FU)^+$  is also
     within $O(|E|)$, and then
claim (iii) follows.\footnote{Corollary \ref{coprtinv}  enables  {\em explicit estimates} for that impact.}
     \end{proof}
     For $r\ll l\le k\ll m$
  the upper bound of claim (ii) of Theorem \ref{commhfg}  on the random variable 
  $f$ is strongly concentrated about its
  expected value  $f''$.
   Let us estimate that value.
  
     Theorem \ref{commhfg} shows that
     $\frac{f''-1}{\kappa(F)\kappa(H_r)\kappa(\Sigma)}\le \nu_{r,m}\nu_{r,l}\nu_{r,l}^+\nu_{k,r}^+$. This upper bound decreases by  a factor of 
$\mathbb E(\nu_{m,r})$ if we normalize the  matrix $G_{m,r}$  by dividing it and an input matrix $M$ by this factor.
     Based  on these observations we obtain the following estimates.
  \begin{corollary}\label{commh+} {\rm [Expected Value of a Factor $f$.]}

Under the assumptions of Theorem \ref{commhfg} let $r\ll l\le k\ll m$
and write $e^2:=7.3890559\dots$.
Then  
%\begin{equation}\label{eqf''fr}
$$\mathbb E\Big(\frac{f''_F-1}{\kappa(F)\kappa(H_r)\kappa(\Sigma)\mathbb E_{F,m,r}}\Big)=\frac{r\sqrt{lr}}{(k-r-1)^{1/2}(l-r-1)^{1/2}}\approx r\sqrt{\frac{r}{k}}~{\rm if}~r\ll\min\{k,l\}$$  
 % ~{\rm and}
%   \end{equation}
%   \begin{equation}\label{eqf''}  
 $$\mathbb E\Big(\frac{f''_{\rm sp}-1}{\kappa(F)\kappa(H_r)\kappa(\Sigma)
 \mathbb E_{{\rm sp},m,r}}\Big)=
  e^2\frac{l+\sqrt{lr}}{l-r}
  \frac{\sqrt{k}}{k-r}\approx
  e^2\sqrt{\frac{1}{k}}~{\rm if}~r\ll\min\{k,l\}.$$
%  \end{equation}
  \end{corollary}

 \begin{remark}\label{reposter} {\rm [Superfast a Posteriori Error
Estimation.]}
 Suppose that Algorithm \ref{alg2}b has computed a rank-$r$ approximation $UV$ of 
 a matrix $M$. Then relationships
  (\ref{eqt123||})--(\ref{eqw13}),   (\ref{eqmh+mh}),
(\ref{eq|u+|}),  and (\ref{eqf}) combined enable us to bound the output error norm
$||M-UV||$ in terms of the norms 
$||M||$, $||E||$, $||F||$, $||H||$,
$||\tilde U||$ and $||\tilde U^+||$.
Observe, however, that in the proof of Lemma \ref{lew123} 
we can proceed based on the equation
$M-UV-E=\tilde U\tilde V^+\tilde M-UV^+M$
rather than $M-UV-E=-(UV^+M-\tilde U\tilde V^+\tilde M)$; then we arrive at a  
 similar  dual expression  in terms of the norms $||\tilde M||$, $||E||$, $||F||$, $||H||$,
$||U||$ and $||U^+||$. We can choose
any (e.g., orthogonal) matrices $F$ and $H$, can readily estimate their norms, and can 
compute the matrices $U$ and $U^+$ and their norms superfast.
Then we yield superfast a posteriori error estimation and correctness verification 
for Algorithm \ref{alg2}b as long as we have some upper bounds
 on the norms $||E||$ and 
 $||\tilde M||\le ||M||+||E||$.
 \end{remark}
                                                                                                                                                                                                                                                                                                                                                                           %------------------------------------------------------------------------------
 
\subsection{Some Recipes for Choosing  Multipliers for LRA}\label{ssprsml} 
 
For a fixed superfast LRA algorithm and a fixed pair of multipliers
$F\in \mathbb R^{k\times m}$ and $H\in \mathbb R^{r\times n}$ we can readily choose 
a hard input  matrix $M$ for which the algorithm 
  fails (cf. Appendix \ref{shrdin}), but our analysis has shown
that the class of such hard inputs is rather narrow.
 
In Section \ref{sgnrml} we cover various  classes of multipliers. 
 The  test results with them and real world inputs are in rather good accordance with our analysis, testifying in favor of our approach.
  \medskip
  
Next we extend to the case of LRA recipes (i)  and (ii)
of Section \ref{sdllsr} for choosing multipliers.

(i)   Sub-identity  multipliers
of  Section \ref{sdllsr} are quite general in application to LRA as well: any triple of a matrix $M\in \mathbb R^{m\times n}$ and orthogonal multipliers $F\in \mathbb R^{k\times m}$
and $H\in \mathbb R^{n\times l}$
can be represented as a triple
  of a matrix $QM\bar Q\in \mathbb R^{m\times n}$ and sub-identity multipliers 
 $(I_k~|~O_{k,m-k})$
and $(I_l~|~O_{n-l})^T$
such that $F=(I_k~|~O_{k,m-k})Q$
and $H=\bar Q(I_l~|~O_{n-l})^T$.
Recall that any orthogonal transformation 
$M\rightarrow QM\bar Q$ preserves all singular
values of a matrix $M$ (and consequently its norm,  rank, and numerical rank) as well as its property of being  a perturbed 
factor-Gaussian matrix (cf. Lemma \ref{lepr3}).

(ii) Apply
Algorithm \ref{alg1} to a fixed input matrix $M$ with   multipliers $F_i=FV_i$ and $H_i=W_iH$
for orthogonal matrices $V_i\in \mathbb R^{m\times m}$  and 
$W_i\in \mathbb R^{n\times n}$, $i=1,2, \dots, g$,
whose choice can be fixed or randomized
(cf. the recipe for
 the computation of LSR at the end of Section \ref{sext}).
This is equivalent to the application of a single pair of multipliers $F$ and  $H$ to the matrices $M_i=V_iMW_i$
for $i=1,2,\dots,g$.
% or in other words to recursive  pre-processing
%of an input matrix $M$ with a sequence of orthogonal $n\times n%$ multipliers $W_1,\dots,W_g$. 
Since for a fixed pair of multipliers $F$ and  $H$  the class of hard inputs $M$ is narrow,
it would be increasingly unlikely if
 all matrices 
$M_1,\dots,M_g$ would keep landing in the narrow set of hard inputs for the pair of  multiplier $F$ and $H$  as the integer $g$ grows. 

Our 
 Abridged  Hadamard and Fourier multipliers of Section \ref{shad} 
 define 
 SRHT  and SRFT multipliers in $g=\log_2n$ 
 steps for 
 $n=2^g$.  Then the application of our algorithms to an $m\times n$ matrix $M$ 
 of low numerical rank
 would give us probabilistic guarantee of success with computation of accurate LRA if we apply multipliers computed at the end of this recursive process.
 In our extensive tests with real world inputs, however, we consistently succeeded 
 in computing accurate LRA with Abridged Hadamard multipliers obtained much sooner, typically in three recursive steps. 
 
 Similar comments apply to recursive representation of Gaussian multipliers
 (cf. Section \ref{sgpbd}).
  
%------------------------------------------------------------------------------

\subsection{Primitive,  Cynical, and Cross-Approximation Algorithms and CUR LRA}\label{sprimcur}

Choosing sub-identity  multipliers $F\in \mathbb R^{k\times m}$
and $H\in \mathbb R^{n\times l}$ amounts to choosing $k$ rows and  $l$ columns of the identity matrices $I_m$ and $I_n$,  respectively. The choice can be random or by a fixed recipe. In both cases we arrive at the  following {\bf Primitive Algorithm} for
 LRA.
\begin{itemize} 
\item%1
Given an $m\times n$ matrix $M$ and a target rank $r$, fix two integers $k$ and $l$ such that $r\le k\le m$ and $r\le l\le n$ and two submatrices 
$C=M_{:,\mathcal J}$ and 
$R=M_{\mathcal I,:}$ made up of $l$ columns and $k$ rows of $M$, respectively.
\item%2
 Let $M_{k,l}=M_{\mathcal I,\mathcal J}$ denote 
the $k\times l$ submatrix made up of 
the common entries of $C$ and $R$. Compute 
its rank-$r$ truncation $M_{k,l,r}:=(M_{k,l})_r$. 

[This is the 
only stage of the algorithm involving  flops and is superfast for $kl\ll mn$ even  if we perform it by using SVD.]
\item%3
 Compute the matrix $U=M_{k,l,r}^+$.
\item%4
Output the matrices $C$, $U$ and $R$,
whose product $CUR$ is a 3-factor LRA of $M$.
\end{itemize}

CUR LRA is  defined by the row set 
$\mathcal I$ and the column set 
$\mathcal J$ or equivalently by the submatrix $M_{k,l}$.
We call $M_{k,l}$ the {\em CUR generator}
 of a matrix $M$ and  call
$U=M_{k,l,r}^+$ the {\em nucleus} of the CUR LRA. 

For a matrix $M$ of rank $r$ the Primitive algorithm outputs its {\em CUR decomposition}, that is, CUR LRA with no errors (see Figure \ref{fig2}).

\begin{figure}
[ht]
\centering
\includegraphics[scale=0.25]{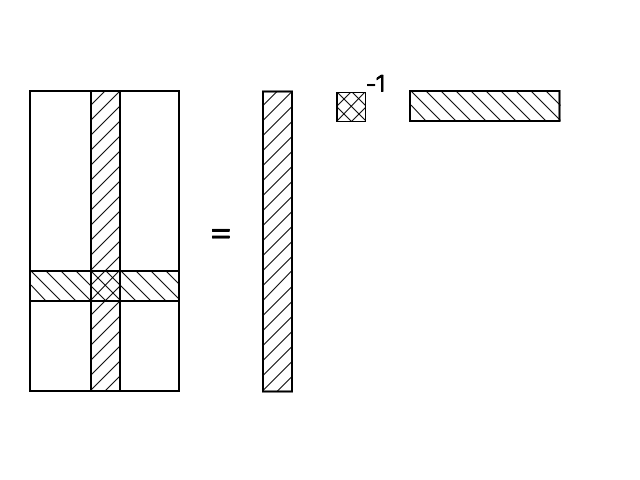}
\caption{CUR decomposition  with a nonsingular CUR generator}
\label{fig2}
\end{figure}

%\end{enumerate}

 By combining the first
or the last two factors of CUR LRA into a single one, we
 obtain a two-factor LRA. 
Conversely {\em given a two-factor LRA of $M$,
the algorithms of Appendix \ref{slratocur} compute its CUR LRA superfast.}

The Primitive Algorithm
 is a particular case of Algorithm \ref{alg2} with orthogonal multipliers $F$ and $H$, and so its output is accurate on the average input that allows its LRA and whp is accurate on a perturbed  factor-Gaussian input of a low  rank.
 
 Clearly the algorithm is superfast for $k\ll m$ and $l\ll n$
 even if we compute the nucleus from SVD of the CUR generator
 $M_{k,l}$.
 
%------------------------------------------------------------------------------
    
Next we generalize the Primitive algorithm. For a pair of integers 
$q$ and $s$ such that   
\begin{equation}\label{eqklmnqsr}  
0<r\le k\le q\le m~{\rm and}~r\le l\le s\le n,
\end{equation}
 fix a random $q\times s$   
 submatrix of an input matrix $M$. 
 Then by applying any auxiliary LRA algorithm compute a  $k\times l$ CUR generator $M_{k,l}$
 of this submatrix, 
 consider it  
 a CUR generator of the  matrix $M$, 
  and build on it a CUR LRA of that matrix. 
  
 For $q=k$ and $s=l$ this is the
Primitive
algorithm again. Otherwise 
the algorithm 
 is  still quite primitive; we call it {\bf Cynical}\footnote{We
allude to  the benefits of the austerity and simplicity of primitive life,  
advocated by Diogenes the Cynic, and not to shamelessness and  distrust associated with modern  cynicism.} (see Figure \ref{fig3}).
  
  \begin{figure}[ht] 
\centering
\includegraphics[scale=0.25]{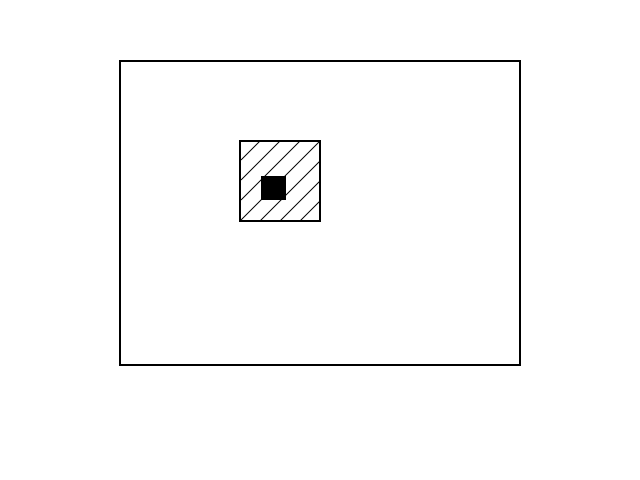}
\caption{A cynical CUR algorithm (the strips mark a $p\times q$ submatrix; a $k\times l$ CUR generator is black)
}\label{fig3} 
\end{figure}
  
 The algorithm is superfast if $pq\max{p,q}\ll mn$ 
 even if we compute a CUR LRA of an 
$p\times q$ initial submatrix from its SVD. 
Clearly the output error bounds in the case of Cynical algorithms
are not greater than  in the case of   the Primitive algorithm, and in our tests 
  Cynical algorithms
 succeeded more consistently than the Primitive algorithm
 (see Tables  
\ref{tb_ranrc} and \ref{tb_lr}).

The following recursive extension of the  Cynical algorithms along the line of our recipe (i) of the previous section is said to be {\bf Cross-Approximation (C--A)} iterations (see Figure \ref{fig4}). 
 
 \begin{itemize} 
\item%1
Given an $m\times n$ matrix $M$ and a target rank $r$, fix four integers $k$, $l$, $p$ and $q$ such that $r\le k\le m$ and $r\le l\le n$.
\item%2
Fix an $p\times n$ ``horizontal" submatrix $M_0$ of the matrix $M$, made up of its $p$ fixed rows. 
\item%3
By applying any (e.g., the Primitive, a Cynical,  or SVD-based) algorithm compute 
a $k\times l$ CUR generator $M_{k,l}$
of this submatrix and reuse it for the matrix $M$.

[This  
stage of the algorithm is superfast  for $l^2\ll n$  even if we perform it by using SVD.]

\item%4
Output the resulting CUR LRA of $M$ if it is close enough. 
\item%5
Otherwise swap $p$ and $q$ and reapply
the algorithm to the transposed matrix $M^T$.

[This is equivalent to fixing a  
 $m\times q$ ``vertical" submatrix $M_1$ of $M$ covering the submatrix $M_{k,l}$
 and computing its $k\times l$ CUR generator. The computation is superfast  for $p^2\ll m$  even if we perform it by using SVD.]
 \item%6
 Recursively alternate such ``horizontal" and ``vertical" 
  steps until an accurate CUR LRA is computed or until the number of recursive C--A steps exceeds a fixed tolerance bound.
 \end{itemize}

\begin{figure} 
[ht]  
\centering
\includegraphics[scale=0.25]{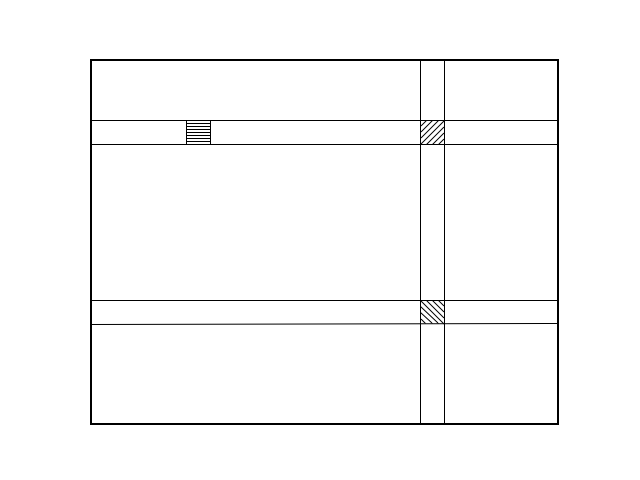}
\caption{The first three recursive C--A steps output three
striped matrices.}
\label{fig4}
\end{figure}
 
 \begin{remark}\label{rec-atnsr}
 In a dual C--A
iterations one begins with a
``vertical" step. 
In the extension to LRA of $d$-dimensional tensors, for $d>2$, one has a variety of search directions 
 in every C--A iteration, for example,  among   $d$  {\em fibers} or among  about $0.5d^2$  {\em slices}.   
\end{remark}

Every C--A step  of the algorithm is performed superfast for $p\ll m$ and $q\ll n$ even it is based on SVD and (according to our study in the previous sections) outputs 
accurate CUR LRA on the average input allowing LRA and whp on perturbed factor-Gaussian input of rank $r$. 

\begin{remark}\label{re vlmmx}  
In a highly efficient implementation
in \cite{GOSTZ10}, it is  sufficient
to apply   $O(nr^2)$ flops in order to initialize the C--A algorithm
 in the case of $n\times n$ input and
 $p=q=r$. Then  the algorithm uses $O(rn)$ flops per C--A step. 
 \end{remark} 
 
Now suppose that a posteriori error estimation for CUR LRA of an input matrix $M$ is superfast (see Remark \ref{reposter} and Appendix \ref{spstr}) or drop the verification stage and just stop the algorithm in a fixed reasonably small number of C--A steps. 
 Then the algorithm  is superfast  overall for every input matrix $M$.

Empirically computation of an accurate  CUR LRA in a quite small number of C--A steps has been observed, and not only on the average, but   
 consistently in extensive worldwide application of C--A algorithm
for more than a decade. The long and extensive efforts towards 
formal support of this empirical observation
mostly relied on revealing and exploiting its link to
maximization of the volume of a CUR generator.\footnote{The volume of a matrix $M$ equals $|max\{|\det(M^TM)|,|\det(M^TM)|\}^{1/2}$. This is 
$|\det (M)|$ if $M$ is a square matrix.} Achieving maximization is sufficient in order to ensure close CUR LRA  (see \cite{GTZ97}, \cite{GT01}, \cite{OZ18}, \cite{PLSZ17}, and the bibliography therein)). This long and important study had major impact on the design of more efficient C--A iterations (see, e.g., \cite{GOSTZ10}), but  formal support for the empirical power of C--A iterations first appeared only in 
 \cite{PLSZ16} and \cite{PLSZ17},  and the support available so far is more general when it relies on estimating the errors of LRA rather than on volume maximization.
 
%------------------------------------------------------------------------------

\subsection{Recovery of LRA in the case of  nonsingular multipliers}\label{srcvr}

%-----------------------------------------------------------------------------------

Multiplicative pre-processing of an input matrix $M$ according the recipe (ii) of Section\ref{ssprsml}  is
an additional resource for enhancing the power of  C--A 
iterations. In this case we should choose nonsingular multipliers $F$ and $H$ in order to support the recovery of an LRA of the original matrix from an LRA of the pre-processed one. Next we discuss this issue.

Suppose that we have computed LRA $\bar U\bar V$ of the matrix 
$\bar M=FMH$ for two nonsingular multipliers $F$ and $H$.
Then   
write $U=:F^{-1}\bar U$ and $V:=\bar VH^{-1}$ and obtain 
the induced LRA $UV$ that
 approximates  the original matrix $M$ within the error norm
 $|M-UV|=|F^{-1}(\bar M-\bar U\bar V)H^{-1}|$, which is  at most
  $|F^{-1}|~|\bar M-\bar U\bar V|~|H^{-1}||=
  \Delta~|F^{-1}|~|H^{-1}|$
  for $\Delta=|\bar M-\bar U\bar V|$.
  
If the   multipliers $F$ and $H$ are  orthogonal, then
   $F^{-1}=F^T$, $H^{-1}=H^T$, $|F^{-1}|=|H^{-1}|=1$,
   and therefore $|M-UV|=\Delta$.
   
   If in addition the multipliers  $F$ and $H$ are sufficiently sparse, then multiplication by them and by their transposes can be performed superfast. 

 If the LRA  $\bar U\bar V$  is actually CUR LRA of the matrix
 $\bar M$, then generally the recovered LRA $UV$ of the original matrix $M$ is not CUR LRA, but we can  extend it to CUR LRA of $M$ superfast by  applying the algorithms of Appendix \ref{slratocur}.

%------------------------------------------------------------------------------
 
\section{Randomized LRA Directed by Leverage 
 Scores}\label{scurlev}
 
%------------------------------------------------------------------------------
 
We begin with a simple observation. 
 
  \begin{theorem}\label{thorthmlt}  
 Algorithm \ref{alg1}  with an $n\times n$ orthogonal multiplier $H$ outputs an optimal LRA of an input matrix.
 \end{theorem}
 \begin{proof}
 If the matrix $H$ is  orthogonal, then so are the matrices $C_1=T_1^*H$ and $C_2=T_2^*H$ in (\ref{eqerrnrm})--(\ref{eqmmr}) as well. Hence $C_1^+=
 C_1^*=H^*T_1$, $C_2C_1^+=
 T_2^*HH^+T_1=T_2^*T_1=O$,  and therefore
 $|M-UV|=|\Sigma_2|$.
 \end{proof}
 
In view of this  observation one can speculate that Algorithms \ref{alg1} 
and  \ref{alg2} output reasonably            
close LRAs when we  apply them with orthogonal multipliers $F$ and $H$ where the ratios $k/m$ and $l/n$ are close to 1. Drineas et al., however, proved in \cite{DMM08}, by extending preceding work
(cf. Section \ref{sextrw}), that even where these ratios are not so large
such algorithms output nearly optimal LRA   whp under a proper 
random choice of sub-identity multipliers
$F$ and $H$ complemented with diagonal scaling, which Drineas et al. call {\em re-scaling}; namely Drineas et al.
proved  that whp
\begin{equation}\label{eqrmmrel}
 ||W-CUR||_F\le (1+\epsilon)\tilde\sigma_{r+1}
 \end{equation}
 for $\tilde\sigma_{r+1}$ of   Lemma 
 \ref{letrnc} and any fixed
   positive $\epsilon$.
     Let us briefly recall and then extend their results. 
 
%------------------------------------------------------------------------------ 
%------------------------------------------------------------------------------

\subsection{Randomized CUR LRA of Drineas et al.}\label{slrasmpsngv}
   
%-------------------------------------------------------- 
%------------------------------------------------------------------------------
   
%Given the top SVD of   
% an $m\times n$ matrix $W$ of numerical %%rank $r\ll \min\{m,n\}$, we can compute a %close CUR LRA of 
%  that matrix  whp by applying 
% randomized Algorithm  \ref{algsvdtocur} %or  \ref{algsvdtocur}a. Moreover  by  %following \cite{DMM08} and
%sampling sufficiently many 
%columns and rows,  we  improve the
 %output accuracy and yield nearly optimal  error bound 
% such  that (\ref{eqrmmrel}) holds.
  
Let $W_r=S^{(r)}\Sigma^{(r)}T^{(r)*}$
be top SVD  where 
$S^{(r)}\in \mathbb C^{m\times r}$,
 $\Sigma^{(r)}\in \mathbb C^{r\times r}$,
and $T^{(r)^*}=
 ({\bf t}_{j}^{(r)})_{j=1}^{n}\in \mathbb C^{r\times n}$.

 Fix  scalars
 $p_1,\dots,p_n$, and  $\beta$ such that
 \begin{equation}\label{eqlevsc}
 0<\beta\le 1,~ 
 p_j \ge (\beta/r)||{\bf t}_{j}^{(r)}||^2
~{\rm for}
~j=1,\dots,n,~{\rm and}~
\sum_{j=1}^np_j=1. 
 \end{equation}
%\begin{equation}\label{eqbetnrmz}
%0<\beta\le 1~{\rm and}~ 
%\sum_{j=1}^np_j=1
% \end{equation}
 Call the 
scalars $p_1,\dots,p_n$
  the {\em SVD-based
leverage scores} for the matrix $W$.  They stay invariant
   if we  pre-multiply 
 the matrix $T^{(r)*}$ by an  orthogonal matrix.
 Furthermore  
 \begin{equation}\label{eqlevsc1} 
p_j = ||{\bf t}_{j}^{(r)}||^2/r~{\rm for}~j=1,\dots,n~{\rm if}~\beta=1. 
\end{equation}

For any $m\times n$ matrix $W$,  
  \cite[Algorithm 5.1]{HMT11} computes the matrix 
 $V^{(r)}$ and leverage scores 
 $p_1,\dots,p_n$ by using $mn$ memory units  and $O(mnr)$ flops.
 
 Given an integer parameter $l$, $1\le l\le n$, and
leverage scores $p_1,\dots,p_n$,  Algorithms \ref{algsmplex}  and \ref{algsmplexp}, reproduced from \cite{DMM08},                                                                                                                                                                                                                                                                                                                                                                                                                                                                                                                                                         compute auxiliary sampling  
 and re-scaling matrices, $S=S_{W,l}$ and
$D=D_{W,l}$, respectively. Namely Algorithm \ref{algsmplex}   samples and re-scales exactly $l$ columns  of an input matrix $W$, while Algorithm  \ref{algsmplexp} samples and re-scales
 at most its $l$ columns in expectation
-- the $i$th column  with probability 
$p_i$ or $\min\{1, lp_i\}$.  Then \cite[Algorithms 1 and 2]{DMM08}
compute a CUR LRA of a matrix $W$ as follows.

%------------------------------------------------------------------------------

\begin{algorithm}\label{algcurlevsc} 
{\rm [CUR LRA by using SVD-based leverage scores.]} 

%------------------------------------------------------------------------------

\begin{description}
  
%------------------------------------------------------------------------------

\item[{\sc Input:}] A matrix
$W\in C^{m\times n}$ with
 $\nrank (W)=r>0$.

%------------------------------------------------------------------------------

\item[{\sc Initialization:}]
Choose two integers $k\ge r$ and $l\ge r$ %satisfying 
%(\ref{eqklmnr})
 and real $\beta$ 
and  $\bar \beta$  in the range $(0,1]$.

\item[{\sc Computations:}]

\begin{enumerate}
\item%1 
Compute the leverage scores
$p_1,\dots,p_n$ of (\ref{eqlevsc}).
\item%2
Compute  sampling   and  re-scaling  
matrices $S$ and $D$ 
by applying Algorithm \ref{algsmplex}
or \ref{algsmplexp}. 
Compute and output a CUR factor 
$C:=WS$.
\item%3 
Compute 
leverage scores
  $\bar p_1,\dots,\bar p_m$
satisfying relationships (\ref{eqlevsc})
under the following replacements:
$\{p_1,\dots,p_n\}\leftarrow 
\{\bar p_1,\dots,\bar p_m\}$,
 $W\leftarrow (CD)^T$ and
$\beta \leftarrow\bar \beta$. 
\item%4
By applying Algorithm \ref{algsmplex}
or \ref{algsmplexp} 
with these 
leverage scores 
%$\bar p_1,\dots,\bar p_m$
compute  $k\times l$ sampling 
matrix $\bar S$ and 
 $k\times k$ re-scaling matrix 
 $\bar D$.
 \item%5
Compute and output a CUR factor 
$R:=\bar S^TW$. 

\item%6
Compute and output a CUR factor
$U:=DM^+\bar D$ for 
$M:=\bar D\bar S^TWSD$.
\end{enumerate} 

%------------------------------------------------------------------------------

\end{description}
\end{algorithm}

%------------------------------------------------------------------------------
%------------------------------------------------------------------------------
 
{\bf Complexity estimates:}
Overall Algorithm \ref{algcurlevsc}
involves  $kn+ml+kl$ memory cells and 
$O((m+k)l^2+kn)$ flops 
in addition to $mn$ cells and $O(mnr)$  flops used at the stage of
computing SVD-based leverage scores
at stage 1. Except for that stage
the algorithm is superfast if 
$k+l^2\ll \min\{m,n\}$. 
 
Bound (\ref{eqrmmrel}) is expected to hold 
for the output of the algorithm
if we bound the integers $k$ and $l$
by combining \cite[Theorems 4 and 5]{DMM08} as follows. 
 
%------------------------------------------------------------------------------

\begin{theorem}\label{thdmm} {\rm [Randomized Estimates for the Output Error Norm
of  Algorithm \ref{algcurlevsc}.]}
Suppose that 

(i) 
$W\in C^{m\times n}$,
 $\nrank (W)=r>0$,  
 $\epsilon,\beta,\bar \beta\in(0,1]$, and $\bar c$ is a sufficiently large constant,

(ii) four integers $k$, $k_-$,
 $l$, and
 $l_-$ satisfy the bounds
\begin{equation}\label{eq3200}
0<l_-=3200 r^2/(\epsilon^2\beta)\le l\le n~{\rm and}~
0<k_-=3200 l^2/(\epsilon^2\bar\beta)\le k\le m
\end{equation}
or 
\begin{equation}\label{eqlog}
l_-=\bar c~r\log(r)/(\epsilon^2\beta)\le l\le n
~{\rm and}~
k_-=\bar c~l\log(l)/(\epsilon^2\bar\beta)
\le k\le m,
\end{equation} 
  
 (iii) we apply Algorithm \ref{algcurlevsc} 
  invoking at stages 2 and 4  either
 Algorithm \ref{algsmplex} 
 under (\ref{eq3200})
or  Algorithm \ref{algsmplexp}
 under (\ref{eqlog}).

Then
%in both cases  
%nearly optimal 
bound (\ref{eqrmmrel})
holds
 with a probability at least 0.7. 
\end{theorem}
   
%------------------------------------------------------------------------------

\begin{remark}\label{reepsbet} {\rm [Linking Sample Size and Output Error Bounds.]}

The bounds $k_-\le m$ and $l_-\le n$
imply that either 
$\epsilon^6\ge
3200^3 r^4/(m\beta^2\bar\beta)$
 and  $\epsilon^2\ge
 3200 r/(n\beta)$ if
 Algorithm \ref{algsmplex} is applied
 or
 $\epsilon^4\ge 
\bar c^2r\log(r)\log(\bar c r\log(r)/
(\epsilon^2\beta))/(m\beta^2\bar\beta)$
and $\epsilon^2\ge \bar c r\log(r)/(n\beta)$
if
 Algorithm \ref{algsmplexp} is applied
 for a sufficiently large constant $\bar c$.
\end{remark}
   
%------------------------------------------------------------------------------
%------------------------------------------------------------------------------

\begin{remark}\label{rebet}  {\rm [Linking the Bounds on the Sample Size with the Parameter $\beta$.]}
The  estimates 
$k_-$ and $l_-$ of (\ref{eq3200}) and
(\ref{eqlog}) are
minimized for 
  $\beta=\bar\beta=1$ and a fixed $\epsilon$.
   By decreasing the values of $\beta$ and $\bar\beta$
 we increase these two estimates
by factors of  $1/\beta$  and
$1/(\beta^2\bar\beta)$, respectively, and
then for any 
values of the leverage scores $p_i$ in the ranges (\ref{eq3200}) and
(\ref{eqlog}) we can  ensure 
 randomized error 
bound (\ref{eqrmmrel}). 
\end{remark}
      
Theorem \ref{thcrerr1} implies that the leverage scores
 are stable in perturbation of 
 an input matrix.

%------------------------------------------------------------------------------

\begin{remark}\label{resmplfc}
 We can obtain the simpler alternative expressions    
$U:=(\bar S^T WS)^+=
(S^TC)^+=
(RS)^+$  at stage 6 of Algorithm \ref{algcurlevsc}, at the price of a little weakening 
 numerical stability of the computation of a  nucleus of a  
 perturbed input matrix $W$.
\end{remark}  

%------------------------------------------

\subsection{Refinement of an LRA by using  
leverage scores}\label{sitreflra}

%------------------------------------------2
  
Recall that for $k\ll m$ and $l\ll n$  the CUR LRA algorithms of \cite{DMM08} are superfast except for the stage of the computation of leverage scores.  Moreover observe that the latter stage is also superfast if we are given a crude but reasonably close LRA 
$UV=\tilde M\approx M$  of an input
 matrix $M$ where $\rank(V)=r$. In view of the results of \cite{DMM08} this enables superfast
 refinement of such an LRA.  
 
Indeed first
 approximate top SVD of the matrix $W'$
 by  applying to it Algorithm \ref{alglratpsvd}, then fix a positive value 
$\beta'\le 1$ and compute leverage scores  $p_1',\dots,p_n'$  of the matrix $W'$ by applying
(\ref{eqlevsc}) with
$\beta'$  replacing $\beta$.
Perform all these computations superfast.

By virtue of Theorem \ref{thsngspc}
the computed  values 
$p_1',\dots,p_n'$  approximate 
 the leverage scores $p_1,\dots,p_n$ of the matrix $W$, and so  
we satisfy (\ref{eqlevsc})
for an input matrix $W$ and 
properly chosen parameters
$\beta$ and $\bar\beta$. Then we 
arrive at a CUR LRA satisfying (\ref{eqrmmrel}) if we sample 
 $k$ rows and $l$ columns 
 within the bounds of Theorem  \ref{thdmm},
 which are inversely proportional
 to $\beta^2\bar\beta$ and $\beta$, respectively. 

%------------------------------------------------------------------------------

\begin{remark}\label{resnrg}
The two approaches to LRA -- by means of C--A iterations and random sampling -- have been proposed and mostly advanced independently of one another, in the NLA and CS research communities, respectively. Their joint study in this paper reveals good chances for {\rm synergy}. For example, one can   
compute a crude but reasonably close LRA by applying C--A iterations and then refine it by applying sampling of the previous subsection. For another example, we can perform C--A 
iterations and  at every C--A step apply the algorithms of \cite{DMM08} supporting Theorem \ref{thsngvsmpl}. These applications are superfast because the inputs have small sizes.
\end{remark}

%------------------------------------------------------------------------------

\subsection{Superfast LRA with 
leverage scores for random and average inputs}\label{slrasmpav} 
  
%------------------------------------------------------------------------------

Next we  
%Algorithm \ref{algcurlevsc} to 
estimate  SVD-based leverage scores of the perturbations $M$ of  a
 factor-Gaussian  matrix $\tilde M$. 
 
 Theorem \ref{thsngspc}
reduces our task to the case of
factor-Gaussian matrix $M$.

The following theorem further reduces it to the case of a Gaussian matrix. 
 
 \begin{theorem}\label{thsngvgs} 
Let $M=GH$ for 
$G\in \mathbb C^{m\times r}$
and $H\in \mathbb C^{r\times n}$ and let
$r=\rank(G)=\rank(H)$.
%\le \min\{m,n\}$.
Then the matrices $M^T$  and $M$
share their SVD-based leverage scores with the matrices $G^T$ 
 and $H$, respectively,
 \end{theorem}
 
\begin{proof}
Let   
$G=S_G\Sigma_GT^*_G\in 
\mathbb C^{m\times r}$ 
and $H=S_H\Sigma_HT^*_H$
be SVDs.
 
Write $W:=\Sigma_GT^*_GS_H\Sigma_H$
and let $W=S_W\Sigma_WT^*_W$ be SVD.

Notice that 
$\Sigma_G$, $T^*_G$, $S_H$, and $\Sigma_H$
are $r\times r$ matrices. 

Consequently so are 
the matrices  $W$,
 $S_W$, $\Sigma_W$, and $T^*_W$.
 
Hence $M=\bar S_G\Sigma_W\bar T^*_H$
where $\bar S_G=S_GS_W$ and $\bar T^*_H=T^*_WT^*_H$ are unitary
matrices.

Therefore $M=\bar S_G\Sigma_W\bar T^*_H$
is SVD.

Hence the columns of the  orthogonal matrices 
$\bar S_G$ and $\bar T_H$
span  the  top rank-$r$ right singular spaces of the
matrices $M^T$ and $M$, respectively,
and so do the columns of the matrices 
$S_G$ and $T_H$ as well because 
$\bar S_G=S_GS_W$ and $\bar T^*_H=T^*_WT^*_H$ where $S_W$ and $T^*_W$
are $r\times r$  orthogonal matrices.
This proves the theorem.
\end{proof}

If $M=GH$ (resp. $M^T=H^TG^T$) is a right or two-sided
factor-Gaussian matrix,
% with a Gaussian matrix $H$, 
then  with probability 1 the matrices $W$ and $H$ 
(resp. $M^T$ and $G^T$) share their  leverage scores  by virtue of Theorem
\ref{thsngvgs}. If we only know that the matrix $M$ either a left or a right factor-Gaussian matrix, we  can apply 
Algorithm \ref{algcurlevsc} to both matrices $M$ and 
$M^T$ and in at least one case reduce the 
computation of the leverage scores to the case of Gaussian matrix. 
 
Let us outline our further steps
of the estimation of the leverage scores
provided that  $r\ll n$.
\begin{outline}\label{out1} 
 Recall from 
\cite[Theorem 7.3]{E89} or \cite{RV09} that  
$\kappa(G)\rightarrow 1$ as 
$r/n\rightarrow 0$ for
$G\in \mathcal G^{r\times n}$.
It follows that  for $r\ll n$ the matrix
$G$ 
is 
close to a scaled orthogonal matrix whp, 
and hence within a  
factor $\frac{1}{\sqrt n}$
it is close to the orthogonal matrix 
$T^*_G$ of 
its right singular space whp.
Therefore the leverage scores $p_j$ of  a Gaussian matrix $G=({\bf g}_j)_{j=1}^n$ are
close to the values $\frac{1}{rn}||{\bf g}_j||^2$, $j=1,\dots,n$.
They, however, are invariant in $j$ and  close to $1/n$
 for all $j$ whp. This choice trivializes the approximation of the leverage scores
of a Gaussian matrix and hence of a factor-Gaussian matrix.  
  As soon as this bottleneck stage of  Algorithm \ref{algcurlevsc} has been made trivial, the entire algorithm becomes superfast.
  \end{outline}
 
Next we elaborate upon this
outline.

\begin{lemma}\label{randomprojection} 
Suppose that 
$G\in \mathcal G^{n\times r}$,
${\bf u}\in \mathbb R^r$, 
${\bf v} = \frac{1}{\sqrt{n}}G{\bf u}$,
and $r\le n$. Fix $\bar \epsilon > 0$.
 Then
$$
\mathrm{Probability}\{
(1-\bar \epsilon)||{\bf u}||^2 \le ||{\bf v}||^2 \le (1+\bar \epsilon)||{\bf u}||^2
\}
\ge 1 - 2e^{-(\bar \epsilon^2-
\bar\epsilon^3)\frac{n}{4}}.
$$
\end{lemma}
\begin{proof}
See \cite[Lemma 2]{AV06}.
\end{proof}
 
\begin{lemma}\label{leprtbsvd} 
 Fix the spectral or Frobenius norm 
 $|\cdot|$
and let $W=S_W\Sigma_WT_W^*$ be SVD.
Then $S_WT_W^*$ is an  orthogonal matrix and
$$|W-S_WT_W^*|\le |WW^*-I|.$$
\end{lemma}
 \begin{proof}
 $S_WT_W^*$ is an  orthogonal matrix because 
 both matrices $S_W$ and $T_W^*$ are  orthogonal and at least one of them is a square matrix.
 
 Next observe that
  $W-S_WT_W^*=S_W\Sigma_WT_W^*-S_WT_W^*=
 S_W(\Sigma_W-I)T_W^*$, and so 
 $$|W-S_WT_W^*|=|\Sigma_W-I|.$$
Likewise 
$WW^*-I=S_W\Sigma_W^2S_W^*-I=
S_W(\Sigma_W^2-I)S_W^*$,
and so
$$|WW^*-I|=|\Sigma_W^2-I|.$$
 Complement these equations for the norms
 with the inequality   
$$|\Sigma_W^2-I|=|\Sigma_W-I|~|\Sigma_W+I|\ge |\Sigma_W-I|,$$ which holds
 because $\Sigma_W$ is a diagonal matrix 
 having only nonnegative entries.
  \end{proof}
 
\begin{lemma}\label{gaussianNearOrthogonal}
Suppose that $\epsilon >0$ and that
 $n$ and $r< n$ are two integers such that  $n$ is sufficiently large and exceeds
 $c \epsilon^{-2}r^4\log_2(r)$
 for a positive constant $c$.
Furthermore let $G=({\bf g}_j)_{j=1}^n\in \mathcal G^{r\times  n}$.
Then whp 
%\begin{equation}\label{eqggti}
$$\Big |\Big |\frac{1}{n}GG^* - I_r\Big |\Big |^2_F < \epsilon.$$
%\end{equation} 
%Furthermore, let $\epsilon < 3r^2$ and let 
\end{lemma}

\begin{proof}
Write
$\epsilon = \frac{2}{3r^2}\bar\epsilon$. 
Then $\bar\epsilon =\frac{3}{2}r^2 \epsilon $ and
$n =\frac{9c}{4\bar \epsilon^{2}}r^6\log_2(r)$.
 
%Let ${\bf g}_i$ be the $i$th row vector %of matrix $\frac{1}{\sqrt{n}}G$.

Let ${\bf e}_j$ denote the $j$th  column of the identity  matrix $I_r$.
Apply Lemma \ref{randomprojection} for 
${\bf u}$ equal to the vectors ${\bf e}_j$ and ${\bf e}_i-{\bf e}_j$
and to ${\bf v}=\frac{1}{\sqrt{n}}{\bf g}_j$ and for 
$i,j=1,\dots,r$ where $i\neq j$, substitute $||{\bf e}_j||=1$ and $||{\bf e}_i-{\bf e}_j||^2=2$
for all $j$ and all $i\neq j$, and  deduce that  
 whp 
 $$1-\bar \epsilon<||{\bf g}_j||^2/n<1+\bar \epsilon~{\rm and}~2-\bar \epsilon< ||{\bf g}_i-{\bf g}_j||^2/n<2+\bar \epsilon$$ 
for all  $j$ and for all $i\ne j$.
 
Now, since  the  $(i,j)$th entry of the matrix $GG^*$ is given by  ${\bf g}_i^*{\bf g}_j$, deduce that whp
$$
\Big |\Big |\frac{1}{n}GG^* -I_r\Big |\Big |^2_F\le \Big (\frac{3}{2}r^2-\frac{r}{2}\Big )\bar \epsilon<
\frac{3}{2}r^2\bar \epsilon=\epsilon.
$$ 
\end{proof}

 Combine Lemmas \ref{leprtbsvd}
  and
 \ref{gaussianNearOrthogonal}
 for $W=\frac{1}{\sqrt n}G$ and obtain
 the following result.
\begin{corollary}\label{cosvdg}
Under the assumptions of Lemma
\ref{gaussianNearOrthogonal} let
%\begin{equation}\label{eqsvdg}
$\frac{1}{\sqrt{n}}G = S\Sigma T^*$
%\end{equation} 
be  SVD. Then whp
%\begin{equation}\label{eqguvt}
$$\Big |\Big | \frac{1}{\sqrt{n}}G - ST^*\Big |\Big |^2_F <~\epsilon.$$
%\end{equation} 
\end{corollary}  

\begin{remark}\label{resvdg} 
Under the assumptions of the corollary
$\Sigma\rightarrow I_r$ as $\epsilon\rightarrow 0$, and then
the norm $||(\Sigma + I_r)^{-1}||_F$
and consequently the ratio
$\frac{||\Sigma - I_r||_F}{||\Sigma^2 - I_r||_F}$ converge to 1/2.
\end{remark}
   
\begin{theorem}\label{leverage_score_approximation}
Given two integers $n$ and $r$ and a
positive $\epsilon$ satisfying the assumptions of Lemma \ref{gaussianNearOrthogonal},
a Gaussian matrix 
$G=({\bf g}_j)_{j=1}^n\in \mathcal G^{r\times  n}$,
and  SVD $\frac{1}{\sqrt n}G=S\Sigma T^*$,  write $T^*=({\bf t}_j)_{j=1}^n$ and 
$\beta=1$ and
define the SVD-based leverage scores 
of the matrix $\frac{1}{\sqrt n}G$,
that is,
$p_j=||{\bf t}_j||^2/r$ for $j=1,\dots,n$
(cf. (\ref{eqlevsc}) and (\ref{eqlevsc1})).
Then whp
$$\Big |p_j - \frac{||{\bf g}_j||^2}{nr}\Big | \le \frac{\epsilon}{r}   ~ \mathrm{for} ~ j = 1,...,n.$$
\end{theorem}

\begin{proof}
Notice that 
$||S{\bf t}_j||=||{\bf t}_j||$ for all $j$
since $S$ is a square  orthogonal matrix;
deduce from Corollary \ref{cosvdg} that whp 
$$\frac{1}{n}||{\bf g}_j||^2-
 ||S{\bf t}_j||^2 < \epsilon~\mathrm{for} ~ i = 1,...,n.$$
\end{proof}

%-----------------------------------------------------------------------------------------------------------------------------------------------

\begin{remark}\label{rebt}
The estimate of the theorem is readily extended to the case where the leverage scores are defined by (\ref{eqlevsc}) rather than (\ref{eqlevsc1}).
\end{remark}

%-----------------------------------------------------------------------------------------------------------------------------------------------

Now observe that the norms $||{\bf g}_j||$
are  iid chi-square random variables $\chi^2(r)$
and therefore are quite 
strongly concentrated in a reasonable range about the expected value of such a variable. 
Hence  we obtain reasonably good approximations to SVD-based  leverage scores for a Gaussian matrix 
$G=({\bf g}_j)_{j=1}^n\in \mathcal G^{r\times  n}$
by choosing  
 $p_j=1/n$  for all $j$,
and then we satisfy bounds (\ref{eqlevsc})  and consequently
%bound
 (\ref{eqrmmrel}) by choosing a reasonably small positive value $\beta$. 

 Let us supply some further 
  details.

\begin{corollary}
Given two integers $n$ and $r$ and a
positive $\epsilon$ satisfying the assumptions of Theorem \ref{leverage_score_approximation},
a Gaussian matrix 
$G=({\bf g}_j)_{j=1}^n\in \mathcal G^{r\times  n}$, 
and denote its SVD-based leverage scores 
as $p_j$ for $j = 1,...,n$. 
Fix $0<\beta < r(n\epsilon + r)^{-1}$ such that
$\Big(\frac{1}{\beta}-\frac{n\epsilon + r}{r}\Big) = \Theta(\ln{n})$, 
then whp
$$
\frac{1}{n} > \beta p_j    ~ \mathrm{for} ~ j = 1,...,n.
$$
\end{corollary}

\begin{proof}
Deduce from 
(\ref{eqchi})  that
\begin{align*}
\mathrm{Probability}\Big\{ \frac{1}{n} >
\beta \Big( \frac{||{\bf g}_j||^2}{nr} + \frac{\epsilon}{r} \Big) \Big\}
&= 1 - \mathrm{Probability}\Big\{ ||{\bf g}_j||^2 - r \ge
\Big(\frac{1}{\beta}-\frac{n\epsilon + r}{r}\Big)r \Big\} \\
&\ge 1 - \exp{(-f(\beta))}
\end{align*}
for $f(\beta)$ being the positive solution of 
$2\sqrt{rx} + 2rx = \Big(\frac{1}{\beta}-\frac{n\epsilon + r}{r}\Big)r$ 
and any $j \in\{1,2,\dots, n\}$.

Furthermore the random variables  $||{\bf g}_j||^2$
 are independent, and therefore
\begin{align*}
\mathrm{Probability}\Big\{ \frac{1}{n} >
\beta \Big( \frac{||{\bf g}_j||^2}{nr} + \frac{\epsilon}{r} \Big) 
\textrm{ for all}~j = 1, 2,\dots, n\Big \}
&\ge \big( 1 - \exp(-f(\beta)) \big)^n \\
&\ge 1 - \exp\Big(-f(\beta) + \ln{n} \Big).
\end{align*}

It can be verified easily that $f(\beta)$ is dominated by $\Big(\frac{1}{\beta}-\frac{n\epsilon + r}{r}\Big)$ and combine this with the result from Theorem 
\ref{leverage_score_approximation} 
and conclude that whp
$$
\frac{1}{n} > \beta p_j    ~ \mathrm{for} ~ j = 1,...,n.
$$
\end{proof}

%By combining the probability density %function  
%of the $\chi^2$-function with Markov %inequality one can specify our argument %quantitatively.

%------------------------------------------------------------------------------

We have completed our formal support for Outline \ref{out1} and arrived at the following result.
 \begin{theorem}\label{ththrdsrit}
 Suppose that the algorithms of \cite{DMM08} have been applied with uniformly distributed leverage scores. Then the algorithms become superfast and whp output accurate CUR LRA
 provided that the input matrix $M$ is
 close enough to a factor-Gaussian matrix of rank $r$, $r\gg \ln n$, and
 $k$ and $l$ satisfy the bounds of Theorem \ref{thdmm}.
 \end{theorem}
         
%------------------------------------------------------------------------------

\subsection{Superfast CUR LRA with   sampling fewer rows and columns}\label{smd09}

%------------------------------------------------------------------------------
  
By virtue of
 Theorem \ref{thdmm} 
Algorithm \ref{algcurlevsc} 
outputs nearly optimal CUR LRA,
but the supporting estimates for the integers   
$l$ and particularly $k$ are
fairly large,
even for relatively large values of 
$\epsilon\le 1$ (cf. Remark \ref{reepsbet}). Such estimates for the integers $l$ and $k$,
however, are overly pessimistic according to the tests 
 in \cite[Section 7]{DMM08}: in these tests  Algorithm \ref{algcurlevsc} has computed accurate CUR LRA of various real world inputs
  by using just reasonably large integers $l$ and $k$, dramatically exceeded by their upper estimates in Theorem \ref{thdmm}, that we reproduced from \cite{DMM08}, as well as by the upper estimates in all subsequent papers.
   
In Sections \ref{simppost}  and \ref{simppre} we proved that
the outputs of 
superfast random sampling algorithms are accurate whp 
in the case of a factor-Gaussian input and hence for  the
average case input. This study  
 covers application of the algorithms of \cite{DMM08} with any non-degenerating leverage scores  where we can use reasonable
 numbers $k$ and $l$ of the row and  column samples.
  This provides
formal support for the empirical results in \cite{DMM08}, which  is their only formal support available so far. 
 
%------------------------------------------------------------------------------
 
 \section{Generation of Multipliers}\label{sgnrml}

%------------------------------------------------------------------------------
%------------------------------------------------------------------------------

\subsection{Section overview}\label{sovrv}

%------------------------------------------------------------------------------
  
 Next we describe various families of well-conditioned  sparse and structured multipliers of full rank (mostlyreal orthogonal or complex  unitary). According to our study, sampling with them enables accurate LRA of average  and whp random inputs, and this can be achieved faster  than  with  SRHT and SRFT multipliers; typically the multipliers remain sparse enough so that multiplication by them stays superfast
  (see classes 13--17 of Section \ref{s17m}).

We proceed in the following order.
Given 
%an input $m\times n$ matrix $M$ and
two integers $l$ and $n$,  $l\ll n$, we first generate
four classes of
 sparse primitive $n\times n$ orthogonal matrices,
then combine them into some basic families of $n\times n$ matrices 
(we denote them $\widehat B$ in this section),
and finally define multipliers  $B$ as  $n\times l$  submatrices 
made up of $l$ columns, which can be fixed (e.g., leftmost) or 
chosen at random. 
In this case $\kappa(B)\le \kappa(\widehat B)$ (cf. \cite[Theorem 8.6.3]{GL13}), and if  the matrix $\widehat B$ is  orthogonal, then   so is the matrix 
$B$. 

 In the next subsection we cover the four primitive types of square matrices  for generation of multipliers.
 
In Section \ref{shad} we cover Family (i) of real orthogonal multipliers linked to Hadamard transforms
and of complex  unitary multipliers, involving complex roots of unity and  linked to the discrete Fourier transform.

These matrices are further linked to Families (ii) 
and (iii) of circulant and sparse multipliers of Section \ref{scrcsp} and
\ref{scrcabr}. 

In Section \ref{sinvchh} we cover 
multipliers obtained by inverting bidiagonal matrices. 

In Section \ref{sflprnd} 
we estimate the number of random variables in all  these multipliers and number of flops involved in their multiplication by a dense real and complex vectors.

In Sections \ref{sbscfml} and \ref{smngflr} we recall some other families of  multipliers.
%and recipes for generating them.

In Section \ref{sgpbd} 
we study approximation of a Gaussian matrix by the products of random bidiagonal and permutation matrices.

The readers may propose many other efficient multipliers.

%------------------------------------------------------------------------------

\subsection{Square matrices of four  primitive types}\label{sdfcnd}

%------------------------------------------------------------------------------
  
%Next we describe four matrix primitives of size $n\times n$. 

%------------------------------------------------------------------------------

%$\big (\begin{smallmatrix} 1& ~~1\\ 1&-1\end{smallmatrix}\big )$

\begin{enumerate}
  \item%1
A fixed  or random
 {\em permutation matrix} $P$.
Their block submatrices form the class of
 CountSketch matrices from
the data stream literature (cf. \cite[Section 2.1]{W14},
\cite{CCF04}, \cite{TZ12}).
\item%2
A {\em diagonal matrix}  $D=\diag(d_i)_{i=0}^{n-1}$, with 
fixed  or random
diagonal entries $d_i$ such that
$|d_i|=1$ for all $i$ (and so all $n$ entries $d_i$ lie 
on the unit circle $\{x:~|z|=1\}$, 
being either nonreal or  $\pm 1$).
\item%3
An $f$-{\em circular shift matrix} 

$$Z_f=\begin{pmatrix}
        0  ~ & \dots&     ~  ~\dots   ~ ~& 0 & ~f\\
        1   & \ddots    &   &    ~ 0 & ~0\\
        \vdots         & \ddots    &   \ddots & ~\vdots &~ \vdots  \\
         \vdots   &    &   \ddots    &~ 0 & ~0 \\
        0   & ~  \dots &      ~ ~~  \dots   ~~ ~ ~ & 1 & ~0 
    \end{pmatrix} 
%=\big (\begin{smallmatrix} {\bf 0}^*  & ~f\\I_{n-1}   & ~{\bf 0}\end{smallmatrix}\big )
$$
and its transpose $Z_f^*$ for a scalar $f$ such that either $f=0$ or $|f|=1$.
We write $Z=Z_0$, call $Z$ {\em unit down-shift matrix}, and call the special permutation 
matrix $Z_1$ the
 {\em unit circulant matrix}. 
\item%4
A   $2s\times 2s$ {\em Hadamard  primitive matrix}
%\begin{equation}\label(eqprimhad)
$H^{(2s)}=\big (\begin{smallmatrix} I_s  & ~I_s  \\
I_s   & -I_s\end{smallmatrix}\big )$
%\end{equation}
 for a positive integer $s$
 (cf. \cite{M11},  \cite{W14}).
\end{enumerate}

All our primitive  $n\times n$ matrices are very sparse and can be
pre-multiplied by a vector
in at most $2n$ flops.
Except for the matrix $Z$, they are unitary or real orthogonal, and so is any  $n\times l$ submatrix of  $Z$ of full rank $l$.
Next 
we 
 combine primitives 1--4 into
families of $n\times n$ 
sparse and/or structured  
 multipliers  $B$.
   
%------------------------------------------------------------------------------

\subsection{Family (i): 
%sparse ARSPH matrices
multipliers based on the Hadamard and Fourier processes}\label{shad}

%------------------------------------------------------------------------------

  We first recall the  following recursive definition of  
dense and orthogonal (up to scaling by constants) $n\times n$ matrices $H_n$ 
of {\em Walsh-Hadamard transform} for $n=2^k$
(cf.   \cite[Section 3.1]{M11} and our Remark \ref{recmb}): 
\begin{equation}\label{eqrfd}
H_{2s}=\begin{pmatrix}
H_{s} & H_{s} \\
H_{s} & -H_{s}
  \end{pmatrix}
%H_{2^{k-d}=\begin{pmatrix}
%I_{2^{k-d} & I_{2^{k-d}  \\
%I_{2^{k-d} & -I_{2^{k-d} 
%\end{pmatrix},
\end{equation}
for $s=2^i$, $i=0,1,\dots,k-1$,
  and the Hadamard  primitive matrix
 $H_{2}=H^{(2)}=
\big (\begin{smallmatrix} 1  & ~~1  \\
1   & -1\end{smallmatrix}\big )$  of type 4
%(cf. (\ref{eqprimhad})
 for $s=1$.  
 
For demonstration, here are the matrices $H_4$ and $H_8$
shown with their entries,  
$$H_{4}=\begin{pmatrix}
1  &  1 & 1  &  1 \\
1  & -1  & 1  &  -1 \\ 
1  &  1 & -1  &  -1 \\
1  & -1  & -1  &  1 
\end{pmatrix}~{\rm and}~ H_{8}=\begin{pmatrix}
1  &  1 & 1  &  1 & 1  &  1 & 1  &  1 \\
1  & -1  & 1  &  -1 & 1  & -1  & 1  &  -1\\ 
1  &  1 & -1  &  -1 & 1  & 1  & -1  &  -1\\
1  & -1  & -1  &  1 &  1  & -1  & -1  &  1 \\
1  &  1 & 1  &  1 & -1  &  -1 & -1  &  -1 \\
1  & -1  & 1  &  -1 & -1  &  1 & -1  &  1\\ 
1  &  1 & -1  &  -1 & -1  &  -1 & 1  &  1\\
1  & -1  & -1  &  1 & -1  &  1 & 1  &  -1
\end{pmatrix},$$
 but for larger dimensions $n$, 
recursive representation (\ref{eqrfd}) enables much faster 
 pre-multiplication of a matrix $H_{n}$ by a vector, namely it is sufficient
to use $nk$ additions
and subtractions for $n=2^k$,
and this representation can be
 efficiently computed in  parallel
 (cf. 
Remark \ref{resprs0}).

Next we sparsify this matrix by defining it by a 
 shorter recursive process, that is, 
by fixing a  {\em recursion depth} $d$, $1\le d<k$, and  applying equation (\ref{eqrfd}) where
$s=2^is_0$, $i=k-d,k-d+1,\dots,k-1$,  and $H_{s_0}=I_{s_0}$ for 
$n=2^ds_0$.
%is the $n\times n$ Hadamard  primitive matrix $H^{(2s)}$ of type 4.
For  two positive integers $d$ and $s$,
 we denote the resulting $n\times n$ matrix
$H_{n,d}$ and  for $1\le d< k$ call it 
 $d$--{\em Abridged Hadamard
 (AH) 
  matrix}. 
In particular, 
$$H_{n,1}=\begin{pmatrix}
I_s  &  I_s  \\ 
I_s  & -I_s  
\end{pmatrix},~
{\rm for}~n=2s;~
H_{n,2}=\begin{pmatrix}
I_s  &  I_s & I_s  &  I_s \\
I_s  & -I_s  & I_s  &  -I_s \\ 
I_s  &  I_s & -I_s  &  -I_s \\
I_s  & -I_s  & -I_s  &  I_s
\end{pmatrix},~{\rm for}~n=4s,~{\rm and}$$
$$H_{n,3}=\begin{pmatrix}
I_s  &  I_s & I_s  &  I_s & I_s  &  I_s & I_s  &  I_s \\
I_s  & -I_s  & I_s  &  -I_s & I_s  & -I_s  & I_s  &  -I_s\\ 
I_s  &  I_s & -I_s  &  -I_s & I_s  & I_s  & -I_s  &  -I_s\\
I_s  & -I_s  & -I_s  &  I_s &  I_s  & -I_s  & -I_s  &  I_s \\
I_s  &  I_s & I_s  &  I_s & -I_s  &  -I_s & -I_s  &  -I_s \\
I_s  & -I_s  & I_s  &  -I_s & -I_s  &  I_s & -I_s  &  I_s\\ 
I_s  &  I_s & -I_s  &  -I_s & -I_s  &  -I_s & I_s  &  I_s\\
I_s  & -I_s  & -I_s  &  I_s & -I_s  &  I_s & I_s  &  -I_s
\end{pmatrix},~{\rm for}~n=8s.$$
For a fixed $d$, the matrix  $H_{n,d}$
is still orthogonal up to scaling,
 has $2^d$ nonzero entries 
in every row and  column, and hence
 is sparse unless 
$k-d$ is a small integer.
Then again,
for larger dimensions $n$, 
 we can pre-multiply such a  matrix by a vector much faster
if we represent it via
recursive process (\ref{eqrfd}), by using  just
  $dn$
additions/subtractions.

We similarly obtain sparse matrices
by shortening a recursive process
of the generation of the $n\times n$ matrix  $F_n$
of {\em discrete Fourier transform (DFT)} at $n$ points,  for $n=2^k$:
\begin{equation}\label{eqdft}
F_n=(\omega_{n}^{ij})_{i,j=0}^{n-1},~{\rm for}~n=2^k~{\rm and~a~primitive}~
n{\rm th~root~of~unity}~\omega_{n}=\exp(2\pi {\bf i}/n),~{\bf i}=\sqrt {-1}.
\end{equation}
In particular
$F_{2}=H^{(2)}=
\big (\begin{smallmatrix} 1  & ~~1  \\
1   & -1\end{smallmatrix}\big )$,
$$F_{4}=\begin{pmatrix}
1  &  1 & 1  &  1 \\
1  & {\bf i}  & -1  &  -{\bf i} \\ 
1  &  -1 & 1  &  -1 \\
1  & -{\bf i}  & -1  &  {\bf i}
\end{pmatrix},~{\rm and}~ F_{8}=\begin{pmatrix}
1  &  1 & 1  &  1 & 1  &  1 & 1  &  1 \\
1 &\omega_{8}&{\bf i}&{\bf i}\omega_{8}&-1& -\omega_{8}& -{\bf i}&-{\bf i}\omega_{8}\\ 
1  &  {\bf i} & -1  &  -{\bf i} & 1  & {\bf i}  & -1  &  -{\bf i}\\
1 &{\bf i}\omega_{8}& -{\bf i}&\omega_{8}& -1 &-{\bf i}\omega_{8}&{\bf i}&-\omega_{8} \\
1  &  -1 & 1  &  -1 & 1  &  -1 & 1  &  -1 \\
1  & -\omega_{8}& {\bf i}&-{\bf i}\omega_{8}& -1&\omega_{8}& -{\bf i}& {\bf i}\omega_{8}\\ 
1  &  - {\bf i} & -1  &   {\bf i} & 1  &  - {\bf i} & -1  &   {\bf i}\\
1  & - {\bf i}\omega_{8}& -{\bf i}&-\omega_{8} & -1&{\bf i}\omega_{8}&{\bf i}& \omega_{8}
\end{pmatrix}.$$

The matrix $F_n$ is unitary up to scaling by $\frac{1}{\sqrt n}$.
We can pre-multiply it by a vector 
by using $1.5nk$ flops, and we can efficiently parallelize this computation
 if instead of the representation by the entries we apply the following recursive representation
(cf. \cite[Section 2.3]{P01}  and our Remark \ref{recmb}):\footnote{This is a representation of FFT, called decimation in frequency (DIF) radix-2 representation. 
Transposition turns it into an alternative
representation of FFT, called
decimation 
in time (DIT) radix-2 representation.}
\begin{equation}\label{eqfd}
F_{2s}=
\widehat P_{2s}
\begin{pmatrix}F_{s}&~~F_{s}\\ 
F_{s}\widehat D_{s}&-F_{s}\widehat D_{s}\end{pmatrix},~
\widehat D_{s}=\diag(\omega_{n}^{i})_{i=0}^{n-1}.
\end{equation}
Here $\widehat P_{2s}$ is the matrix of odd/even permutations 
 such that 
$\widehat P_{2^{i}}({\bf u})={\bf v}$, ${\bf u}=(u_j)_{j=0}^{2^{i}-1}$, 
${\bf v}=(v_j)_{j=0}^{2^{i}-1}$, $v_j=u_{2j}$, $v_{j+2^{i-1}}=u_{2j+1}$, 
$j=0,1,\ldots,2^{i-1}-1$;
$s=2^i$, $i=0,1,\dots,k$,  and
$F_{1}=(1)$ is the scalar 1.
%$H_{2}=\big (\begin{smallmatrix} 1  & ~~1  \\
%1   & -1\end{smallmatrix}\big ).$

We can sparsify  this matrix by defining it by a 
 shorter recursive process, that is, 
by fixing a recursion depth $d$, $1\le d<k$,  
replacing $F_{s}$ for $s=n/2^{d}$
by the identity matrix $I_{s}$,
and then  applying equation (\ref{eqfd}) for
$s=2^i$, $i=k-d,k-d+1,\dots,k-1$.
For $1\le d<k$ and $n=2^ds$,   
we denote the resulting $n\times n$ matrix
$F_{n,d}$ and call it 
 $d$-{\em Abridged   Fourier
 (AF)
  matrix}. It is also unitary (up to scaling),
 has $s=2^d$ nonzero entries 
in every row and column, and thus is
  sparse unless 
$k-d$ is a small integer.  
Then again its
pre-multiplication by a vector involves just  $1.5dn$
flops
and enables
highly efficient  
 parallel implementation
if we rely on recursive 
representation (\ref{eqfd}). 

%$F_{n,1}=$, $F_{n,2}=$, 
%and $F_{n,3}=$.

\medskip

By applying fixed or random permutation and scaling to AH matrices
$H_{n,d}$ and  AF matrices $F_{n,d}$, we obtain the 
families of 
$d$--{\em Abridged Scaled and Permuted 
 Hadamard (ASPH)} matrices, $PDH_n$, and 
$d$--{\em Abridged Scaled and Permuted 
Fourier  (ASPF)} $n\times n$
 matrices, $PDF_n$ where $P$ and $D$ are 
two 
matrices of permutation and diagonal scaling of
primitive classes 1 and 2, respectively. 
Likewise  we define the families of
ASH, ASF, APH, and APF matrices, 
 $DH_{n,d}$, $DF_{n,d}$, $PH_{n,d}$, and  $PF_{n,d}$, respectively.
Each random permutation or scaling  
 contributes up to $n$ random parameters.  

\begin{remark}\label{recmb}
The following equations are equivalent  to (\ref{eqrfd}) and (\ref{eqfd}):
%combinations of our  primitives 1--4:
$$H_{2s}=\diag(H_s,H_s)H^{(2s)}~{\rm and}~F_{2s}=
\widehat P_{2s} \diag(F_{s},F_{s}\widehat D_s)H^{(2s)}.$$
Here $H^{(2s)}$ denotes a $2s\times 2s$ Hadamard's primitive matrix of type 4.
By extending the latter recursive representation we can 
define matrices that involve more random parameters. Namely we can
 recursively  incorporate random  permutations and diagonal scaling as follows:
\begin{equation}\label{eqhfspd}
\widehat H_{2s}=P_{2s}D_{2s}\diag(\widehat H_s,\widehat H_s)H^{(2s)}~{\rm and}~
\widehat F_{2s}=
P_{2s}D_{2s} \diag(F_{s},F_{s}\widehat D_s)H^{(2s)}.
\end{equation} 
Here $P_{2s}$ are $2s\times 2s$ random permutation matrices of primitive class 1,
while $D_{2s}$ are $2s\times 2s$ random matrices of diagonal scaling of primitive class 2,
for all $s$. Then again we define $d$--abridged matrices $\widehat H_{n,d}$
and $\widehat F_{n,d}$ by applying only $d$ recursive steps (\ref{eqhfspd})
initiated at the primitive matrix $I_{s}$, for $s=n/2^{d}$.

With these recursive steps we can pre-multiply matrices $\widehat H_{n,d}$
and $\widehat F_{n,d}$ by a vector
by using  at most $2dn$ additions and subtractions
and at most $2.5dn$ flops,
respectively, provided that $2^{d}$ divides
 $n$.
\end{remark}
 
%------------------------------------------------------------------------------

\subsection{ $f$-circulant, sparse   
$f$-circulant,  and uniformly  sparse matrices}\label{scrcsp}

%------------------------------------------------------------------------------

An
 {\em $f$-circulant matrix}
$$Z_f({\bf v})
=\begin{pmatrix}v_0&fv_{n-1}&\cdots&fv_1\\ v_1&v_0&\ddots&\vdots\\ \vdots&\ddots&\ddots&fv_{n-1}\\ v_{n-1}&\cdots&v_1&v_0\end{pmatrix}=\sum_{i=0}^{n-1}v_iZ_f^i$$
for  
the matrix $Z_f$ of $f$-circular shift,
is defined by a 
scalar $f\neq 0$  and by
the first column ${\bf v}=(v_i)_{i=0}^{n-1}$ and
 is called {\em  circulant} if $f=1$ and {\em skew-circulant} if $f=-1$.
Such a matrix is nonsingular with probability 1 (see Theorem \ref{thrnd}) and whp 
is  well-conditioned \cite{PSZ15}
if  $|f|=1$  and if the vector ${\bf v}$ is Gaussian.

% or is random and uniformly bounded.

\begin{remark}\label{restr}
One can compute the product of an $n\times n$ circulant  matrix with 
an $n\times n$ Toeplitz or Toeplitz-like matrix
by using $O(n\log (n))$ flops (see \cite[Theorem 2.6.4 and Example 4.4.1]{P01}).
\end{remark}

{\bf FAMILY (ii)}  of  {\em sparse} $f$-{\em circulant matrices} 
$\widehat B=Z_f({\bf v})$ is
 defined by a fixed or random scalar $f$, $|f|=1$, and by
the  first column having exactly 
$s$ nonzero entries, for $s\ll n$.
The positions and the values of nonzeros can be
 randomized (and then the matrix would depend on up to $2n+1$ random values).

We can  pre-multiply such a matrix by a vector by using at most 
$(2s-1)n$ flops  or, in the real case where $f=\pm 1$ and $v_i=\pm 1$
for all $i$, by using at most
$sn$ additions and subtractions. 

The same cost estimates apply in the case of the  generalization 
of $Z_f({\bf v})$  to
a {\em uniformly   
 sparse matrix} with exactly $s$ nonzeros entries, $\pm 1$, 
in every row and in every column for $1\le s\ll n$.
Such a matrix
is the sum $\widehat B=\sum_{i=1}^s\widehat D_iP_i$
for fixed or random matrices  $P_i$  and $\widehat D_i$ of  primitive  types 1 and 2,
respectively.

%------------------------------------------------------------------------------

\subsection{Abridged   
$f$-circulant matrices}\label{scrcabr}

%------------------------------------------------------------------------------

First recall the following well-known expression for a
 $g$-circulant matrix: 
$$Z_g({\bf v})=\sum_{i=0}^{n-1}v_iZ_g^i=D_f^{-1}F_n^*DF_nD_f$$
where $g=f^n$,  $D_f=\diag(f^i)_{i=0}^{n-1}$, ${\bf v}=(v_i)_{i=0}^{n-1}=(F_nD_f)^{-1}{\bf u}$, 
${\bf u}=(u_i)_{i=0}^{n-1}$, and
$D=\diag(u_i)_{i=0}^{n-1}$
(cf. \cite[Theorem 2.6.4]{P01}).
For $f=1$, the expression is simplified: $g=1$, $D_f=I_n$, and
$Z_g({\bf v})=\sum_{i=0}^{n-1}v_iZ_1^i$
is a circulant matrix:
\begin{equation}\label{eqcrcl}
Z_1({\bf v})=F_n^*DF_n,~D=\diag(u_i)_{i=0}^{n-1},~{\rm for}~ 
{\bf u}=(u_i)_{i=0}^{n-1}=F_n{\bf v}.
\end{equation}
Pre-multiplication of an $f$-circulant matrix by a vector 
is reduced to pre-multiplication of each of the matrices $F$
and $F^*$ by a vector and in  addition 
to performing $4n$ flops (or $2n$ flops in case of a circulant matrix).
 This involves $O(n\log (n))$ flops overall
and then again allows efficient
 parallel implementation.

For a fixed scalar $f$ and $g=f^n$, we can define the matrix $Z_g({\bf v})$ by 
 any of the two vectors ${\bf u}$ or  ${\bf v}$. 
The matrix is unitary (up to scaling) if $|f|=1$ and if $|u_i|=1$ for all $i$
and is defined by $n+1$ real parameters 
(or by $n$ such parameters for a fixed $f$),
which we can fix or choose at random.

Now suppose that $n=2^ds$, $1\le d<k$, $d$ and $k$ are integers,
and substitute a pair of AF
matrices of recursion length $d$ for two factors $F_n$ in the
above expressions. 
Then the resulting {\em abridged $f$-circulant matrix} $Z_{g,d}({\bf v})$
{\em  of recursion depth} $d$ is still unitary  (up to scaling),
defined by $n+1$ or $n$ parameters $u_i$ and $f$, 
is sparse unless the positive integer $k-d$ is  small,
and can be pre-multiplied by a vector by using $(3d+3)n$ flops.
Instead of AF matrices, we can substitute  a pair of 
ASPF, APF, ASF, AH,
ASPH, APH, or ASF
matrices
for the factors
$F_n$. 
All such matrices form {\bf FAMILY (iii)} of
 $d$--{\em abridged $f$-circulant matrices}.

\begin{remark}\label{resrft}
Recall that $n\times l$ SRFT and SRHT matrices are the products  
 $\sqrt{n/l}~DF_nR$ and  $\sqrt{n/l}~DH_nR$, respectively,
 where  $H_n$ and $F_n$ 
are the matrices
of (\ref{eqrfd}) and (\ref{eqdft}), $D=\diag(u_i)_{i=0}^{n-1}$, 
$u_i$ are iid variables uniformly distributed on the circle
$\{u:~|u|=\sqrt{n/l}\}$, and $R$ is the  $n\times l$ 
submatrix  formed by $l$ columns of the identity matrix $I_n$
chosen uniformly at random. Equation (\ref{eqcrcl}) shows that
we can obtain a SRFT matrix by
pre-multiplying a circulant matrix by the matrix $F_n$ and 
post-multiplying it by the above matrix $R$.
\end{remark}

%------------------------------------------------------------------------------

\subsection{Inverses of bidiagonal matrices  
%modified pairs of Householder reflections
}\label{sinvchh}

%------------------------------------------------------------------------------

{\bf FAMILY (iv)}  is formed by the {\em inverses of $n\times n$ bidiagonal matrices}
$$\widehat B=(I_n+DZ)^{-1}~{\rm or}~\widehat B=(I_n+Z^TD)^{-1}$$ for 
%fixed or random
  %$Z=\big (\begin{smallmatrix}  
       %0  & ~0\\
        %I_{n-1}   & ~0 
    %\end{smallmatrix}\big )$ 
% the  $n\times n$ matrix filled with zeros,
%except for the $n-1$ entries of its first subdiagonal,
% filled with ones.
a matrix $D$  of   primitive  type 2 and the down-shift matrix
$Z$. In particular,

 $$\widehat B=(I_n+DZ)^{-1}=\begin{pmatrix}
        ~~1  ~ & ~~0&   \dots &~\dots   ~ ~& 0 & 0\\
       ~~ b_2b_3   &~~ 1    &  0 &   & ~ 0 & ~0\\
        ~~-b_2b_3b_4  & ~~b_3b_4  &  ~~ 1 &   \ddots & ~\vdots &~ \vdots  \\
         \vdots  &  \ddots  &   ~~~  \ddots &   \ddots     &~ \vdots & ~\vdots \\
		&	&	\ddots & \ddots  &	\\
	&	&	& & 1 &0	\\
         %(-1)^n
\pm b_2\cdots b_n  & \dots & \dots &-b_{n-2}b_{n-1}b_n & b_{n-1}b_n & ~1 
    \end{pmatrix}$$ if
$$I_n+DZ=\begin{pmatrix}
        1  ~ & 0&     ~  ~\dots   ~ ~& 0 & ~0\\
        -b_2   & 1    &  \ddots &    ~ 0 & ~0\\
         0 &   -b_3      & \ddots   & ~\vdots &~ \vdots  \\
         \vdots   &    &   \ddots    &~ 1 & ~0 \\
        0   & ~  \dots &      ~ ~~  \dots   ~~ ~ ~&-b_n & ~1 
    \end{pmatrix}.$$
In order to pre-multiply a matrix $\widehat B=(I_n+DZ)^{-1}$ by a vector ${\bf v}$,
however,
we do not compute its entries, but solve the linear system of equations
$(I_n+DZ){\bf x}={\bf v}$ by
using $2n-1$ flops or, in the real case, just $n-1$ additions and subtractions. 

We can randomize the matrix $\widehat B$ 
 by choosing up to $n-1$ random diagonal entries of
the matrix $D$
(whose leading entry  makes no impact on $\widehat B$).

Finally, $||\widehat B||\le \sqrt {n}$ because nonzero entries of the  lower triangular 
matrix $\widehat B=(I_n+DZ)^{-1}$ have absolute values 1,  and
clearly $||\widehat B^{-1}||=||I_n+DZ||\le \sqrt 2$. Hence
$\kappa(\widehat B)
\le \sqrt {2n}$ for $\widehat B=(I_n+DZ)^{-1}$,
and the same bound holds for $\widehat B=(I_n+Z^TD)^{-1}$.
 
%------------------------------------------------------------------------------

\subsection{Flops and random variables involved}\label{sflprnd}

%------------------------------------------------------------------------------

Table \ref{tabmlt}  shows 
upper bounds on 

(a) the numbers of random variables involved into the 
$n\times n$ matrices $\widehat B$
of the four families (i)--(iv)
and 

(b) the numbers of flops for pre-multiplication of such a matrix by
 a (dense) vector.\footnote{The asterisks in the table
% \ref{tabmlt} 
show that the matrices 
of families (i) AF, (i) ASPF, and (iii) involve nonreal roots of unity.}   \\
%These data include  $n$ random parameters defining a random  $n\times n$
%permutation  or diagonal matrix.
For comparison, 
$n^2$ random variables 
and $(2n-1)n$ flops are involved in the case of 
 an  $n\times n$  Gaussian multiplier and $2n$ variables and 
 order of $n\log(n)$ real or complex flops   are involved  in the case of 
 an   $n\times n$  SRHT or SRFT multiplier, respectively. 
%The estimates can change in the transition $\widehat B\rightarrow B$.
%\end{remark}

One can readily extend the estimates to $n\times l$ submatrices $B$ of the matrices
 $\widehat B$.

%------------------------------------------------------------------------------

\begin{table}[ht] 
  \caption{The numbers of random variables and flops}
\label{tabmlt}

  \begin{center}
    \begin{tabular}{|*{8}{c|}}
      \hline
family &  (i) AH & (i) ASPH & (i) AF &  (i) ASPF & (ii) &  (iii)& (iv)  
\\ \hline
random variables & 0  &$2n$ & 0  & $2n$ & $2q+1$  
& $n$ &   $n-1$  
\\\hline
flops complex  &  $dn$ & $(d+1)n$ & $1.5dn$ & $(1.5d+1)n$  & $(2q-1)n$     & $(3d+2)n$ &   $2n-1$
 
\\\hline
 flops in real case & $dn$  & $(d+1)n$  & * & *  & $qn$ &  *    
 &  $n-1$ 
\\\hline

    \end{tabular}
  \end{center}
\end{table}

\begin{remark}\label{resprs0}
Other observations besides the flop estimates  can be  decisive.
For example, a
special recursive structure 
 of  an ARSPH matrix $H_{2^{k},d}$ and 
an ARSPF matrix $F_{2^{k},d}$
allows
highly efficient  
 parallel implementation of  
their pre-multiplication by a vector based on 
Application Specific Integrated Circuits (ASICs) and 
Field-Programmable Gate Arrays (FPGAs), incorporating Butterfly
Circuits \cite{DE}.
\end{remark}

%------------------------------------------------------------------------------

\subsection{Other families of multipliers}\label{sbscfml}

%------------------------------------------------------------------------------

In this subsection we recall some other interesting  matrix families of candidate multipliers.

According to \cite[Remark 4.6]{HMT11}, ``among the structured random matrices ....
one of the strongest candidates involves sequences of random Givens rotations".
They are dense  unitary matrices
$$\frac{1}{\sqrt n}D_1G_1D_2G_2D_3F_n,$$
for the DFT matrix $F_n$, three random diagonal matrices
$D_1$, $D_2$ and $D_3$ of primitive type  2,
and two chains of Givens rotations  $G_1$ and $G_2$, 
each of the form
$$G(\theta_1,\dots,\theta_{n-1})=P\prod_{i=1}^{n-1}G(i,i+1,\theta_i)$$ for
a  random permutation matrix $P$,
$$G(i,i+1,\theta_i)=\diag(I_{i-1},\big (\begin{smallmatrix} c_i  & s_i\\
-s_i   & c_i\end{smallmatrix}\big ),I_{n-i-1}),~
c_i=\cos \theta_i,~
s_i=\sin \theta_i,~c_i^2+s_i^2=1.$$
Here
$\theta_1,\dots,\theta_{n-1}$ denote
$n-1$ random angles of rotation
 uniformly distributed in the range 
$0\le \phi< 2\pi$.

The DFT factor $F_n$ makes the resulting matrices dense, 
but we sparsify them
by  replacing that factor by an 
 AF, ASF, APF, or ASPF matrix having recursion depth $d<\log_2(n)$. This would also decrease 
the number of flops  involved in pre-multiplication of such a multiplier by a vector
from order $n\log_2(n)$ to $1.5dn+O(n)$. 
We turn Givens sequences into distinct candidate families of efficient multipliers 
by replacing 
either or both of the Givens products with sparse matrices of Householder reflections
 matrices  of the form
$I_n-\frac{2{\bf w}{\bf w}^*}{{\bf w}^*{\bf w}}$
for fixed or random sparse
vectors ${\bf w}$ (cf. \cite[Section 5.1]{GL13}).

We  obtain a variety of  multiplier families 
by combining matrices of basic families (i)--(iv) and 
 the above matrices. Besides linear  combinations, we
can apply block representation as in the following    
real  
$2\times 2$ block matrix $\frac{1}{\sqrt n}\begin{pmatrix}
Z_1({\bf u}) & Z_1({\bf v})  \\
Z_1({\bf v}) & -Z_1({\bf u}) 
\end{pmatrix}D$
for two  vectors ${\bf u}$ and ${\bf v}$
and a  matrix $D$ of primitive class 2.
%For another sample option, we can %intertwine 
%the Hadamard and Fourier recursive steps.

 The reader 
can find other useful families of multipliers
in our Section \ref{ststs}. For example, according to our tests
in  Section \ref{ststs}, it turned out to be  efficient
to use  nonsingular well-conditioned (rather than unitary)
diagonal factors in the definition of some of our basic matrix families.

%------------------------------------------------------------------------------

\subsection{Some heuristic recipes for multipliers}\label{smngflr} 

%\subsection{Recursive Application of %Proto-algorithm}\label{smlbsc}
 
Our study implies that in a sense Algorithm \ref{alg1}  outputs accurate LRA for most of the pairs of  inputs $M$ and well-conditioned full-rank multipliers $B$,
such as the sparse and structured multipliers of this section. In the unlikely case where the algorithm fails for such a pair $M$ and $B_1$, one can stay with the same input and successively try other multipliers $B_2$,  
$B_3$, and so on. If the algorithm  
fails with $s$ multipliers $B_1,\dots,B_s$, one can heuristically try  new multiplier of the form  $$B=\sum_{j=1}^sc_jB_j$$
 where $c_j=\pm 1$ for all $j$ and for a fixed or random choice of the signs $\pm$. Such a choice decreases the computational cost of the generation and application of a multiplier.
%(More generally, one can choose complex %values $c_j$
%on the unit circle, letting $|c_j|=1$ for %all $j$.)
This recipe 
has consistently worked in
our extensive tests for benchmark inputs in  Section \ref{ststs}, but we also 
succeeded when we applied the products  $B=B_1B_2$ of two candidate multipliers $B_1$ and $B_2$.
%\medskip
%------------------------------------------------------------------------------

\subsection{Generation of a Gaussian matrix from bidiagonal and permutation matrices}\label{sgpbd}
 
%------------------------------------------------------------------------------
      
In this subsection we prove that partial products of random  bidiagonal and permutation matrices converge to a Gaussian matrix and present the results of our numerical tests
that show fast convergence. 

We were motivated by the idea of
 recursive 
pre-processing with 
%VP 
these 
partial products until we obtain accurate CUR LRA, but the result may have independent value for the study of Gaussian matrices.

Let
$B_i$  be bidiagonal matrices 
with ones on the diagonal and $\pm 1$ 
on the first subdiagonal for random choice of the signs $\pm$, let
$P_i$ be random permutation matrices,
for  $i=1,2,\dots,q$, and let $q$ be the
  minimal integer for which our selected  superfast algorithm outputs accurate
  LRA of  the matrix 
  $WG_q:=W\prod_{i=1}^qB_iP_i$. 
  
  A matrix that approximates a Gaussian matrix must be  dense, and using it  as a multiplier
would be expensive, but 
based on our study in Sections \ref{sbsalgsp} we expect that our goal
can be  achieved  already for reasonably small integer $q$ for which pre-processing 
with the multiplier $G_q$ is still superfast.

In our tests we combined multiplication of twenty inverse-bidiagonal matrices with random column permutations. 
%Figure \ref{LowRkTest1} shows the %distribution of a single randomly chosen %entry and the scattered plot of two such %entries. 
The output $1024\times 1024$ matrices
turned out to be very close to Gaussian distribution and
have passed the Kolmogorov-Smirnov test for normality in all  tests repeated 1000 times (see Figure \ref{LowRkTest3}).

\begin{figure}[htb] 
\centering
\includegraphics[scale=0.7]{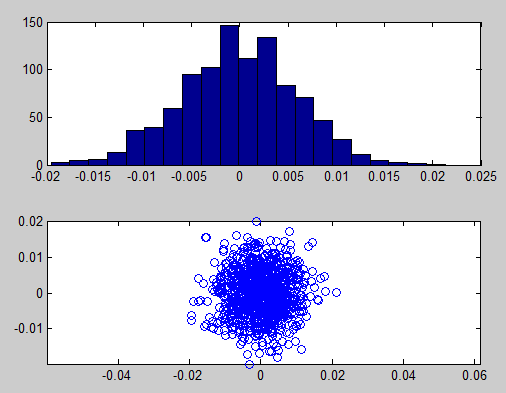}
\caption{Distribution of a randomly
chosen entry}
\label{LowRkTest3}
\end{figure}

In the rest of this section  we nontrivially prove convergence
of the partial products of such matrices to a Gaussian matrix.  

%------------------------------------------------------------------------------

Suppose that $n$ is a positive integer, $P$ is a random permutation matrix,
and define the $n\times n$ matrix 

\begin{equation}
\label{def11}
B: = \left[
\begin{array}{ccccc}
1 & & & & \pm 1\\
\pm 1& 1\\
	& \pm 1 & 1 \\
	&		& \cdots & \cdots\\
	&		&		 & \pm 1 & 1\\

\end{array} 
\right] P
\end{equation}
where each $\pm 1$ represents an independent Bernoulli random variable. 

\begin{theorem}
\label{thm1}
Let $B_0$, ..., $B_T$ be independent random matrices of the form \eqref{def11}. 
As $T\rightarrow \infty$, the  distributions of the matrices $G_T := \prod_{t=1}^{T}B_t$ 
converge to the  distribution of a Gaussian  matrix.
\end{theorem}

For demonstration of the approach first consider its simplified version. 
Let

\begin{equation}
\label{def2}
A := \left[
\begin{array}{ccccc}
1/2 & & & & 1/2\\
 1/2& 1/2\\
	&  1/2 & 1/2 \\
	&		& \cdots & \cdots\\
	&		&		 &  & 1/2\\

\end{array}
\right] P
\end{equation}
where $P$ is a random permutation matrix. We are going to prove the following theorem.

\begin{theorem}
\label{thm2}
Let $A_0$, ..., $A_T$ be independent random matrices defined by \eqref{def2}. 
As $T\rightarrow \infty$, the  distributions of the matrices
$\Pi_T := \prod_{t=1}^{T}A_t$ converge to the distribution
of the matrix $$\left[ 
\begin{array}{cccc} 
1/n &\cdots &\cdots  & 1/n\\
\vdots &\ddots &\ddots & \vdots \\ 
\vdots &\ddots &\ddots & \vdots \\
1/n &\cdots &\cdots  & 1/n\\
\end{array}
\right].$$
\end{theorem}

\begin{proof}
First examine the effect of multiplying by such  a random matrix $A_i$:
let $$M :=( {\bf m}_1~ |~{\bf m}_2~ |~ ... ~ |~{\bf m}_n)$$ and let
the permutation matrix of $A_i$ defines a column permutation 
$\sigma: \{1,...,n\}\rightarrow \{1,...,n\}$. Then 
%\begin{equation}
$$MA_i = \Big(\frac{{\bf m}_{\sigma(1)} + {\bf m}_{\sigma(2)}}{2}~\Big|~
\frac{{\bf m}_{\sigma(2)} + {\bf m}_{\sigma(3)}}{2}~\Big|~\dots~
\Big|~\frac{{\bf m}_{\sigma(n)} + {\bf m}_{\sigma(1)}}{2}\Big).
$$
%\end{equation}
Here each new column is written as a linear combination of $\{{\bf m}_1, ..., {\bf m}_n\}$. More generally, if we consider $MA_0 A_1 \cdots A_t$, i.e., $M$ multiplied with $t$ matrices of the form \eqref{def2}, each new column has the following linear expression:
%\begin{equation}
$$MA_0 A_1 \cdots A_t = \Big(\sum_k \pi^t_{k1}{\bf m}_k~\Big|~
\sum_k \pi^t_{k2}{\bf m}_k~\Big|~\dots~\Big|~\sum_k \pi^t_{kn}{\bf m}_k\Big).
$$
%\end{equation}
Here $\pi^t_{ij}$ is the coefficient of ${\bf m}_i$ in the linear expression of the $j$th column of the product matrix.
Represent the column permutation defined by matrix of $A_t$ by the map 
$$\sigma_t: \{1,...,n\}\rightarrow \{1,...,n\}$$ 
and readily verify the following lemma.
\begin{lemma}
\label{lem1} It holds that
$\sum_j \pi^t_{ij} = 1~{\rm for~all}~i$ 
and
%\begin{equation}
$$\pi^{t+1}_{ij} = \frac{1}{2}(\pi^{t}_{i\sigma_{t}(j)} + \pi^{t}_{i\sigma_{t}(j+1)})~{\rm for~all~pairs~of}
~i~{\rm and}~j
%\end{equation}
%\begin{equation}
.$$
%\end{equation}
\end{lemma}

Next we prove the following result.
\begin{lemma}
\label{lem2}
For any triple of $i,j$ and $\epsilon>0$,
%\begin{equation} 
$$\lim_{T\rightarrow\infty} {\rm Probability}\Big(\Big|\pi^T_{ij} - \frac{1}{n}\Big| > \epsilon\Big) = 0.$$
%\end{equation}
\end{lemma}

{\it Proof.}
%Firstly notice that this is a monotonically decreasing sequence, since
%\begin{align*}
%Pr(|\pi^{t+1}_{ij} - 1|>\epsilon)
%&< Pr(\frac{1}{2}|\pi^{t}_{i\sigma(j)} - 1|>\frac{1}{2}\epsilon and 
%\frac{1}{2}|\pi^{t}_{i\sigma(j+1)} - 1|>\frac{1}{2}\epsilon)\\
%&< Pr(\frac{1}{2}|\pi^{t}_{i\sigma(j)} - 1|>\frac{1}{2}\epsilon) +
%Pr(\frac{1}{2}|\pi^{t}_{i\sigma(j+1)} - 1|>\frac{1}{2}\epsilon)
%\end{align*}
Fix $i$, define 
%\begin{equation}
$$\mathcal{F}^t := \sum_{j} \Big(\pi^t_{ij} - \frac{1}{n}\Big)^2.$$
%\end{equation}
Then 
\begin{align*}
\mathcal{F}^{t+1} - \mathcal{F}^t
&= \sum_{j} \Big(\pi^{t+1}_{ij} - \frac{1}{n}\Big)^2 - \sum_{j} \Big(\pi^t_{ij} - \frac{1}{n}\Big)^2\\
&= \sum_{j} \Big[ \Big(\Big(\frac{\pi^t_{i\sigma(j)}+\pi^t_{i\sigma(j+1)}}{2}) - \frac{1}{n}\Big)^2 - 
\frac{1}{2}\Big(\pi^t_{i\sigma(j)}-\frac{1}{n}\Big)^2 - \frac{1}{2}\Big(\pi^t_{i\sigma(j+1)}-\frac{1}{n}\Big)^2 \Big]\\
&= \sum_j \Big[ \Big(\frac{\pi^t_{i\sigma(j)}+\pi^t_{i\sigma(j+1)}}{2}\Big)^2 - \frac{1}{2}(\pi^t_{i\sigma(j)})^2 - \frac{1}{2}(\pi^t_{i\sigma(j+1)})^2 \Big]\\
&= \sum_j -\frac{1}{4}\big[ (\pi^t_{i\sigma(j)})^2 - 2\pi^t_{i\sigma(j)}\pi^t_{i\sigma(j+1)} + (\pi^t_{i\sigma(j+1)})^2\big]\\
&= -\frac{1}{4}\sum_j \big( \pi^t_{i\sigma(j)} - \pi^t_{i\sigma(j+1)}\big)^2\\
&\leq -\frac{1}{4n}\big(\sum_i |\pi^t_{i\sigma(j)} - \pi^t_{i\sigma(j+1)}|\big)^2\\
&= -\frac{1}{4n}\big(\pi^t_{max} - \pi^t_{min} \big)^2.
\end{align*}
Here $\pi^t_{max}: = \max_j\{\pi^t_{ij}\}$ and $\pi^t_{min}: = \min_j\{\pi^t_{ij}\}$.

Furthermore, since $\pi^t_{max}\ge \pi^t_{ij}, \forall j$ and $\frac{1}{n}\ge \pi^t_{min}\ge 0$, it follows that
\begin{align*}
\mathcal{F}^t
= \sum_j (\pi^t_{ij} - \frac{1}{n})^2
\le n(\pi^t_{max} - \pi^t_{min})^2.
\end{align*}

Therefore
%\begin{equation}
%\aligned
$$\mathcal{F}^{t+1} - \mathcal{F}^t 
\le  -\frac{1}{4n}\big(\pi^t_{max} - \pi^t_{min} \big)^2
\le -\frac{1}{4n^2}\mathcal{F}^t.$$
%\endaligned
%\end{equation}

Now our monotone decreasing sequence has the only stationary value 
 when all values $\pi^t_{ij}$ coincide with each other. 
Together with Lemma \ref{lem2} this implies
%\begin{equation}
$$\lim_{T\rightarrow\infty}  {\rm  Probability}\Big(\Big|\pi^T_{ij} - \frac{1}{n}\Big | > \epsilon\Big) = 0.$$
%\end{equation}
\end{proof}

Next we prove Theorem \ref{thm1}.
\begin{proof}
Let $S^t_{i}$ denote the values $\pm 1$ on each row. By definition 
 $S^t_{i'}$ and $S^t_i$ are independent for $i\neq i'$. 
Moreover the following lemma can be  readily verified.

\begin{lemma}
$\prod_{t=1}^T S^t_{i_t}$ and $\prod_{t=1}^T S^t_{i'_t}$ are independent 
as long as there is at least one index $t$ such that $i_t\neq i'_t$.
\end{lemma}
 
Write
%\begin{equation}
$$MB_0 B_1 \cdots B_t = 
\Big(\sum_k \gamma^t_{k1}{\bf m}_k~\Big|~\sum_k \gamma^t_{k2}{\bf m}_k~
\Big|~.\dots~\Big|~\sum_k \gamma^t_{kn}{\bf m}_k\Big)$$
%\end{equation}
and notice that each $\gamma^T_{ij}$ can be written as a sum of random values $\pm 1$ whose signs 
are determined by $\prod_{t=1}^T S^t_{i_t}$. Since different signs are independent, 
we can represent $\gamma^T_{ij}$ as the  difference of two positive integers $\alpha - \beta$
 whose sum is $2^T\pi^T_{ij}$. 

Theorem \ref{thm2} implies that
the sequence $\pi^T_{ij}$ converges to $\frac{1}{n}$ almost surely as $T\rightarrow \infty$.

Therefore $\gamma^T_{ij}/2^T$ converges to Gaussian distribution as $T\rightarrow\infty$. 
Together with independence of the random values $\gamma^T_{ij}$ for all pairs $i$ and $j$,
this implies that eventually the entire matrix 
 converges to a Gaussian  matrix (with i.i.d. entries).
\end{proof}
 
 The speed of the convergence to Gaussian distribution is determined by the speed of
 the convergence (i) of the values  
 $\pi^T_{ij}$  to 1 as 
 $T\rightarrow\infty$
and (ii) of the binomial distribution with the mean $\pi^t_{ij}$ 
 to the Gaussian distribution. For (i), we have the  following estimate: 
\begin{align*}
|\pi^t_{ij} - 1| \le& \mathcal{F}^t
\le \Big(1-\frac{1}{4n^2}\Big)^{t-1}\mathcal{F}^0;
\end{align*}
and for (ii) we have the following Berry--Esseen theorem (cf. \cite{B41}).
\begin{theorem}
Let $X_1$, $X_2$, $\dots$ be independent random variables with $E(X_i)=0$, $E(X_i^2)=\sigma_i^2>0$ and $E(|X_i|^3)=\rho_i<\infty$ for all $i$. Furthermore  let
\begin{align*}
S_n: = \frac{X_1 + \cdots + X_n}{\sqrt{\sigma_1^2 + \sigma_2^2 +\cdots + \sigma_n^2}}
\end{align*}
be a normalized $n$-th partial sum. 
Let $F_n$ and  $\Phi$ denote the cumulative distribution functions of $S_n$
  a Gaussian variable, respectively. 
Then for some constant $c$ and for all $n$, 

\begin{align*}
\sup_{x\in\mathbb{R}}|F_n(x) - \Phi(x)| \le c 
%\psi
%\end{align*}
%~~{\rm where}~~ 
%\begin{align*}
%\psi: = 
\cdot (\sum_{i=1}^{n}\sigma_i^2)^{-3/2} \max_{1\le i\le n}\rho_i.
\end{align*}
\end{theorem}

In our case, for any fixed pair of $i$ and $j$, write $N_t = 2^t\pi_{ij}^t$ and $\gamma_{ij}^t  = X_1 + \cdots + X_{N_t}$ where $X_i$ are i.i.d. $\pm 1$ variables. Then 
$E(X_i)=0$, $E(X_i^2)=1$,  $E(X_i^3)=1$, and 
\begin{align*}
&S_{N_t}=\frac{X_1 + \cdots + X_{N_t}}{\sqrt{\sigma_1^2 + \sigma_2^2 +\cdots + \sigma_{N_t}^2}}
%\\
= \frac{\gamma_{ij}^t}{\sqrt{N_t}}
%\\
= \frac{\sqrt{\pi_{ij}^t}\gamma_{ij}^t}{2^{t/2}}.
\end{align*}
Furthermore
\begin{align*}
\sup_{x\in\mathbb{R}}|F_N(x) - \Phi(x)| \le&  c\cdot\Big(\sum_{i=1}^{N_t} \sigma_i^2 \Big)^{-3/2} \max_{1\le i\le N_t}\rho_i\\
\le& ~c\cdot N_t^{-3/2} ~\rightarrow 0 ~{\rm as}~t\rightarrow \infty.
\end{align*}

%------------------------------------------------------------------------------
%------------------------------------------------------------------------------
 
%\subsection{Randomized factorization of Gaussian  matrices}%\label{sfig}
   
%------------------------------------------------------------------------------

\section{An Application of LRA -- Superfast Multipole Method}\label{sextpr} 
 
%------------------------------------------------------------------------------
%------------------------------------------------------------------------------

Our superfast algorithms can be extended to numerous important computational problems linked to  LSR and LRA.
Next we specify a sample
acceleration of the Fast Multipole Method (FMM) to superfast level.
  
%------------------------------------------------------------------------------

%\subsection{Superfast Multipole Method}%\label{ssfmm}

%------------------------------------------------------------------------------
%Namely 

FMM has been devised for  superfast multiplication by a vector of a 
special structured matrix, 
called  HSS matrix,
provided that low rank generators are available for its off-diagonal blocks.
 Such generators are not available in some important applications, however (see, e.g, \cite{XXG12}, 
 \cite{XXCB14}, and \cite{P15}),
 and then their computation 
 by means of the known algorithms
 is not superfast. According to our study  C--A and some other algorithms  perform
 this stage superfast on the average input, thus turning FMM into {\em Superfast 
 Multipole Method}.
 
 Since the method is highly important
we supply some details of its basic and bottleneck stage of HSS computations, which we perform superfast as soon as we incorporate superfast LRA at that stage.

 HSS matrices
 naturally extend the class of banded matrices and their inverses, 
are closely linked to FMM,  
 and 
are increasingly popular
(see \cite{BGH03},  \cite{GH03}, \cite{MRT05},
 \cite{CGS07},
 \cite{VVGM05}, 
\cite{VVM07/08}, 
 \cite{B10}, \cite{X12},   \cite{XXG12},
\cite{EGH13}, \cite{X13},
 \cite{XXCB14},
  and the bibliography therein). 

\begin{definition}\label{defneut} {\rm (Neutered Block Columns. See 
\cite{MRT05}.)} With each diagonal block of a block matrix 
associate 
its complement in its block column,
and call this complement a {\em neutered block column}.
\end{definition}

\begin{definition}\label{defqs} {\rm (HSS matrices. See 
% \cite{GR87}, \cite{CGR88}, \cite{T00}, 
% \cite{BGH03},  \cite{GH03}, \cite{B10},
\cite{CGS07},
 \cite{X12},  \cite{X13}, \cite{XXCB14}.)}
  
A block 
matrix $M$ of size  $m\times n$ is 
called a $r$-{\em HSS matrix}, for a positive integer $r$, 

(i) if all diagonal blocks of this matrix
consist of $O((m+n)r)$ entries overall
and
 
(ii) if $r$ is the maximal rank of its neutered block columns. 
\end{definition}

\begin{remark}\label{reqs}
Many authors work with $(l,u)$-HSS
(rather than $r$-HSS) matrices $M$ for which $l$ and $u$
are the maximal ranks of the sub- and super-diagonal blocks,
respectively.
The $(l,u)$-HSS and $r$-HSS matrices are closely related. 
If a neutered block column $N$
is the union of a sub-diagonal block $B_-$ and 
a super-diagonal block $B_+$,
then
 $\rank (N)\le \rank (B_-)+\rank (B_+)$,
 and so
an $(l,u)$-HSS matrix is a $r$-HSS matrix,
for $r\le l+u$,
while clearly a $r$-HSS matrix is  
a $(r,r)$-HSS matrix.
\end{remark}

The FMM exploits the $r$-HSS structure of a matrix as follows:

(i) Cover all off-block-diagonal entries
with a set of  non-overlapping neutered block columns.  

(ii) Express every neutered block column $N$ of this set
  as the product
  $FH$ of two  {\em generator
matrices}, $F$ of size $h\times r$
and $H$ of size $r\times k$. Call the 
pair $\{F,H\}$ a {\em length $r$ generator} of the 
neutered block column $N$. 

(iii)  Multiply 
the matrix $M$ by a vector by separately multiplying generators
and diagonal blocks by subvectors,  involving $O((m+n)r)$ flops
overall, and

(iv) in a more advanced application of  
 FMM solve a nonsingular $r$-HSS linear system of $n$
equations  by using
$O(nr\log^2(n))$ flops under some mild additional assumptions on  the input. 

This approach is readily extended to the same operations with  
$(r,\xi)$-{\em HSS matrices},
that is, matrices approximated by $r$-HSS matrices
within a perturbation norm bound $\xi$ where a  positive tolerance 
$\xi$ is small in context (for example, is the unit round-off).
 Likewise, one defines 
an  $(r,\xi)$-{\em HSS representation} and 
$(r,\xi)$-{\em generators}.

$(r,\xi)$-HSS matrices (for $r$ small in context)
appear routinely in matrix computations,
and computations with such matrices are 
performed  efficiently by using the
above techniques.

%------------------------------------------------------------------------------

In some applications of the FMM (see  \cite{BGP05}, \cite{VVVF10})
stage (ii) is omitted because short generators for all 
neutered block columns are readily available,
 but this is not the case in a variety of other  important applications
 (see \cite{XXG12}, \cite{XXCB14}, and \cite{P15}). 
This stage of the computation of  generators is precisely 
 LRA of the neutered block
columns, which turns out to be 
the bottleneck stage of FMM in these applications, and superfast LRA algorithms  
provide a remedy. 

Indeed apply a fast  algorithm at this
 stage, e.g., the algorithm of \cite{HMT11}
 with a Gaussian multiplier.
Multiplication of a $q\times h$ matrix
by an $h\times r$ Gaussian matrix requires $(2h-1)qr$ flops,
while
standard HSS-representation of an $n\times n$ 
HSS matrix includes $q\times h$ neutered 
 block columns for $q\approx m/2$ and $h\approx n/2$. In this case 
the cost of computing an $r$-HSS representation of the
matrix $M$ is at least of order $mnr$.
For $r\ll \min\{m,n\}$, this
is much  greater than 
 $O((m+n)r\log^2(n))$ flops, used  at the other stages of 
the computations. 

Can we alleviate such a problem?  
Yes, heuristically we can 
compute LRA 
to  
$(r,\xi)$-generators
superfast
by applying  superfast LRA algorithms of this paper, which are accurate   
for the average input
and whp for
a perturbed factor-Gaussian input.
    
%------------------------------------------------------------------------------

\section{Numerical Tests of Superfast Least Squares Regression (LSR)}\label{ststs}

% - - - - - - - - - - - - - - - - - - - - - - - - - - - - - - - - - - - - -
%-----------------------------------------------------------------------------
% edited by Qi 5/27/2018
%\subsection{Tests for the Least Squares Regression}%\label{testlstsq}

In this section, we present the results of our tests  of Algorithm \ref{algapprls}
for the Least Squares Regression (LSR).
We worked with
 random well-conditioned
multipliers and computed the relative residual norms
$$\displaystyle \frac{\min_{\bf x} || LA{\bf x} - L{\bf b}||}{\min_{\bf x} ||A{\bf x}-{\bf b}||}$$ for evaluation.  In our tests they  approximated one from above quite closely.
 
We performed the tests on a machine with Intel Core i7 processor running Windows 7 64bit; 
we invoked the \textit{lstsq} function from Numpy 1.14.3 for solving the LSR problems.

 We generated the test matrices $A\in \mathbb{R}^{m\times n}$   by following
 (with a few modifications) the  recipes of extensive tests in \cite{AMT10},  which compared the running time 
of the regular LSR problems and the reduced ones with WHT, DCT, and DHT pre-processing. 
%VP

We used test
matrices $A$  of the 
following types: Gaussian matrices, ill-conditioned random matrices, semi-coherent matrices, and coherent matrices.
%VP
 
Table \ref{lsrRandom} displays the test results for  Gaussian input matrices. 
%VP

Table \ref{lsrill} displays the results for  ill-conditioned random inputs 
 generated through SVD $A = S\Sigma T^*$, where we 
generated the orthogonal matrices $S$ and $T$  of singular vectors 
as the Q factors in the QR-factorization of independent Gaussian matrices
 and 
chose $\Sigma$ to have leading singular values $10^t$ with $t = 4, 3, 2\dots , -9,$ and the rest $10^{-10}$. 
%VP

Table \ref{lsrsemi} displays the test results for semi-coherent input matrices
$$A_{m\times n} = 
\begin{bmatrix}
%VP G_{m-n/2 \times n/2} & \\
G_{(m-n/2) \times n/2} & \\
	& D_{n/2} \\
\end{bmatrix}
$$
where 
%VP $G_{m-n/2 \times n/2}$ 
$G_{(m-n/2) \times n/2}$ 
is a random Gaussian matrix and $D_{n/2}$ is a diagonal matrix with diagonal
entries chosen independently uniformly from $\pm 1$. 
%VP

Table \ref{lsrcoh} displays the test results for coherent input matrices 
$$A_{m\times n} = 
\begin{bmatrix}
D_{n}\\
0
\end{bmatrix}$$
 where $D_{n}$ is a random diagonal matrix defined 
 %VP
in the same way as above. 
 %VP
  
The coherence of a matrix $A_{m\times n} = S\Sigma T^T$ is defined as the maximum squared row norm of its left singular matrix, with 1 
being its
%VP the 
maximum and $n/m$ being its
%VP the 
minimum. If the test input has coherence 1, 
%VP
then in order to have an accurate result 
the multiplier must "sample" the corresponding rows with maximum row norm in the left singular matrix.

%VP

The semi-coherent and coherent inputs have coherence $1$ and are 
%VP considered 
the harder cases.

We applied random multipliers $L\in \mathbb{R}^{k\times m}$, 
%VP
for $k \le m$,  from 
the families introduced in
%VP Section 
Appendix \ref{ssprsml}, namely random 
circulant matrices,  inverses of bidiagonal matrices, 
%VP and 
random matrices of Householder reflections,
and ASPH matrices
%VP
with recursion depth 3, 6, and
%VP
8, depending on the inputs. 
 For comparison we also included the
 test results with Gaussian multipliers.

We defined our random circulant matrices 
by filling their first columns with  random elements $\pm 1$, with each sign 
$+$ or $-$ chosen with probability 1/2. 

We generated  
 bidiagonal matrices having 1 on the diagonal and  random  
entries $\pm 1$ on the sub- or  super-diagonal;  then we used pairwise sums of the inverses of the sub- and super-bidiagonal 
matrices as our multipliers.
These multipliers have substantially larger condition numbers than our other multipliers, and we consistently arrived   
 at larger output errors when we used them -- in good accordance with Remark \ref{renonorth}. 

%VP . The bidiagonal multiplier used here is 
% different from those of
%VP ones introduced in 
%Section \ref{sfig}.

We generated an $m\times m$ random Householder Reflection matrix
%VP Matrix 
$R = \prod_{i=1}^{k/2} P_iR_i$ for random permutation matrices $P_i$ and for
$R_i = I_n - \frac{2{\bf w}_i{\bf w}_i^T}{{\bf w}_i^T{\bf w}_i}$ where 
 ${\bf w}_i$ were Gaussian vectors. 
 
%VP  We have chosen 
Our input matrices $A$ are  highly over-determined, having many more rows than columns.  
%VP 
 
We have empirically chosen $k = 6n$ for 
the multipliers $L$.   We tried to choose the ratio $k/n$ smaller in order to
 accelerate the solution 
but had
%VP
to keep it large enough in order to be able to
obtain accurate solution. 

%VP For , 
In order to 
%VP handle 
decrease the probability of failure
in the cases of 
semi-coherent and coherent inputs, we simply generated three independent multipliers 
%VP for each 
per input
%VP 
and chose the best performing one  by comparing the ratios of residual norms.

We performed 100 tests for every triple of the input class, multiplier class, and test sizes, and computed
the  mean and standard deviation of the relative residual norm.

The test results displayed 
%VP  from 
in Tables \ref{lsrRandom}--\ref{lsrcoh} show that our multipliers were consistently
effective for random matrices.   The performance was not affected 
by the conditioning of the input matrices.
The 
%VP results 
outputs were 
%VP a little
 less accurate where we used  pairwise sums  of the inverses of bidiagonal matrices
for multipliers, and 
%VP for 
in the cases of semi-coherent and coherent inputs with random Householder Reflection matrices.

\begin{table}
\begin{center}
\begin{tabular}{l  c c c}
\hline
Multiplier & Matrix Sizes $(k, m, n)$ & Mean & STD \\
\hline
Gaussian & (600, 4096, 100) & 1.098 & 1.53E-2 \\
Gaussian & (2400, 16384, 400) & 1.095 & 7.20E-3 \\
3-ASPH & (600, 4096, 100) & 1.084 & 1.13E-2 \\
3-ASPH & (2400, 16384, 400) & 1.084 & 6.28E-3\\
Bidiagonal & (600, 4096, 100) & 1.460 & 7.87E-2 \\
Bidiagonal & (2400, 16384, 400) & 1.479 & 4.53E-2 \\
Circulant & (600, 4096, 100) & 1.096 & 1.42E-2\\
Circulant& (2400, 16384, 400) & 1.095 & 7.38E-3\\
Householder& (600, 4096, 100) & 1.085 & 1.21E-2\\
Householder& (2400, 16384, 400) & 1.084 & 6.92E-3\\
\hline
\end{tabular}
\caption{Relative residual norms in tests with Gaussian input matrices}\label{lsrRandom}
\end{center}   
\end{table}

\begin{table}
\begin{center}
\begin{tabular}{l  c c c}
\hline
Multiplier & Matrix Sizes $(k, m, n)$ & Mean & STD \\
\hline
Gaussian & (600, 4096, 100) & 1.096 & 1.48E-2 \\
Gaussian & (2400, 16384, 400) & 1.095 & 6.96E-3 \\
3-ASPH & (600, 4096, 100) & 1.082 & 1.20E-2 \\
3-ASPH & (2400, 16384, 400) & 1.082 & 5.81E-3\\
Bidiagonal & (600, 4096, 100) & 1.469 & 7.43E-2 \\
Bidiagonal & (2400, 16384, 400) & 1.471 & 3.80E-2 \\
Circulant & (600, 4096, 100) & 1.092 & 1.24E-2\\
Circulant& (2400, 16384, 400) & 1.094 & 7.11E-3\\
Householder& (600, 4096, 100) & 1.083 & 1.08E-2\\
Householder& (2400, 16384, 400) & 1.084 & 6.71E-3\\
\hline
\end{tabular}
\caption{Relative residual norms in the tests with ill-conditioned random inputs}\label{lsrill}
\end{center}   
\end{table}

\begin{table}
\begin{center}
\begin{tabular}{l  c c c}
\hline
Multiplier & Matrix Sizes $(k, m, n)$ & Mean & STD \\
\hline
Gaussian & (600, 4096, 100) & 1.084 & 9.70E-3 \\
Gaussian & (2400, 16384, 400) & 1.090 & 4.55E-3 \\
6-ASPH & (600, 4096, 100) & 1.074 & 8.47E-3 \\
8-ASPH & (2400, 16384, 400) & 1.079 & 4.83E-3\\
Circulant & (600, 4096, 100) & 1.045 & 1.81E-2\\
Circulant& (2400, 16384, 400) & 1.043 & 1.34E-2\\
\hline
\end{tabular}
\caption{Relative residual norms in the tests with semi-coherent inputs}\label{lsrsemi}
\end{center}
\end{table}

\begin{table}
\begin{center}
\begin{tabular}{l  c c c}
\hline
Multiplier & Matrix Sizes $(k, m, n)$ & Mean & STD \\
\hline
Gaussian & (600, 4096, 100) & 1.083 & 9.15E-3 \\
Gaussian & (2400, 16384, 400) & 1.090 & 5.24E-3 \\
6-ASPH & (600, 4096, 100) & 1.078 & 9.49E-2 \\
8-ASPH & (2400, 16384, 400) & 1.081 & 4.97E-3\\
Circulant & (600, 4096, 100) & 1.044 & 2.03E-2\\
Circulant& (2400, 16384, 400) & 1.043 & 1.48E-2\\
\hline
\end{tabular}
\caption{Relative residual norms in the tests with coherent inputs}\label{lsrcoh}
\end{center}   
\end{table}

% end of edit by Qi 7/9/2018

%----------------------------------------------------------------------------
 
\section{Tests for LRA by Means of Random Sampling}\label{srndsmpl}

  Liang Zhao has performed the tests for Tables \ref{tab1}--\ref{ExpHMT2}    by using MATLAB  
 in the Graduate Center of the City University of New York 
on a Dell computer with the Intel Core 2 2.50 GHz processor and 4G memory running 
Windows 7; 
in particular the standard normal distribution function randn of MATLAB
has been applied in order to generate Gaussian matrices.

He has calculated the $\xi$-rank, i.e., the number of singular values 
exceeding $\xi$, by applying the MATLAB function "svd()". 
He has set $\xi=10^{-5}$ in Sections \ref{ststssvd} and
 \ref{ststslo} and  $\xi=10^{-6}$
in Section \ref{s17m}. 

John Svadlenka has performed the tests for Tables \ref{SuperfastTable}--\ref{tb:lscore}  on a 64-bit Windows machine with an Intel i5 dual-core 1.70 GHz processor using custom programmed software in $C^{++}$ and compiled with LAPACK version 3.6.0 libraries.
   
%----------------------------------------------------------------------------
 
\subsection{Tests for LRA of inputs generated via SVD}\label{ststssvd}

% - - - - - - - - - - - - - - - - - - - - - - - - - - - - - - - - - - - - -

In the tests of this subsection we generated $n\times n$ input matrices $M$  
 by extending the customary recipes of [H02, Section 28.3]. Namely, we first  
generated matrices $S_M$ and $T_M$ 
by means of the orthogonalization of  
$n\times n$ Gaussian matrices. Then we defined 
$n\times n$ matrices $M$ by 
their compact SVDs, $M=S_M\Sigma_M T_M^*$,
for $\Sigma_M=\diag(\sigma_j)_{j=1}^n$; 
 $\sigma_j=1/j,~j=1,\dots,r$,
$\sigma_j=10^{-10},~j=r+1,\dots,n$, 
 and  $n=256,
512,
1024$.
(Hence $||M||=1$ and 
$\kappa(M)=10^{10}$.) 

Table \ref{tab1} shows 
the average output error norms 
over  1000 tests of  Algorithm \ref{alg1}a applied to 
these matrices $M$ for 
each pair of $n$ and $r$, 
%with
  $n=256,
512,
1024$, $r=8,32$, and
each of the following three groups of multipliers:
%\begin{enumerate}
%\item%1 
3-AH multipliers, 
%\item%2
3-ASPH  multipliers, both
 defined by 
 Hadamard recursion (\ref{eqfd}),  
 for $d=3$, and 
%\item%3
dense  multipliers $B=B(\pm 1,0)$  
having iid entries $\pm 1$ and 0,
each value chosen with probability 1/3.
%\end{enumerate}  

%------------------------------------------------------------------------------

\begin{table}[ht] 
  \caption{Error norms for SVD-generated inputs 
and 3-AH, 3-ASPH,  and $B(\pm 1,0)$ multipliers}
\label{tab1}

  \begin{center}
    \begin{tabular}{|*{8}{c|}}
      \hline
$n$ & $r$  &3-AH&3-ASPH& $B(\pm 1,0)$ %&4-BRI 
\\ \hline
256 & 8 & 2.25e-08 & 2.70e-08 & 2.52e-08 
%& 2.49e-08  
\\\hline
256 & 32 &5.95e-08 & 1.47e-07 & 3.19e-08 
%& 4.38e-08 
\\\hline
512 & 8 &4.80e-08 & 2.22e-07 & 4.76e-08 
%& 1.71e-07 
\\\hline
512 & 32 & 6.22e-08 & 8.91e-08 & 6.39e-08 
%& 6.10e-08
\\\hline
1024 & 8 & 5.65e-08 & 2.86e-08 & 1.25e-08
%& 1.45e-07
\\\hline
1024 & 32 & 1.94e-07 & 5.33e-08 & 4.72e-08
% & 1.15e-08
\\\hline
    \end{tabular}
  \end{center}
\end{table}

% - - - - - - - - - - - - - - - - - - - - - - - - - - - - - - - - - - - - -

%------------------------------------------------------------------------------
 
%\clearpage

Table \ref{LowRkEx} displays  the
average error norms in
the case of multipliers $B$
of eight kinds defined below, all  
generated from the following  Basic Sets 1, 2 and 3
of $n\times n$ multipliers:

{\em Basic Set 1}:  3-APF
multipliers defined by 
three Fourier recursive steps of
 equation (\ref{eqfd}), for $d=3$,
with no scaling, but  
with a random column permutation.

{\em Basic Set 2}: Sparse real circulant matrices $Z_1({}\bf v)$
of family (ii) of  Section \ref{scrcsp} (for $q=10$) 
 having the first column vectors  ${\bf v}$ filled with zeros,
except for  ten random coordinates filled with random integers $\pm 1$.

{\em Basic Set 3}:   Sum of two scaled inverse bidiagonal matrices. 
We first filled the main diagonals of both matrices with the integer 101
 and  their first subdiagonals 
 with $\pm 1$. Then
we  multiplied  each matrix by a  diagonal 
matrix $\diag(\pm 2^{b_i})$, where $b_i$ were random integers
uniformly chosen from 0 to 3.

For multipliers $B$ we used the $n\times r$  western 
(leftmost) blocks of $n\times n$ matrices
of the following classes:
\begin{enumerate}
\item%1
  a matrix from Basic Set 1; 
\item%2
  a matrix from Basic Set 2;
\item%3
 a matrix from Basic Set 3;
\item%4
 the product of two matrices of Basic Set 1;
\item%5
 the product of two matrices of Basic Set 2;
\item%6
 the product of two matrices of Basic Set 3;
\item%7
 the sum of two matrices of Basic Sets 1 and 3,
and 
\item%8
 the sum of two matrices of Basic Sets 2 and 3.
\end{enumerate}
The tests
 produced the results similar to the ones of Table \ref{tab1}.
 
In sum, for all classes of input  pairs $M$ and $B$ and all pairs of integers $n$ and $r$,
Algorithm \ref{alg1}a with our 
pre-processing 
has consistently output approximations to rank-$r$ input matrices with  the
average error norms 
 ranged from $10^{-7}$ or $10^{-8}$ to about $10^{-9}$
in all our tests.

\begin{table}[ht] 
  \caption{Error norms  for SVD-generated inputs  and
multipliers of eight classes}
\label{LowRkEx}
  \begin{center}
    \begin{tabular}{| c |  c | c |  c |c|c|c|c|c|c|}
      \hline
$n$ & $r$ &class 1 &class 2 &class 3 &class 4 &class 5 &class 6 &class 7 &class 8 \\\hline
256 & 8 &5.94e-09 &4.35e-08 &2.64e-08 &2.20e-08 &7.73e-07 &5.15e-09 &4.08e-09 &2.10e-09  \\\hline
256 & 32 &2.40e-08 &2.55e-09 &8.23e-08 &1.58e-08 &4.58e-09 &1.36e-08 &2.26e-09 &8.83e-09 \\\hline 
512 & 8 &1.11e-08 &8.01e-09 &2.36e-09 &7.48e-09 &1.53e-08 &8.15e-09 &1.39e-08 &3.86e-09  \\\hline
512 & 32 &1.61e-08 &4.81e-09 &1.61e-08 &2.83e-09 &2.35e-08 &3.48e-08 &2.25e-08 &1.67e-08\\\hline 
1024 & 8 &5.40e-09 &3.44e-09 &6.82e-08 &4.39e-08 &1.20e-08 &4.44e-09 &2.68e-09 &4.30e-09 \\\hline 
1024 & 32 &2.18e-08 &2.03e-08 &8.72e-08 &2.77e-08 &3.15e-08 &7.99e-09 &9.64e-09 &1.49e-08\\\hline 
    \end{tabular}
  \end{center}
\end{table}

%------------------------------------------------------------------------------
%------------------------------------------------------------------------------

  We summarize the  results of the tests of this subsection for $n=1024$ and $r=8,32$
in Figure \ref{LowRkTest1}. 

\begin{figure}[htb] 
\centering
\includegraphics[scale=0.7] {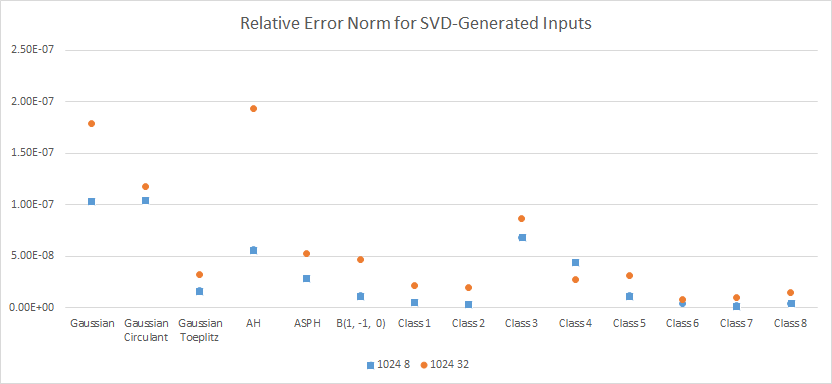}
\caption{Error norms in the tests of Section \ref{ststssvd}}
\label{LowRkTest1}
\end{figure}

%In the tests  average error norm ranged 
%from $10^{-7}$ to $10^{-9}$.
 
%\clearpage

%In the tests  average error norm ranged 
%from $10^{-7}$ to $10^{-9}$.
 
%\clearpage

% - - - - - - - - - - - - - - - - - - - - - - - - - - - - - - - - - - - - -

\subsection{Tests for LRA of inputs generated via the discretization of a Laplacian operator
and via the approximation of an inverse finite-difference operator}\label{ststslo} 

% - - - - - - - - - - - - - - - - - - - - - - - - - - - - - - - - - - - - -

Next we present the test results for Algorithm \ref{alg1}a applied
 to input matrices for computational problems of two kinds,
both taken from  \cite{HMT11}, namely, the matrices of
 
(i) the discretized single-layer Laplacian operator and 

(ii) the approximation of the inverse of a finite-difference operator.

{\em Input matrices (i).} We considered the Laplacian operator
%\begin{equation}
$[S\sigma](x) = c\int_{\Gamma_1}\log{|x-y|}\sigma(y)dy,x\in\Gamma_2$,
%\end{equation}
from  \cite[Section 7.1]{HMT11},
for two contours $\Gamma_1 = C(0,1)$ and $\Gamma_2 = C(0,2)$  on the complex plane.
Its dscretization defines an $n\times n$ matrix $M=(m_{ij})_{i,j=1}^n$  
where
%\begin{equation}
$m_{i,j} = c\int_{\Gamma_{1,j}}\log|2\omega^i-y|dy$
for a constant $c$ such that  $||M||=1$ and
%\end{equation} 
for the arc $\Gamma_{1,j}$  of the contour $\Gamma_1$ defined by
the angles in the range $[\frac{2j\pi}{n},\frac{2(j+1)\pi}{n}]$.

We applied Algorithm \ref{alg1}a
% supported by three iterations of the Power Scheme of \cite[Section 9.3]{HMT11} and used 
with multipliers 
 $B$ being the $n\times r$ leftmost submatrices of $n\times n$
matrices of 
the following five  classes: 
\begin{itemize} 
\item%1
 Gaussian  multipliers 
\item%2
 Gaussian Toeplitz  multipliers $T=(t_{i-j})_{i=0}^{n-1}$ 
for iid Gaussian variables $t_{1-n},\dots,t_{-1}$,$t_0,t_1,\dots,t_{n-1}$
  \item%3
 Gaussian circulant  multipliers $\sum_{i=0}^{n-1}v_iZ_1^i$, 
for iid Gaussian variables $v_0,\dots,v_{n-1}$ and the unit circular matrix $Z_1$
of Section \ref{sdfcnd}  
\item%4
 Abridged permuted  Fourier (3-APF) multipliers  
\item%5
 Abridged permuted Hadamard (3-APH) multipliers.
\end{itemize}

As in the previous subsection,
we defined each 
%Abridged Permuted Fourier or Hadamard
 3-APF and 3-APH  matrix by applying
three recursive steps of equation (\ref{eqfd}) followed
by a single random column permutation.

%We have increased the output accuracy by 
%applying the Power Scheme of Remark \ref{rernlrpr}.
 
We applied Algorithm \ref{alg1}a with  multipliers of all five listed classes.
For each setting we repeated the test 1000 times and calculated the mean and standard deviation of the error norm $||UV - M||$. 
%Table \ref{ExpHMT1} displays test results for $n$ up to 4000.
  
\begin{table}[ht]
\caption{LRA  of Laplacian  matrices}
\label{ExpHMT1}
\begin{center}
\begin{tabular}{|*{5}{c|}}
\hline
$n$ 	& multiplier 	& $r$	& mean	& std\\\hline
200 & Gaussian &  3.00 & 1.58e-05 & 1.24e-05\\\hline
200 & Toeplitz &  3.00 & 1.83e-05 & 7.05e-06\\\hline
200 & Circulant &  3.00 & 3.14e-05 & 2.30e-05\\\hline
200 & 3-APF &  3.00 & 8.50e-06 & 5.15e-15\\\hline
200 & 3-APH &  3.00 & 2.18e-05 & 6.48e-14\\\hline
400 & Gaussian &  3.00 & 1.53e-05 & 1.37e-06\\\hline
400 & Toeplitz &  3.00 & 1.82e-05 & 1.59e-05\\\hline
400 & Circulant &  3.00 & 4.37e-05 & 3.94e-05\\\hline
400 & 3-APF &  3.00 & 8.33e-06 & 1.02e-14\\\hline
400 & 3-APH &  3.00 & 2.18e-05 & 9.08e-14\\\hline
2000 & Gaussian &  3.00 & 2.10e-05 & 2.28e-05\\\hline
2000 & Toeplitz &  3.00 & 2.02e-05 & 1.42e-05\\\hline
2000 & Circulant &  3.00 & 6.23e-05 & 7.62e-05\\\hline
2000 & 3-APF &  3.00 & 1.31e-05 & 6.16e-14\\\hline
2000 & 3-APH &  3.00 & 2.11e-05 & 4.49e-12\\\hline
4000 & Gaussian &  3.00 & 2.18e-05 & 3.17e-05\\\hline
4000 & Toeplitz &  3.00 & 2.52e-05 & 3.64e-05\\\hline
4000 & Circulant &  3.00 & 8.98e-05 & 8.27e-05\\\hline
4000 & 3-APF &  3.00 & 5.69e-05 & 1.28e-13\\\hline
4000 & 3-APH &  3.00 & 3.17e-05 & 8.64e-12\\\hline
\end{tabular}
\end{center}
\end{table}

{\em Input matrices (ii).} We similarly applied Algorithm \ref{alg1}a to the input matrix $M$ 
being the inverse of a large sparse matrix obtained from a finite-difference operator
of  \cite[Section 7.2]{HMT11} 
and observed  similar results
with all structured  and Gaussian multipliers.

We performed 1000 tests for every class of pairs of $n\times n$ or $m\times n$ matrices 
of classes (i) or (ii), respectively,
and 
$n\times r$ multipliers for every fixed triple of $m$, $n$, and $r$ or pair of $n$ and $r$.

Tables \ref{ExpHMT1} and \ref{ExpHMT2} display the resulting data for the mean values and standard deviation of the error norms, and we summarize the results of the tests of this subsection 
in Figure \ref{LowRkTest2}.

\begin{table}[ht]
\caption{ LRA of the matrices of discretized finite-difference operator}
\label{ExpHMT2}
\begin{center}
\begin{tabular}{|*{6}{c|}}
\hline
$m$ 	& $n$ 	& multiplier 	& $r$	& mean	& std\\\hline
88 & 160 & Gaussian &  5.00 & 1.53e-05 & 1.03e-05\\\hline
88 & 160 & Toeplitz &  5.00 & 1.37e-05 & 1.17e-05\\\hline
88 & 160 & Circulant &  5.00 & 2.79e-05 & 2.33e-05\\\hline
88 & 160 & 3-APF &  5.00 & 4.84e-04 & 2.94e-14\\\hline
88 & 160 & 3-APH &  5.00 & 4.84e-04 & 5.76e-14\\\hline
208 & 400 & Gaussian & 43.00 & 4.02e-05 & 1.05e-05\\\hline
208 & 400 & Toeplitz & 43.00 & 8.19e-05 & 1.63e-05\\\hline
208 & 400 & Circulant & 43.00 & 8.72e-05 & 2.09e-05\\\hline
208 & 400 & 3-APF & 43.00 & 1.24e-04 & 2.40e-13\\\hline
208 & 400 & 3-APH & 43.00 & 1.29e-04 & 4.62e-13\\\hline
408 & 800 & Gaussian & 64.00 & 6.09e-05 & 1.75e-05\\\hline
408 & 800 & Toeplitz & 64.00 & 1.07e-04 & 2.67e-05\\\hline
408 & 800 & Circulant & 64.00 & 1.04e-04 & 2.67e-05\\\hline
408 & 800 & 3-APF & 64.00 & 1.84e-04 & 6.42e-12\\\hline
408 & 800 & 3-APH & 64.00 & 1.38e-04 & 8.65e-12\\\hline
\end{tabular}
\end{center}
\end{table} 

%------------------------------------------------------------------------------

\begin{figure}[htb] 
\centering
\includegraphics[scale=0.7] {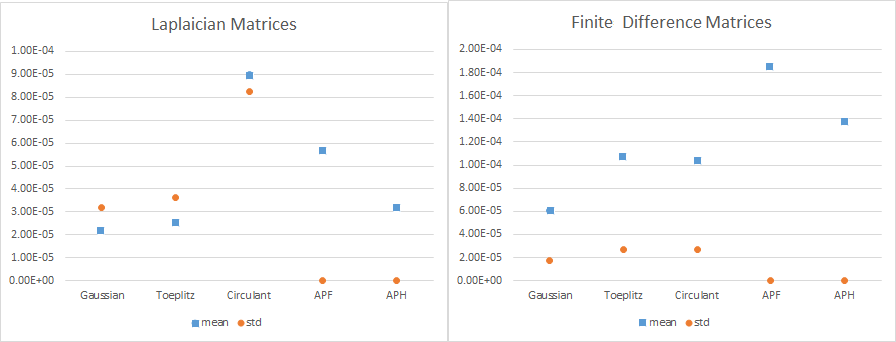}
\caption{Error norms in the tests of Section \ref{ststslo}}
\label{LowRkTest2}
\end{figure}

%\clearpage 

\subsection{LRA tests with additional classes of multipliers}\label{s17m} 

In this subsection we display the mean values and standard deviations
of the  error norms observed 
when we repeated the tests of the two previous subsections 
for the same three classes of input matrices
 (that is, SVD-generated, Laplacian, and matrices obtained by discretization of 
 finite difference operators), but now we applied Algorithm \ref{alg1}a with
  seventeen  additional classes of multipliers (besides its control application with
 Gaussian multipliers). 
 
We tested  Algorithm \ref{alg1}a applied to $1024\times 1024$ SVD-generated input matrices having numerical nullity $r = 32$, to $400 \times 400$ Laplacian input matrices
having numerical nullity $r = 3$, 
and
to $408 \times 800$ matrices having numerical nullity $r = 263$  and
representing finite-difference inputs. 

Then again we repeated the tests 1000 times for each class of input matrices and each 
size of an input and a multiplier, and we display the resulting average error norms 
in Table 
\ref{SuperfastTable}
 and Figures 
 \ref{SuperfastSVD}--\ref{SuperfastFD}.

We used multipliers defined as the eighteen sums of $n\times r$ matrices
of the following basic families:

\begin{itemize}
  \item%1
  3-ASPH matrices
\item%2
  3-APH matrices
\item%3
  Inverses of bidiagonal matrices 
\item%4
 Random permutation matrices
\end{itemize}

We  obtained every 3-APH matrix by applying three Hadamard's recursive steps
(\ref{eqrfd}) followed by random column permutation defined by random permutation of the integers from 1 to $n$  inclusive. We similarly
define every 3-ASPH matrix, but here we also apply random scaling  with a diagonal matrix $D=\diag(d_i)_{i=1}^n$
choosing the values of random iid variables $d_i$ under the uniform
probability distribution  from the set
$\{4, -3, -2, -1, 0, 1 ,2, 3, ,4\}$.
 
We permuted all inverses of bidiagonal matrices except for Class 5 of multipliers.

Describing our multipliers we use the following acronyms and abbreviations:
``IBD" for ``the inverse of a bidiagonal",
``MD" for ``the main diagonal", ``SB" for ``subdiagonal", and ``SP" for ``superdiagonal".
We write ``MD$i$", ``$k$th SB$i$" and ``$k$th SP$i$" in order to denote
that the main diagonal, the $k$th subdiagonal, or  the $k$th superdiagonal 
of a bidiagonal matrix, respectively,
was filled with the integer $i$.

\begin{itemize}

\item Class 0:	Gaussian
\item Class 1:	Sum of a 3-ASPH  and two IBD matrices: \\
	B1 with MD$-1$  and  2nd SB$-1$ and
	B2 with MD$+1$ and 1st SP$+1$
\item Class 2:	Sum of a 3-ASPH  and two IBD matrices: \\
        B1 with MD$+1$ and 2nd SB$-1$  and  
        B2 with MD$+1$ and 1st SP$-1$
\item Class 3:	Sum of   a 3-ASPH  and two IBD matrices: \\
	B1 with MD$+1$ and  1st SB$-1$ and
	B2 with MD $+1$  and 1st SP$-1$ 
\item Class 4:	Sum of  a 3-ASPH  and two IBD matrices: \\
	 B1 with MD$+1$ and 1st SB$+1$ and
        B2 with MD$+1$ and 1st SP$-1$
\item Class 5:	Sum of  a 3-ASPH  and two IBD matrices:  \\
	B1 with MD$+1$ and 1st SB$+1$ and B2 with MD$+1$ and 1st SP$-1$
\item Class 6:	Sum of a 3-ASPH  and three IBD matrices:\\
	B1 with MD$-1$ and  2nd SB$-1$,
	B2 with MD$+1$ and 1st SP$+1$ and
	B3 with MD$+1$ and 9th SB$+1$
\item Class 7:	Sum of a 3-ASPH  and three IBD matrices:\\
	 B1 with  MD$+1$ and  2nd SB$-1$, 
         B2 with MD$+1$ and 1st SP$-1$, and
	B3 with  MD$+1$ and  8th SP$+1$
\item Class 8:	Sum of a 3-ASPH  and three IBD matrices:\\
	B1 with   MD$+1$ and 1st SB$-1$,
	B2 with   MD$+1$ and 1st SP$-1$, and
	B3 with   MD$+1$ and 4th  SB$+1$
\item Class 9:	Sum of a 3-ASPH  and three IBD matrices:\\
	B1 with   MD$+1$ and 1st SB$+1$,
	B2 with   MD$+1$ and 1st SP$-1$, and
	B3 with   MD$-1$ and 3rd SP$+1$
\item Class 10:	Sum of three IBD matrices:\\
	B1 with   MD$+1$ and 1st SB$+1$,
	 B2 with   MD$+1$ and 1st SP$-1$, and
	 B3 with   MD$-1$ and 3rd SP$+1$
\item Class 11:	Sum of a 3-APH  and three IBD matrices:\\
	 B1 with   MD$+1$ and  2nd SB$-1$,
	 B2 with   MD$+1$ and 1st SP$-1$, and
	 B3 with   MD$+1$ and  8th SP$+1$
\item Class 12:	Sum of a 3-APH  and two IBD matrices:\\
	 B1 with   MD$+1$ and 1st SB$-1$ and
	 B2 with   MD$+1$ and 1st SP$-1$
\item Class 13:	Sum of a 3-ASPH  and a permutation matrix
\item Class 14:	Sum of a 3-ASPH  and two permutation matrices
\item Class 15:	Sum of a 3-ASPH  and three permutation matrices
\item Class 16:	Sum of a 3-APH  and three  permutation matrices
\item Class 17:	Sum of a 3-APH  and two permutation matrices
\end{itemize}

\begin{figure}[htb] 
\centering
\includegraphics[scale=0.7]{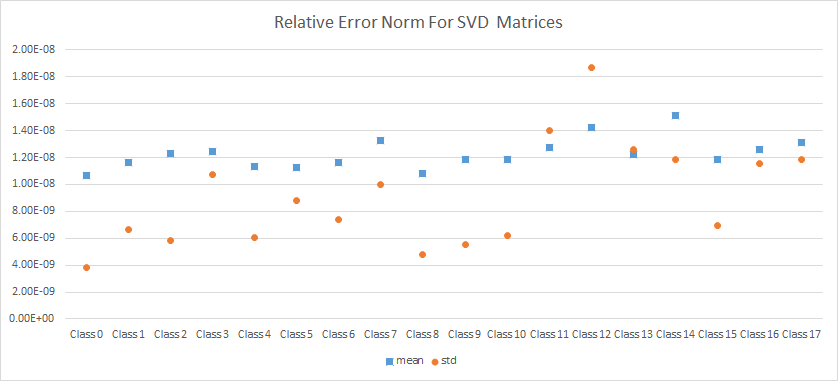}
\caption{Relative Error Norms for SVD-generated Input Matrices}
\label{SuperfastSVD}
\end{figure}

\begin{figure}[htb] 
\centering
\includegraphics[scale=0.7] {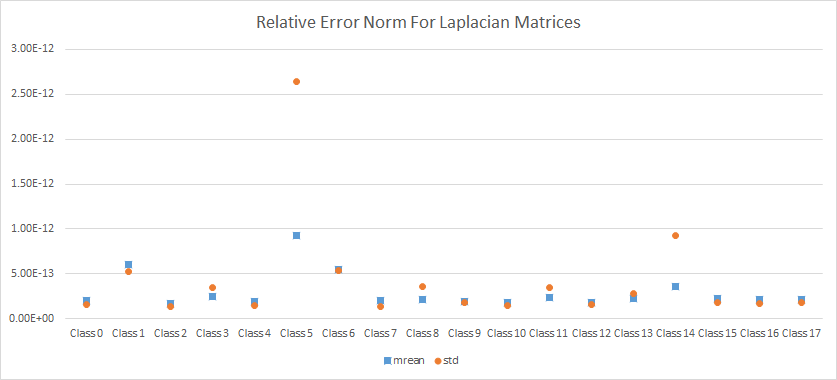}
\caption{Relative Error Norms for Lapacian Input  Matrices}
\label{SuperfastLP}
\end{figure}

\begin{figure}[htb] 
\centering
\includegraphics[scale=0.7] {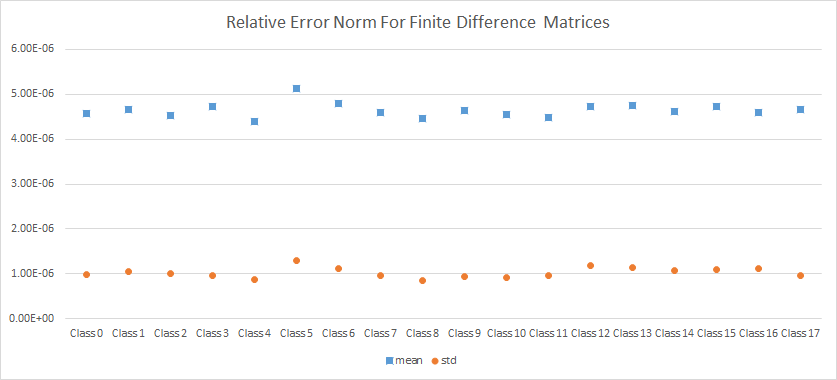}
\caption{Relative Error Norms for Finite-Difference Input Matrices}
\label{SuperfastFD}
\end{figure}

\begin{table}[htb]
\begin{center}
\begin{tabular}{|c|c|c|c|c|c|c|}
\hline
			& \multicolumn{2}{|c|}{SVD-generated Matrices} & \multicolumn{2}{|c|}{Laplacian Matrices} & \multicolumn{2}{|c|}{Finite Difference Matrices}\\\hline
 \text{Class No.} & \text{Mean} & \text{Std} & \text{Mean} & \text{Std} & \text{Mean} & \text{Std} \\\hline
Class 0 & 4.61e-09 & 4.71e-09 & 1.19e-07 & 1.86e-07 & 2.44e-06 & 2.52e-06\\\hline
Class 1 & 4.47e-09 & 5.92e-09 & 1.04e-07 & 1.82e-07 & 2.32e-06 & 2.60e-06\\\hline
Class 2 & 4.60e-09 & 5.82e-09 & 1.43e-07 & 2.17e-07 & 1.63e-06 & 1.79e-06\\\hline
Class 3 & 3.47e-09 & 3.30e-09 & 7.71e-08 & 1.35e-07 & 2.07e-06 & 2.23e-06\\\hline
Class 4 & 3.36e-09 & 3.70e-09 & 1.56e-07 & 2.64e-07 & 2.17e-06 & 2.55e-06\\\hline
Class 5 & 4.05e-09 & 3.93e-09 & 1.34e-07 & 2.28e-07 & 1.95e-06 & 2.41e-06\\\hline
Class 6 & 3.59e-09 & 3.32e-09 & 1.10e-07 & 1.46e-07 & 2.34e-06 & 2.47e-06\\\hline
Class 7 & 4.21e-09 & 4.85e-09 & 1.11e-07 & 1.86e-07 & 2.28e-06 & 2.35e-06\\\hline
Class 8 & 3.95e-09 & 3.74e-09 & 1.30e-07 & 2.18e-07 & 2.09e-06 & 2.39e-06\\\hline
Class 9 & 3.93e-09 & 3.67e-09 & 1.19e-07 & 2.04e-07 & 2.63e-06 & 2.77e-06\\\hline
Class 10 & 4.24e-09 & 6.16e-09 & 1.02e-07 & 1.79e-07 & 1.79e-06 & 1.77e-06\\\hline
Class 11 & 3.77e-09 & 3.70e-09 & 1.12e-07 & 2.11e-07 & 2.31e-06 & 3.18e-06\\\hline
Class 12 & 4.34e-09 & 4.98e-09 & 1.13e-07 & 1.83e-07 & 1.90e-06 & 2.09e-06\\\hline
Class 13 & 5.01e-09 & 8.50e-09 & 2.32e-07 & 2.33e-07 & 5.99e-06 & 7.51e-06\\\hline
Class 14 & 3.80e-09 & 4.37e-09 & 1.91e-07 & 2.13e-07 & 3.74e-06 & 4.49e-06\\\hline
Class 15 & 4.30e-09 & 4.89e-09 & 1.66e-07 & 1.82e-07 & 2.64e-06 & 3.34e-06\\\hline
Class 16 & 3.80e-09 & 4.73e-09 & 1.91e-07 & 1.95e-07 & 1.90e-06 & 2.48e-06\\\hline
Class 17 & 3.95e-09 & 4.48e-09 & 1.81e-07 & 2.01e-07 & 2.71e-06 & 3.33e-06\\\hline

\end{tabular}
\caption{Relative Error Norms in Tests with Multipliers of  Additional Classes}
 \label{SuperfastTable}
\end{center}
\end{table}

The outputs were quite accurate  even where we applied Algorithm \ref{alg1}a
with very sparse multipliers of classes 13--17.

%------------------------------------------------------------------------------

We extended these tests of Algorithm \ref{alg1}a with additional                                                                                                                                                                                                                                                                                                                                                                                                                                                                                                                                                                                                                                                                                 classes of  multipliers  to the $n\times n$ input matrices of discretized Integral Equations from the San Jose University matrix database for $n=1000$. 
The matrices came from discretization (based on Galerkin or quadrature methods) of the Fredholm  Integral Equations of the first kind.

 We applied our tests 
to the dense  matrices with smaller ratios of ``numerical rank/$n$"  from the built-in test problems in Regularization 
Tools,\footnote{See 
%database at 
 http://www.math.sjsu.edu/singular/matrices and 
  http://www2.imm.dtu.dk/$\sim$pch/Regutools 
  
For more details see Chapter 4 of the Regularization Tools Manual at \\
  http://www.imm.dtu.dk/$\sim$pcha/Regutools/RTv4manual.pdf } namely
to the following six input classes  from the Database:

\medskip

{\em baart:}       Fredholm Integral Equation of the first kind,

{\em shaw:}        one-dimensional image restoration model,
 
{\em gravity:}     1-D gravity surveying model problem,
 
%heat:        inverse heat equation.

%parallax:    Stellar parallax problem with 28 fixed, real observations.

%tomo:        a 2D tomography test problem. 

%ursell:      integral equation with no square integrable solution.

wing:        problem with a discontinuous
 solution,

{\em foxgood:}     severely ill-posed problem,
 
{\em inverse Laplace:}   inverse Laplace transformation.

\medskip

%We executed our  experiments
%for these classes of matrices
%on a 64-bit Windows machine with an Intel i5 dual-core 1.70 GHz processor using custom programmed software in $C^{++}$ and %compiled with LAPACK version 3.6.0 libraries.

 We summarize the results of the tests in Tables \ref{tab8.9}  
 and \ref{tab8.10},  
 where we provided the numerical rank of each input matrix in parentheses.

The results  show high output accuracy
with error norms in the range from about $10^{-6}$ to $10^{-9}$
 with the exception of multiplier classes 13-17 for the inverse Laplace input matrix, in which case the range was from about $10^{-3}$ to $10^{-5}$.

\begin{table}
\begin{center}
\begin{tabular}{|c|c|c|c|c|c|c|}
\hline
			& \multicolumn{2}{|c|}{wing (4)} & \multicolumn{2}{|c|}{baart (6)} & \multicolumn{2}{|c|}{inverse Laplace (25)}\\\hline
 
 \text{Class No.} & \text{Mean} & \text{Std} & \text{Mean}  & \text{Std} & \text{Mean} 	& \text{Std} \\\hline
 Class 0 	&	1.20E-08	&	6.30E-08	&	1.82E-09	&	1.09E-08	&	2.72E-08	&	7.50E-08\\\hline
 Class 1 	&	3.12E-09	&	1.23E-08	&	1.85E-09	&	1.71E-08	&	5.91E-08	&	2.32E-07\\\hline
 Class 2 	&	1.10E-09	&	5.36E-09	&	2.01E-10	&	9.93E-10	&	4.31E-08	&	1.22E-07\\\hline
 Class 3 	&	8.08E-09	&	5.77E-08	&	5.14E-10	&	3.15E-09	&	2.00E-08	&	4.95E-08\\\hline
 Class 4 	&	2.62E-09	&	1.27E-08	&	2.10E-09	&	1.68E-08	&	2.13E-08	&	6.30E-08\\\hline
 Class 5 	&	1.94E-09	&	9.99E-09	&	2.16E-09	&	1.91E-08	&	1.42E-07	&	5.12E-07\\\hline
 Class 6 	&	2.34E-09	&	1.90E-08	&	3.81E-09	&	1.51E-08	&	2.83E-08	&	9.32E-08\\\hline
 Class 7 	&	1.59E-09	&	8.88E-09	&	4.34E-10	&	2.40E-09	&	3.79E-08	&	9.40E-08\\\hline
 Class 8 	&	1.78E-08	&	1.16E-07	&	1.32E-09	&	8.96E-09	&	1.65E-08	&	4.54E-08\\\hline
 Class 9 	&	2.72E-08	&	2.56E-07	&	4.41E-10	&	2.11E-09	&	4.68E-08	&	1.92E-07\\\hline
 Class 10 	&	4.68E-10	&	3.21E-09	&	1.12E-09	&	6.77E-09	&	3.07E-08	&	7.79E-08\\\hline
 Class 11 	&	2.06E-09	&	1.40E-08	&	5.53E-10	&	2.50E-09	&	3.26E-08	&	7.80E-08\\\hline
 Class 12 	&	2.19E-09	&	1.05E-08	&	3.28E-10	&	1.71E-09	&	2.35E-08	&	7.15E-08\\\hline
 Class 13 	&	2.00E-09	&	1.34E-08	&	2.46E-09	&	1.40E-08	&	1.21E-03	&	4.13E-03\\\hline
 Class 14 	&	7.96E-09	&	4.18E-08	&	5.31E-10	&	3.00E-09	&	6.61E-04	&	2.83E-03\\\hline
 Class 15 	&	3.01E-09	&	2.23E-08	&	5.55E-10	&	2.74E-09	&	3.35E-04	&	1.81E-03\\\hline
 Class 16 	&	2.27E-09	&	1.07E-08	&	2.10E-09	&	1.28E-08	&	3.83E-05	&	1.66E-04\\\hline
 Class 17 	&	3.66E-09	&	1.57E-08	&	1.10E-09	&	5.58E-09	&	3.58E-04	&	2.07E-03\\\hline

\end{tabular}
\caption{Relative Error Norms for benchmark input matrices of discretized Integral Equations from the San Jose University singular matrix database in Tests with Multipliers of Additional Classes}\label{tab8.9}
\end{center}
\end{table}

\begin{table}[htb] 
\begin{center}
\begin{tabular}{|c|c|c|c|c|c|c|}
\hline
			& \multicolumn{2}{|c|}{foxgood (10)} & \multicolumn{2}{|c|}{shaw (12)} & \multicolumn{2}{|c|}{gravity (25)}\\\hline
 
 \text{Class No.} & \text{Mean} & \text{Std} & \text{Mean}  & \text{Std} & \text{Mean} 	& \text{Std} \\\hline
 Class 0 	&	1.56E-07	&	4.90E-07	&	2.89E-09	&	1.50E-08	&	2.12E-08	&	4.86E-08\\\hline
 Class 1 	&	3.70E-07	&	2.33E-06	&	1.79E-08	&	8.70E-08	&	3.94E-08	&	1.14E-07\\\hline
 Class 2 	&	2.03E-07	&	4.86E-07	&	1.99E-09	&	1.19E-08	&	1.98E-08	&	5.66E-08\\\hline
 Class 3 	&	1.46E-07	&	4.06E-07	&	7.74E-09	&	4.15E-08	&	4.15E-08	&	1.15E-07\\\hline
 Class 4 	&	8.45E-08	&	2.51E-07	&	7.18E-09	&	3.32E-08	&	3.22E-08	&	8.57E-08\\\hline
 Class 5 	&	3.46E-07	&	1.21E-06	&	3.88E-09	&	2.11E-08	&	3.53E-08	&	9.66E-08\\\hline
 Class 6 	&	1.05E-06	&	5.07E-06	&	7.11E-09	&	3.78E-08	&	3.86E-08	&	9.54E-08\\\hline
 Class 7 	&	1.89E-07	&	5.36E-07	&	1.51E-08	&	5.48E-08	&	1.96E-08	&	6.43E-08\\\hline
 Class 8 	&	1.30E-07	&	4.20E-07	&	7.88E-09	&	3.60E-08	&	5.12E-08	&	1.09E-07\\\hline
 Class 9 	&	2.05E-07	&	5.74E-07	&	8.96E-09	&	3.45E-08	&	2.90E-08	&	7.73E-08\\\hline
 Class 10 	&	1.69E-07	&	6.72E-07	&	1.07E-08	&	5.01E-08	&	2.66E-08	&	7.15E-08\\\hline
 Class 11 	&	1.85E-07	&	6.20E-07	&	6.48E-09	&	3.29E-08	&	2.16E-08	&	7.46E-08\\\hline
 Class 12 	&	7.60E-08	&	2.38E-07	&	9.21E-09	&	4.52E-08	&	3.98E-08	&	1.21E-07\\\hline
 Class 13 	&	1.76E-06	&	3.76E-06	&	1.46E-08	&	5.92E-08	&	4.81E-08	&	1.26E-07\\\hline
 Class 14 	&	9.77E-07	&	1.71E-06	&	1.11E-08	&	6.67E-08	&	2.82E-08	&	8.37E-08\\\hline
 Class 15 	&	7.16E-07	&	1.14E-06	&	1.87E-08	&	1.04E-07	&	5.70E-08	&	2.52E-07\\\hline
 Class 16 	&	7.52E-07	&	1.24E-06	&	4.77E-09	&	1.79E-08	&	6.32E-08	&	1.99E-07\\\hline
 Class 17 	&	9.99E-07	&	2.27E-06	&	1.03E-08	&	3.81E-08	&	3.94E-08	&	1.00E-07\\\hline

\end{tabular}
\caption{Relative Error Norms for benchmark input matrices of discretized Integral Equations from the San Jose University singular matrix database in Tests with Multipliers of Additional Classes} \label{tab8.10}
\end{center}
\end{table}

%------------------------------------------------------------------------------

\subsection{Testing perturbation of leverage scores}
\label{ststlvrg}

%------------------------------------------------------------------------------

Table \ref{tb:lscore}
%\ref{lscore} 
shows the means and standard deviations of 
the norms of the relative errors of approximation of the input matrix $W$ and of its LRA $AB$ and similar data for the maximum difference between the SVD-based leverage scores of the pairs of these matrices. We also include  numerical ranks of the input matrices 
$W$ defined up to  tolerance  
$10^{-6}$.

In these tests we reused 
input matrices $W$ and their approximations $AB$ from our tests in Section \ref{sinteq} (using the Singular Matrix Database of San Jose University).

In addition, the last three lines of Table \ref{tb:lscore}
%\ref{lscore}
 show similar results for  perturbed
 diagonally scaled factor-Gaussian matrices $GH$ with expected numerical rank $r$
approximating input matrices $W$ up to  perturbations.

%------------------------------------------------------------------------------

\begin{table}[ht]\label{lscore}
\centering
\begin{tabular}{|c|c c|c c|c c|}
\hline
 & & &\multicolumn{2}{c|}{LRA Rel Error} & \multicolumn{2}{c|}{Leverage Score Error}\\\hline
Input Matrix &	r &	rank &	mean &	std &	mean &	std\\\hline

baart &	4 &	6 &	6.57e-04 &	1.17e-03 &	1.57e-05 &	5.81e-05\\\hline
baart &	6 &	6 &	7.25e-07 &	9.32e-07 &	5.10e-06 &	3.32e-05\\\hline
baart &	8 &	6 &	7.74e-10 &	2.05e-09 &	1.15e-06 &	3.70e-06\\\hline
foxgood &	8 &	10 &	5.48e-05 &	5.70e-05 &	7.89e-03 &	7.04e-03\\\hline
foxgood &	10 &	10 &	9.09e-06 &	8.45e-06 &	1.06e-02 &	6.71e-03\\\hline
foxgood &	12 &	10 &	1.85e-06 &	1.68e-06 &	5.60e-03 &	3.42e-03\\\hline
gravity &	23 &	25 &	3.27e-06 &	1.82e-06 &	4.02e-04 &	3.30e-04\\\hline
gravity &	25 &	25 &	8.69e-07 &	7.03e-07 &	4.49e-04 &	3.24e-04\\\hline
gravity &	27 &	25 &	2.59e-07 &	2.88e-07 &	4.64e-04 &	3.61e-04\\\hline
laplace &	23 &	25 &	2.45e-05 &	9.40e-05 &	4.85e-04 &	3.03e-04\\\hline
laplace &	25 &	25 &	3.73e-06 &	1.30e-05 &	4.47e-04 &	2.78e-04\\\hline
laplace &	27 &	25 &	1.30e-06 &	4.67e-06 &	3.57e-04 &	2.24e-04\\\hline
shaw &	10 &	12 &	6.40e-05 &	1.16e-04 &	2.80e-04 &	5.17e-04\\\hline
shaw &	12 &	12 &	1.61e-06 &	1.60e-06 &	2.10e-04 &	2.70e-04\\\hline
shaw &	14 &	12 &	4.11e-08 &	1.00e-07 &	9.24e-05 &	2.01e-04\\\hline
wing &	2 &	4 &	1.99e-02 &	3.25e-02 &	5.17e-05 &	2.07e-04\\\hline
wing &	4 &	4 &	7.75e-06 &	1.59e-05 &	7.17e-06 &	2.30e-05\\\hline
wing &	6 &	4 &	2.57e-09 &	1.15e-08 &	9.84e-06 &	5.52e-05\\\hline
factor-Gaussian &	25 &	25 &	1.61e-05 &	3.19e-05 &	4.05e-08 &	8.34e-08\\\hline
factor-Gaussian &	50 &	50 &	2.29e-05 &	7.56e-05 &	2.88e-08 &	6.82e-08\\\hline
factor-Gaussian &	75 &	75 &	4.55e-05 &	1.90e-04 &	1.97e-08 &	2.67e-08\\\hline

\end{tabular}
\caption{Tests for the perturbation of leverage scores}
\label{tb:lscore}
\end{table}

%------------------------------------------------------------------------------

\section{Numerical Tests of Primitive, Cynical, and Cross-Approximation (C--A) Algorithms for LRA and HSS matrices}
%VP
\label{sprc-a}

%------------------------------------------------------------------------------

%\section{Numerical Experiments}\label{sexp}
 
%------------------------------------------------------------------------------

\subsection{Test Overview}

%------------------------------------------------------------------------------
 
We cover our tests of 
Primitive and C--A algorithms for CUR LRA of random input matrices and benchmark matrices of discretized Integral and Partial  Differential Equations (PDEs). We have performed
 the tests   
 in the Graduate Center of the City University of New York 
 %on a Dell computer with the Intel Core 2 %2.50 GHz processor and 4G memory running 
%Windows 7 and 
by using  MATLAB. In particular we applied its standard normal distribution function "randn()" in order to
generate Gaussian matrices and	 calculated 
 numerical ranks of the input matrices
 by using the MATLAB's function 
 "rank(-,1e-6)",
 which only counts singular values greater than $10^{-6}$. 
  
 Our tables display the  mean value of
 the
 spectral norm of the relative 
 output error over 1000 runs for every class of inputs
 as well as the standard deviation (std).

  In Section \ref{ExpHMT1} we present the results of our tests where we incorporate C--A algorithms
  into the computation of low rank generators for HSS matrices.

%------------------------------------------------------------------------------

\subsection{Four algorithms used}

%------------------------------------------------------------------------------

 In our tests 
we applied and compared the following four algorithms
for computing CUR LRA to input matrices $W$ having numerical rank $r$:
\begin{itemize}
\item
{\bf Tests 1 (The
rtf5+
 algorithm for $k=l=r$):}
Randomly choose two index sets $\mathcal{I}$ and $\mathcal{J}$, both of cardinality $r$, then compute a nucleus 
$U=W_{\mathcal{I}, \mathcal{J}}^{-1}$ and define  CUR LRA
\begin{equation}\label{eqtsts1}
W':=CUR=W_{:, \mathcal{J}} \cdot  W_{\mathcal{I}, \mathcal{J}}^{-1} \cdot W_{\mathcal{I},\cdot}.
\end{equation} 
%\item   
%{\bf Method 2}
%{\color{red} We need to drop this method, correct?}
%
%randomly choose index sets $\mathcal{I}$ and $\mathcal{J}$, then construct CUR decomposition from LRA of $A$ with respect to $A_{\mathcal{I},\cdot}$ and $A_{:, \mathcal{J}}$. More specifically, compute compact QR factorizations
\item    
{\bf Tests 2 (Five loops of  C--A):}
Randomly choose an  initial row index set 
$\mathcal{I}_0$ of cardinality $r$, then perform five loops of C--A
  by applying Algorithm 1 of \cite{P00}
 as a subalgorithm  that
  produces $r\times r$ CUR generators.
  At the end compute a nucleus $U$ and define 
CUR LRA  as
in Tests 1.
\end{itemize} 
 
\begin{itemize}
\item
{\bf Tests 3 (A Cynical algorithm for $p=q=4r$ and $k=l=r$):}
Randomly choose a row index set $\mathcal{K}$ and a column index set $\mathcal{L}$,
both of cardinality $4r$, and then apply Algs. 1 and 2 from \cite{P00}
to compute a $r\times r$ submatrix $W_{\mathcal{I}, \mathcal{J}}$ of $W_{\mathcal{K}, \mathcal{L}}$  having locally maximal volume. Compute a nucleus and obtain  CUR LRA by applying 
equation (\ref{eqtsts1}).

\item
{\bf Tests 4 (Combination of a single 
 C--A loop with Tests 3):}
Randomly choose a  column index set 
$\mathcal{L}$
of cardinality $4r$; then
perform a single C--A
loop
(made up of  a single horizontal  step and  a single vertical step): First by applying Alg. 1 from \cite{P00} find an index set $\mathcal{K}'$ of cardinality $4r$ such that  $W_{\mathcal{K}', \mathcal{L}}$ has locally maximal volume in  $W_{:, \mathcal{L}}$, then by applying this algorithm to matrix  $W_{\mathcal{K}',:}$ find an index set 
$\mathcal{L}'$ of cardinality $4r$ such that  $W_{\mathcal{K}', \mathcal{L}'}$ has locally maximal volume in $W_{\mathcal{K}',:}$. Then proceed as in Tests 3 -- find an $r\times r$ submatrix $W_{\mathcal{I}, \mathcal{J}}$ having locally maximal volume in $W_{\mathcal{K}', \mathcal{L}'}$, compute a nucleus, and define CUR LRA.
\end{itemize}

\subsection{CUR LRA of random input matrices}
In the tests of this subsection 
we  
% set $\epsilon = 10^{-10}$ and 
computed CUR LRA with random 
row- and column-selection for
a perturbed $n\times n$
factor-Gaussian matrices with expected 
 rank $r$, that is, matrices $W$ in the form
%\begin{equation}
$$W = G_1 * G_2 + 10^{-10} G_3,$$
%\end{equation}
for three Gaussian matrices $G_1$
of size $n\times r$, $G_2$ of size $r\times n$,
and  $G_3$
 of size $n\times n$. 
%(Recall from Definition \ref{defgsfc} %that we call Gaussian a random matrix %whose entries are i.i.d. standard %Gaussian variables.)
Table \ref{tb_ranrc} shows the test results for all four test algorithms for $n =256, 512, 1024$ and $r = 8, 16, 32$.

Tests 2 have output the mean values of
the relative error norms in the range $[10^{-6},10^{-7}]$; other tests mostly in the range $[10^{-4},10^{-5}]$.
 
%The results suggest that Method 2 has smaller approximation error than Method 1, while Method 3 further decreases the computational error.
 
\begin{table}[ht]

\begin{center}
\begin{tabular}{|c c|c c|c c|c c|c c|}
\hline     
 &  &\multicolumn{2}{c|}{\bf Tests 1} & \multicolumn{2}{c|}{\bf Tests 2}&\multicolumn{2}{c|}{\bf Tests 3} & \multicolumn{2}{c|}{\bf Tests 4}\\ \hline
{\bf n} & {\bf r} & {\bf mean} & {\bf std}  & {\bf mean} & {\bf std}& {\bf mean} & {\bf std} & {\bf mean} & {\bf std}\\ \hline

%64 & 8 & 1.69e-06 & 1.99e-05 & 4.56e-07 & 1.14e-05\\ \hline
%64 & 16 & 2.85e-06 & 2.14e-05 & 1.31e-07 & 8.30e-07\\ \hline
%64 & 32 & 1.15e-04 & 3.11e-03 & 2.49e-07 & 2.42e-06\\ \hline
%128 & 8 & 5.10e-06 & 5.79e-05 & 1.73e-06 & 4.66e-05\\ \hline
%128 & 16 & 6.04e-06 & 5.10e-05 & 2.54e-07 & 1.26e-06\\ \hline
%128 & 32 & 1.70e-05 & 1.45e-04 & 4.56e-07 & 3.28e-06\\ \hline
256 & 8 & 1.51e-05 & 1.40e-04 & 5.39e-07 & 5.31e-06& 8.15e-06 & 6.11e-05 & 8.58e-06 & 1.12e-04\\ \hline
256 & 16 & 5.22e-05 & 8.49e-04 & 5.06e-07 & 1.38e-06& 1.52e-05 & 8.86e-05 & 1.38e-05 & 7.71e-05\\ \hline
256 & 32 & 2.86e-05 & 3.03e-04 & 1.29e-06 & 1.30e-05& 4.39e-05 & 3.22e-04 & 1.22e-04 & 9.30e-04\\ \hline
512 & 8 & 1.47e-05 & 1.36e-04 & 3.64e-06 & 8.56e-05&2.04e-05 & 2.77e-04 & 1.54e-05 & 7.43e-05\\ \hline
512 & 16 & 3.44e-05 & 3.96e-04 & 8.51e-06 & 1.92e-04&  2.46e-05 & 1.29e-04 & 1.92e-05 & 7.14e-05\\ \hline
512 & 32 & 8.83e-05 & 1.41e-03 & 2.27e-06 & 1.55e-05& 9.06e-05 & 1.06e-03 & 2.14e-05 & 3.98e-05\\ \hline
1024 & 8 & 3.11e-05 & 2.00e-04 & 4.21e-06 & 5.79e-05& 3.64e-05 & 2.06e-04 & 1.49e-04 & 1.34e-03\\ \hline
1024 & 16 & 1.60e-04 & 3.87e-03 & 4.57e-06 & 3.55e-05& 1.72e-04 & 3.54e-03 & 4.34e-05 & 1.11e-04\\ \hline
1024 & 32 & 1.72e-04 & 1.89e-03 & 3.20e-06 & 1.09e-05& 1.78e-04 & 1.68e-03 & 1.43e-04 & 6.51e-04\\ \hline
\end{tabular}
\caption{CUR LRA of random input matrices}
\label{tb_ranrc}
\end{center}
\end{table}
 
% - - - - - - - - - - - - - - - - - - - - - - - - - - - - - - - - - - - - -

\subsection{CUR LRA of   matrices of discretized Integral Equations}\label{sinteq}
Table  \ref{tabLowRkTest2} displays the mean values of the relative error norms (mostly in the range $[10^{-6},10^{-7}]$) that we observed in
 Tests 2  applied to  
 $1,000\times 1,000$
 matrices from the Singular Matrix Database
of the San Jose 
University. (Tests 1   
 produced much less accurate 
CUR LRA for the same input sets, and we do not display their results.)  We  
have tested dense  matrices with smaller ratios of "numerical rank/$\min(m, n)$"  from the built-in test problems in Regularization 
Tools.\footnote{See 
%database at 
 http://www.math.sjsu.edu/singular/matrices and 
  http://www2.imm.dtu.dk/$\sim$pch/Regutools 
  
For more details see Chapter 4 of the Regularization Tools Manual at \\
  http://www.imm.dtu.dk/$\sim$pcha/Regutools/RTv4manual.pdf } 
The matrices came from discretization (based on Galerkin or quadrature methods) of the Fredholm  Integral Equations of the first kind.

 We applied our tests to the following six input classes  from the Database:

\medskip

{\em baart:}       Fredholm Integral Equation of the first kind,

{\em shaw:}        one-dimensional image restoration model,
 
{\em gravity:}     1-D gravity surveying model problem,
 
%heat:        inverse heat equation.

%parallax:    Stellar parallax problem with 28 fixed, real observations.

%tomo:        a 2D tomography test problem. 

%ursell:      integral equation with no square integrable solution.

wing:        problem with a discontinuous
 solution,

{\em foxgood:}     severely ill-posed problem,
 
{\em inverse Laplace:}   inverse Laplace transformation.

%\medskip

% We tested three choices of the number 
% $r$ for row- and column-sampling:
%$$r: = \textit{numerical rank}-2, +0, +2.%$$ 
 
 \begin{table}[ht]
\begin{center}
\begin{tabular}{|c|c c|c c|}
\hline
 &  &  & \multicolumn{2}{c|}{\bf Tests 2}\\ \hline
{\bf Inputs}&{\bf m}& ${\bf r}$  & {\bf mean} & {\bf std}\\ \hline

\multirow{3}*{baart}
&1000 & 4 & 1.69e-04 & 2.63e-06 \\ \cline{2-5}
&1000 & 6 & 1.94e-07 & 3.57e-09 \\ \cline{2-5} 
&1000 & 8 & 2.42e-09 & 9.03e-10 \\ \hline
\multirow{3}*{shaw}
&1000 & 10 & 9.75e-06 & 3.12e-07 \\ \cline{2-5} 
&1000 & 12 & 3.02e-07 & 6.84e-09 \\ \cline{2-5} 
&1000 & 14 & 5.25e-09 & 3.02e-10 \\ \hline 
\multirow{3}*{gravity}
&1000& 23 & 1.32e-06 & 6.47e-07 \\ \cline{2-5} 
&1000 & 25 & 3.35e-07 & 1.97e-07 \\ \cline{2-5} 
&1000 & 27 & 9.08e-08 & 5.73e-08 \\ \hline
\multirow{3}*{wing}
&1000 & 2 & 9.23e-03 & 1.46e-04 \\ \cline{2-5} 
&1000 & 4 & 1.92e-06 & 8.78e-09 \\ \cline{2-5} 
&1000 & 6 &   8.24e-10 & 9.79e-11\\ \hline 
\multirow{3}*{foxgood}
&1000 & 8 & 2.54e-05 & 7.33e-06 \\ \cline{2-5} 
&1000 & 10 & 7.25e-06 & 1.09e-06 \\ \cline{2-5} 
&1000 & 12 & 1.57e-06 & 4.59e-07 \\ \hline 
\multirow{3}*{inverse Laplace}
&1000 & 23 & 1.04e-06 & 2.85e-07 \\ \cline{2-5} 
&1000 & 25 & 2.40e-07 & 6.88e-08 \\ \cline{2-5} 
&1000 & 27 & 5.53e-08 & 2.00e-08 \\ \hline 
\end{tabular}
\caption{CUR LRA of benchmark input matrices of discretized Integral Equations from the  San Jose University Singular Matrix Database}
\label{tabLowRkTest2}
\end{center}
\end{table}

% - - - - - - - - - - - - - - - - - - - - - - - - - - - - - - - - - - - - -

%\subsection{Tests for inputs generated via the discretization of a Laplacian operator
%and via the approximation of an inverse finite-difference operator}
\subsection{Tests with bidiagonal pre-processing for benchmark input  matrices from \cite{HMT11}}
\label{sbdpr} 

Tables 
\ref{tb_Lp}
 and \ref{tb_lr} 
display the  results of Tests 1, 3, and 4 applied
 to pre-processed matrices  of two kinds,
 from \cite[Section 7.1 and 7.2]{HMT11}, namely, the matrices of
 classes (i) and (ii) from Section \ref{ststslo}.
 
Application of Tests 1, 3, and 4
to the matrices of class (i) without pre-processing tended to produce results with large errors, and so we pre-processed
   every input matrix by
   multiplying it by 20  matrices,
each obtained by means of
random column permutations 
of a random bidiagonal matrix (see Section \ref{sgpbd}).
%of (\ref{def11}). 
%For each setting we applied 10 different %random multipliers, and 
%then again we display the mean value and the %standard deviation of the results.
  Then we observed output errors in the range $[10^{-3},10^{-8}]$, with Tests 4 showing the best performance.
  
  In the case of the input matrices of class (ii) the  results of application of Tests 1, 3, and 4 were similar
with all structured  and Gaussian multipliers.
    
\begin{table}[ht]
\label{LowRktest3}
\begin{center}  
\begin{tabular}{|c c|c c|c c|}
\hline
 &  &\multicolumn{2}{c|}{\bf Tests 1}& \multicolumn{2}{c|}{\bf Tests 4}\\ \hline
{\bf n} & {\bf r}  & {\bf mean} & {\bf std} & {\bf mean} & {\bf std}\\ \hline

256 & 31 & 1.37e-04 & 2.43e-04& 9.46e-05 & 2.11e-04\\ \hline
256 & 35 & 5.45e-05 & 7.11e-05& 1.03e-05 & 1.08e-05\\ \hline
256 & 39 & 6.18e-06 & 6.32e-06& 1.24e-06 & 1.72e-06\\ \hline
512 & 31 & 7.80e-05 & 6.00e-05& 2.04e-05 & 1.52e-05\\ \hline
512 & 35 & 1.56e-04 & 1.53e-04& 6.74e-05 & 1.79e-04\\ \hline
512 & 39 & 5.91e-05 & 1.10e-04& 4.27e-05 & 1.20e-04\\ \hline
1024 & 31 & 9.91e-05 & 6.69e-05& 2.79e-05 & 3.13e-05\\ \hline
1024 & 35 & 4.87e-05 & 4.35e-05& 1.66e-05 & 1.50e-05\\ \hline
1024 & 39 & 6.11e-05 & 1.33e-04& 3.83e-06 & 5.77e-06\\ \hline
\end{tabular}
\caption{CUR LRA of Laplacian input matrices}
\label{tb_Lp}
\end{center}
\end{table}
 
\begin{table}[ht]

\begin{center}
\begin{tabular}{|c c|c c|c c|c c|}
\hline
 &  &\multicolumn{2}{c|}{\bf Tests 1} & \multicolumn{2}{c|}{\bf Tests 3}& \multicolumn{2}{c|}{\bf Tests 4}\\ \hline
{\bf n} & {\bf r} & {\bf mean} & {\bf std} & {\bf mean} & {\bf std} & {\bf mean} & {\bf std}\\ \hline

800 & 78 & 4.85e-03 & 4.25e-03& 3.30e-03 & 8.95e-03& 3.71e-05 & 3.27e-05\\ \hline
800 & 82 & 2.67e-03 & 3.08e-03& 4.62e-04 & 6.12e-04& 2.23e-05 & 2.24e-05\\ \hline
800 & 86 & 2.14e-03 & 1.29e-03& 4.13e-04 & 8.45e-04& 6.73e-05 & 9.37e-05\\ \hline
1600 & 111 & 1.66e-01 & 4.71e-01& 1.11e-03 & 1.96e-03& 1.21e-04 & 1.17e-04\\ \hline
1600 & 115 & 3.75e-03 & 3.18e-03& 1.96e-03 & 3.93e-03& 4.03e-05 & 2.79e-05\\ \hline
1600 & 119 & 3.54e-03 & 2.27e-03& 5.56e-04 & 7.65e-04& 5.38e-05 & 8.49e-05\\ \hline
3200 & 152 & 1.87e-03 & 1.37e-03& 3.23e-03 & 3.12e-03& 1.68e-04 & 2.30e-04\\ \hline
3200 & 156 & 1.92e-03 & 8.61e-04& 1.66e-03 & 1.65e-03& 1.86e-04 & 1.17e-04\\ \hline
3200 & 160 & 2.43e-03 & 2.00e-03& 1.98e-03 & 3.32e-03& 1.35e-04 & 1.57e-04\\ \hline
\end{tabular}
\caption{CUR LRA of finite difference matrices}
\label{tb_lr}
\end{center}
\end{table}

\subsection{Tests with abridged  randomized Hadamard and Fourier pre-pro\-cess\-ing}
\label{tabahaf}

Table \ref{tb_fh} displays the results of our Tests 2 for  CUR LRA with abridged randomized
Hadamard and Fourier pre-processing. We used the same input matrices as in previous two subsections. 
For these input matrices Tests 1 have no longer output stable accurate LRA. 
For the data from discretized integral equations of Section \ref{sinteq} 
we observed relative error norm bounds in the range $[10^{-6},10^{-7}]$; for the data from Section \ref{LowRkTest2} they were near $10^{-3}$.

\begin{table}[ht]

\begin{center}
\begin{tabular}{|c |c c c|c c|c c|}
\hline
Multipliers & & & &\multicolumn{2}{c|}{\bf Hadamard} & \multicolumn{2}{c|}{\bf  Fourier}\\ \hline
{Input Matrix} &  {\bf m} & {\bf n} & {\bf r} & {\bf mean} & {\bf std}  & {\bf mean} & {\bf std}\\ \hline
gravity &   1000 & 1000 & 25 & 2.72e-07 & 3.95e-08 & 2.78e-07 & 4.06e-08 \\ \hline
wing  &   1000 & 1000& 4 & 1.22e-06 & 1.89e-08 & 1.22e-06 & 2.15e-08 \\ \hline
foxgood & 1000 & 1000& 10 & 4.49e-06 & 6.04e-07 & 4.50e-06 & 5.17e-07 \\ \hline
shaw &   1000 & 1000& 12 & 3.92e-07 & 2.88e-08 & 3.91e-07 & 2.98e-08 \\ \hline
bart & 1000 & 1000& 6 & 1.49e-07 & 1.37e-08 & 1.49e-07 & 1.33e-08 \\ \hline
inverse Laplace & 1000 & 1000& 25 & 3.62e-07 & 1.00e-07 & 3.45e-07 & 8.64e-08 \\ \hline
\multirow{3}*{Laplacian}  
&256 & 256 & 15 & 4.08e-03 & 1.14e-03 & 3.94e-03 & 5.21e-04 \\ \cline{2-8} 
&512 & 512 & 15 & 3.77e-03 & 1.34e-03 & 4.28e-03 & 6.07e-04 \\ \cline{2-8} 
&1024 & 1024 & 15 & 3.97e-03 & 1.22e-03 & 4.09e-03 & 4.47e-04 \\ \hline
\multirow{3}*{finite difference}  
&408 & 800 & 41 & 4.50e-03 & 1.12e-03 & 3.76e-03 & 8.36e-04 \\ \cline{2-8} 
&808 & 1600 & 59 & 4.01e-03 & 1.10e-03 &   3.80e-03 & 1.70e-03 \\ \cline{2-8} 
&1608 & 3200 & 80 & 4.60e-03 & 1.53e-03 & 3.85e-03 & 1.27e-03\\ \hline
\end{tabular}
\caption{Tests 2 for CUR LRA with ARFT/ARHT pre-processors}
\label{tb_fh}
\end{center}
\end{table}

%------------------------------------------------------------------------------

\subsection{Testing C--A acceleration of the random sampling algorithms of   \cite{DMM08}}
\label{ststc-alvrg}

%------------------------------------------------------------------------------
 
 Tables 6
 %\ref{tabcadmm1} 
 and  
 %\ref{tabcadmm}
 7 display the results of our tests  where we performed eight C--A iterations for the input matrices of Section \ref{sinteq}  by applying  Algorithm 1 of \cite{DMM08} to all vertical and horizontal sketches
 (see the lines marked ``C--A"), and, so overall the computations are superfast unless the C--A steps become too numerous. For comparison with  this algorithm, we computed LRA of the same matrices by applying to them fast Algorithm 2
 of \cite{DMM08} (see the lines marked ``CUR"). The columns of the tables marked with "nrank" display the numerical rank of an input matrix.
 The columns of the tables  marked with "$k=l$" show the number of rows and  columns in a square matrix of CUR generator. The fast algorithms have output
closer approximations, but in most cases just slightly closer. 

In these tests  the superfast algorithm consistently yielded the same or nearly the same (within at most a factor of 10) output accuracy as the fast algorithm.

%------------------------------------------------------------------------------

\begin{table}[ht]\label{tabcadmm1}
\centering
\begin{tabular}{|c|c|c|c|c|c|c|c|}
\hline
input & algorithm & m & n & nrank & k=l & mean & std \\ \hline
finite diff & C--A & 608 & 1200 & 94 & 376  & 6.74e-05 & 2.16e-05 \\ \hline
finite diff & CUR & 608 & 1200 & 94 & 376  & 6.68e-05 & 2.27e-05 \\ \hline
finite diff & C--A & 608 & 1200 & 94 & 188  & 1.42e-02 & 6.03e-02 \\ \hline
finite diff & CUR & 608 & 1200 & 94 & 188  & 1.95e-03 & 5.07e-03 \\ \hline
finite diff & C--A & 608 & 1200 & 94 & 94   & 3.21e+01 & 9.86e+01 \\ \hline
finite diff & CUR & 608 & 1200 & 94 & 94   & 3.42e+00 & 7.50e+00 \\ \hline

baart & C--A & 1000 & 1000 & 6 & 24 & 2.17e-03 & 6.46e-04 \\ \hline
baart & CUR & 1000 & 1000 & 6 & 24 & 1.98e-03 & 5.88e-04 \\ \hline
baart & C--A & 1000 & 1000 & 6 & 12 & 2.05e-03 & 1.71e-03 \\ \hline
baart & CUR & 1000 & 1000 & 6 & 12 & 1.26e-03 & 8.31e-04 \\ \hline
baart & C--A & 1000 & 1000 & 6 & 6  & 6.69e-05 & 2.72e-04 \\ \hline
baart & CUR & 1000 & 1000 & 6 & 6  & 9.33e-06 & 1.85e-05 \\ \hline

shaw & C--A & 1000 & 1000 & 12 & 48 & 7.16e-05 & 5.42e-05 \\ \hline
shaw & CUR & 1000 & 1000 & 12 & 48 & 5.73e-05 & 2.09e-05 \\ \hline
shaw & C--A & 1000 & 1000 & 12 & 24 & 6.11e-04 & 7.29e-04 \\ \hline
shaw & CUR & 1000 & 1000 & 12 & 24 & 2.62e-04 & 3.21e-04 \\ \hline
shaw & C--A & 1000 & 1000 & 12 & 12 & 6.13e-03 & 3.72e-02 \\ \hline
shaw & CUR & 1000 & 1000 & 12 & 12 & 2.22e-04 & 3.96e-04 \\ \hline
\end{tabular}
\caption{LRA errors of Cross-Approximation (C--A) tests  incorporating 
\cite[Algorithm 1]{DMM08} in comparison to the errors of stand-alone 
\cite[Algorithm 2]{DMM08} (for  three input classes from Section \ref{sinteq}).}
\end{table}

\begin{table}[ht]\label{tabcadmm}
\centering
\begin{tabular}{|c|c|c|c|c|c|c|c|}
\hline
input & algorithm & m & n & nrank & $k=l$ & mean & std \\ \hline
foxgood & C--A & 1000 & 1000 & 10 & 40 & 3.05e-04 & 2.21e-04 \\\hline
foxgood & CUR & 1000 & 1000 & 10 & 40 & 2.39e-04 & 1.92e-04 \\ \hline
foxgood & C--A & 1000 & 1000 & 10 & 20 & 1.11e-02 & 4.28e-02 \\ \hline
foxgood & CUR & 1000 & 1000 & 10 & 20 & 1.87e-04 & 4.62e-04 \\ \hline
foxgood & C--A & 1000 & 1000 & 10 & 10 & 1.13e+02 & 1.11e+03 \\ \hline
foxgood & CUR & 1000 & 1000 & 10 & 10 & 6.07e-03 & 4.37e-02 \\ \hline
wing & C--A & 1000 & 1000 & 4 & 16 & 3.51e-04 & 7.76e-04 \\ \hline
wing & CUR & 1000 & 1000 & 4 & 16 & 2.47e-04 & 6.12e-04 \\ \hline
wing & C--A & 1000 & 1000 & 4 & 8  & 8.17e-04 & 1.82e-03 \\ \hline
wing & CUR & 1000 & 1000 & 4 & 8  & 2.43e-04 & 6.94e-04 \\ \hline
wing & C--A & 1000 & 1000 & 4 & 4  & 5.81e-05 & 1.28e-04 \\ \hline
wing & CUR & 1000 & 1000 & 4 & 4  & 1.48e-05 & 1.40e-05 \\ \hline

gravity & C--A & 1000 & 1000 & 25 & 100 & 1.14e-04 & 3.68e-05 \\ \hline
gravity & CUR & 1000 & 1000 & 25 & 100 & 1.41e-04 & 4.07e-05 \\ \hline
gravity & C--A & 1000 & 1000 & 25 & 50  & 7.86e-04 & 4.97e-03 \\ \hline
gravity & CUR & 1000 & 1000 & 25 & 50  & 2.22e-04 & 1.28e-04 \\ \hline
gravity & C--A & 1000 & 1000 & 25 & 25  & 4.01e+01 & 2.80e+02 \\ \hline
gravity & CUR & 1000 & 1000 & 25 & 25  & 4.14e-02 & 1.29e-01 \\ \hline
inverse Laplace & C--A & 1000 & 1000 & 25 & 100 & 4.15e-04 & 1.91e-03 \\ \hline
inverse Laplace & CUR & 1000 & 1000 & 25 & 100 & 5.54e-05 & 2.68e-05 \\ \hline
inverse Laplace & C--A & 1000 & 1000 & 25 & 50  & 3.67e-01 & 2.67e+00 \\ \hline
inverse Laplace & CUR & 1000 & 1000 & 25 & 50 &  2.35e-02 & 1.71e-01 \\ \hline
inverse Laplace & C--A & 1000 & 1000 & 25 & 25 &  7.56e+02 & 5.58e+03 \\ \hline
inverse Laplace & CUR & 1000 & 1000 & 25 & 25 &  1.26e+03 & 9.17e+03 \\ \hline
\end{tabular}
\caption{LRA errors of Cross-Approximation (C--A) tests  incorporating 
\cite[Algorithm 1]{DMM08}  in comparison to  the errors of
stand-alone  
\cite[Algorithm 2]{DMM08} (for four input classes from Section \ref{sinteq}).} 
\end{table}

%--------------------------------------------------------------------

\subsection{Computation of CUR LRAs for benchmark HSS matrices}
%VP
\label{sHSS}

%------------------------------------------------------------------------------

We tested superfast computation of CUR LRA of the generators for the off-diagonal blocks of 
HSS matrices that approximate $1024 \times 1024$ Cauchy-like matrices 
%vp
derived from  benchmark Toeplitz matrices B, C, D, E, and F of \cite[Section 5]{XXG12}. For the computation of CUR LRA we applied the algorithm of \cite{GOSTZ10}.

%VP
Table 9.15 displays the relative errors of
the approximation of the  $1024 \times 1024$ HSS input matrices
 in the spectral and Chebyshev norms averaged over 100 tests.
 Each approximation was obtained by 
 means of   
 combining the exact diagonal blocks 
 and CUR LRA of the off-diagonal blocks
 (cf. Section \ref{sextpr}). We
 computed CUR LRA of all these blocks superfast.
 
The numerical experiments covered in Table \ref{tabhss} 
were executed on a 64-bit Windows machine with an Intel i5 dual-core 1.70 GHz processor using software custom programmed in $C^{++}$ and compiled with LAPACK version 3.6.0 libraries.

As can be expected from our formal study, already the first C--A loop consistently yielded reasonably close CUR LRA, but our  further improvement was achieved in five C--A loops in our tests for all but one of the five families of input matrices.

The reported HSS rank is the larger of the numerical ranks for the $512 \times 512$ off-diagonal blocks. This HSS rank was used as an upper bound in our binary search that determined the numerical rank of each off-diagonal block for the purpose of computing its LRA. We based 
the binary search on  minimizing the difference (in the spectral norm) between each off-diagonal block and its LRA. 

The output error norms were quite low. Even in the case of
 the matrix C, obtained from Prolate Toeplitz matrices, known to be extremely ill-conditioned, they ranged form $10^{-3}$
 to $10^{-6}$.

We have also performed further numerical experiments on all the HSS input matrices by using a hybrid LRA algorithm: we used random pre-processing with Gaussian and Hadamard (abridged and permuted)  multipliers by incorporating
 Algorithm 4.1 of \cite{HMT11}, but  only for the  off-diagonal blocks of smaller sizes while retaining our previous 
 way for computing CUR LRA of the larger off-diagonal blocks. We have not displayed the results of these experiments because they yielded no substantial improvement in accuracy in comparison to the exclusive use of the less expensive CUR LRA on all off-diagonal blocks. 

\begin{table}[ht]\label{tabhss}
\begin{center}
\begin{tabular}{|c|c|c|c|c|c|c|}
\hline
 &  & & \multicolumn{2}{c|}{\bf Spectral Norm} & \multicolumn{2}{c|}{\bf Chebyshev Norm} \\ \hline

{\bf Inputs} & {\bf C--A loops}  &  {\bf HSS rank} & {\bf mean} & {\bf std} & {\bf mean} & {\bf std} \\ \hline

%\multirow{2}*{B}
%&1 &  26 & 1.27e-05 & 5.20e-06 & 7.61e-06 & 3.92e-06  \\  \cline{2-7}
%&5 &  26 & 1.24e-05 & 5.74e-06 & 7.19e-06 & 4.19e-06  \\  \hline
%\multirow{2}*{C} 
%&1  &  16  & 1.95e-02 & 2.36e-02 & 7.37e-03 & 9.50e-03 \\ \cline{2-7}
%&5  &  16  & 1.44e-02 & 2.06e-03 & 4.80e-03 & 8.38e-04 \\  \hline
%\multirow{2}*{D} 
%&1  &  13  & 3.90e-04 & 1.26e-04 & 3.56e-04 & 1.65e-04  \\ \cline{2-7}
%&5  &  13  & 4.30e-04 & 1.29e-04 & 3.44e-04 & 1.14e-04  \\  \hline
%\multirow{2}*{E}
%&1  &  14  & 2.41e-02 & 8.45e-03 & 7.92e-03 & 3.04e-03  \\  \cline{2-7}
%&5  &  14  & 2.31e-02 & 6.44e-03 & 7.24e-03 & 2.42e-03  \\  \hline
%\multirow{2}*{F}
%&1  &  37  & 2.67e-04 & 6.75e-05 & 1.14e-04 & 2.88e-05  \\   \cline{2-7}
%&5  &  37  & 2.41e-04 & 6.38e-05 & 1.04e-04 & 3.21e-05  \\   \hline 

\multirow{2}*{B}
&1 &  26 & 8.11e-07 & 1.45e-06 & 3.19e-07 & 5.23e-07 \\  \cline{2-7}
&5 &  26 & 4.60e-08 & 6.43e-08 & 7.33e-09 & 1.22e-08 \\  \hline
\multirow{2}*{C}
&1 &  16 & 5.62e-03 & 8.99e-03 & 3.00e-03 & 4.37e-03 \\  \cline{2-7}
&5 &  16 & 3.37e-05 & 1.78e-05 & 8.77e-06 & 1.01e-05 \\  \hline
\multirow{2}*{D}
&1 &  13 & 1.12e-07 & 8.99e-08 & 1.35e-07 & 1.47e-07 \\  \cline{2-7}
&5 &  13 & 1.50e-07 & 1.82e-07 & 2.09e-07 & 2.29e-07 \\  \hline
\multirow{2}*{E}
&1 &  14 & 5.35e-04 & 6.14e-04 & 2.90e-04 & 3.51e-04 \\  \cline{2-7}
&5 &  14 & 1.90e-05 & 1.04e-05 & 5.49e-06 & 4.79e-06 \\  \hline
\multirow{2}*{F}
&1 &  37 & 1.14e-05 & 4.49e-05 & 6.02e-06 & 2.16e-05 \\  \cline{2-7}  
&5 &  37 & 4.92e-07 & 8.19e-07 & 1.12e-07 & 2.60e-07 \\  \hline

\end{tabular}
\caption{CUR C--A approximation of HSS input matrices from \cite{XXG12}}
\end{center}
%\label{tabhss}
\end{table}

%------------------------------------------------------------------------------

%-----------------------------------------------------------------------------

%\medskip

%------------------------------------------------------------------------------

\medskip

\medskip

%------------------------------------------------------------------------------

\medskip

%------------------------------------------------------------------------------

{\bf \Large Appendix} 
\appendix

\section{Randomized Error Estimates for LRA from \cite{HMT11}}\label{sgssrht} 

Halko et al. estimated the errors of Algorithm \ref{alg1}a based on 
 the following deterministic bound
on $|UV-M|$ in the case of
  any multiplier $H$ (they prove this bound in  \cite[Section 10]{HMT11}):
\begin{equation}\label{eqerrnrm} 
 |UV-M|^2\le 
 |\Sigma_2|^2+|\Sigma_2C_2C_1^+|^2
 \end{equation}
  where  
  \begin{equation}\label{eqc12}
C_1=T^*_1H,~ 
 C_2=T^*_2H,
 \end{equation}
\begin{equation}\label{eqmmr} 
 M=\begin{pmatrix} S_1&\Sigma_1&& T_1^* 
 \\  S_2&&\Sigma_2&T_2^*
 \end{pmatrix},~
 M_r=S_1\Sigma_1 T_1^*,~{\rm and}~
M-M_r=S_2\Sigma_2 T_2^*
\end{equation}
  are SVDs
 of the matrices $M$, its rank-$r$ truncation $M_r$, and  $M-M_r$,
 respectively.
 
  Notice that  $T_1^*$ is the matrix of the top $r$ singular vectors of $M$,
 $\Sigma_2=O  $ and $UV=M$ if 
 $\rank(M)=r$,
  and the $r\times l$ matrix $C_1$ has full rank $r$.  
 
By nontrivially combining
 bound (\ref{eqerrnrm}) with Lemma \ref{lepr3}  
 and Theorems \ref{thsignorm}--  \ref{thgssnpr}, Halko et al. prove in
  \cite[Theorems 10.5 and 10.6]{HMT11}
  that
 \begin{equation}\label{eqgss1}
\mathbb E||M-UV||_F^2
 \le \Big(1+\frac{r}{l-r-1}\Big)~\tau_{r+1}^2(M);
  \end{equation}
\begin{equation}\label{eqgss2}
\mathbb E||M-UV||\le \Big(1+\frac{r}{l-r-1}\Big)^{1/2}~
\sigma_{r+1}(A)+
 \frac{e\sqrt l}{l-r} \tau_{r+1}(M), 
 \end{equation} 
   in both cases provided that 
   $2\le r\le l-2$.
   
   For $l\ge r+4$
   \cite[Theorems 10.7 and 10.8]{HMT11}  bound the norms $|M-UV|$ in probability as follows:
\begin{equation}\label{eqerrnrmfr}
||M-UV||_F\le  \Big(1+
t\cdot\sqrt{\frac{3r}{l-r+1}}\Big)\tau_{r+1}(M)+ut\frac{e\sqrt l}{l-r+1}
\sigma_{r+1}(M); 
\end{equation}
 \begin{equation}\label{eqerrnrmsp} 
 ||M-UV||\le \Big [\Big (1+t\cdot\sqrt{\frac{3r}{l-r+1}}\Big)\sigma_{r+1}(A)+t\frac{e\sqrt l}{l-r+1}\tau_{r+1}(M)\Big]+ut\cdot \frac{e\sqrt l}{l-r+1}\sigma_{r+1}(M),       
  \end{equation}
  in both cases with a failure probability at most $2t^{r-l}+e^{-u^2/2}$ and in both cases Halko et al. first estimate the norms
 $ |M-UV|$ in terms of the norms
  $|\Sigma_2|$ and $|B_1^+|$
  and then invoke the assumption that
  $H$ is a Gaussian matrix and the
   estimates for the norms $|B_1^+|$.
     
   In the case of SRHT and SRFT multipliers \cite{T11} and \cite[Theorem 11.2]{HMT11}  prove that
   \begin{equation}\label{eqsrhft}
|M-UV|\le\sqrt{l+7n/l}~\tilde \sigma_{r+1}(M) 
\end{equation}
 with a failure probability at most $O(1/r)$ provided that
$$4[\sqrt r+\sqrt{8\log(rn)}]^2\log(r)\le l\le n.$$

%------------------------------------------------------------------------------
%------------------------------------------------------------------------------

\section{Proofs of the Error Estimates for Superfast LRA}\label{serrcs}  
  
%-------------------------------------------------------------

We keep writing $W_r$
for the rank-$r$ truncation of a matrix
$W$.

%------------------------------------------------------------------------------
  
\subsection{Proof of Theorems  \ref{thcrerr2} and \ref{thcrerr1} }

%---------------------------------------------
  
   \begin{lemma}\label{lew123}
Suppose that Algorithm \ref{alg1}
has  been applied to  an
 an $m\times n$ rank-$r$ 
matrix $\tilde M$  and its
perturbation $M=\tilde M+E$ and has output matrices $\tilde U$, $\tilde V$,  $U$, and $V$, respectively, such that $\tilde M=
\tilde U\tilde V$.
%with probability 1 (see Theorem 
%\ref{thmmh}).  Let
%$\eta=||(\tilde MH)^+||~||EH||<1$.
 Then 
 %\begin{equation}\label{eqt123}
$$M-U V=E-(W_1+W_2+W_3),$$
%\end{equation}
 \begin{equation}\label{eqt123||}
|M-U V|\le |E|+|W_1|+|W_2|+|W_3|. 
\end{equation}
where 
\begin{equation}\label{eq123tld}
W_1=(U-\tilde U)~U_r^+~M,~W_2=\tilde U~(U_r^+-\tilde U^+)~M,~W_3=
\tilde U~\tilde U^+~(M-\tilde M).
\end{equation} 
 Furthermore 
 \begin{equation}\label{eqw13}
W_1=HEU_r^+~M,~ W_3=\tilde U~\tilde U^+E,
\end{equation}
\begin{equation}\label{eqnunu+}
|\tilde MH|\le |A|\nu_{r,l}|H|~{\rm and}~
|(\tilde MH)^+|\le|H_r^+|\nu_{r,l}^+|A^+|.
\end{equation}
\end{lemma}
\begin{proof}
Substitute $M-E=\tilde M=\tilde U\tilde V$, $V=U^+M$, and 
 $\tilde V=\tilde U^+\tilde M$ and
 obtain $M-UV-E=-(UV^+M-\tilde U\tilde V^+\tilde M)$. Then readily verify bounds
 (\ref{eqt123||})--(\ref{eq123tld}). Substitute the equations $M-\tilde M=E$ and
 $U-\tilde U=HE$ into (\ref{eq123tld})  and obtain (\ref{eqw13}).
 Finally deduce bounds (\ref{eqnunu+})
 from  Lemmas \ref{lepr3} and 
\ref{lehg}.
\end{proof} 
\medskip

{\bf Proof of Theorems \ref{thcrerr2}.}
If Algorithm \ref{alg1}a is applied, then
%\begin{equation}\label{equtldu}
$$|U|=|U^+|=|\tilde U|=|\tilde U^+|=1$$ 
%\end{equation}
and moreover  
% \begin{equation}\label{equ+qr}
$$ ||\tilde U^+-U^+||\le \sqrt 2~||(\tilde MH)^+||~
 ||EH||_F+O(||EH||_F^2)$$
%\end{equation} 
by virtue of Lemma \ref{lepert1}.
  Thus the norm
 $|\tilde U^+-U^+|$ is in $O(|E|)$,
 like $|\tilde U-U|$, $|W_1|$, and $|W_3|$  (cf. (\ref{eqw13})). Now 
 claim (i) of Theorem \ref{thcrerr2} 
  readily follows from Lemma \ref{lew123};  claim (ii)   follows from  bounds (\ref{eqnunu+}) and claim (i).
\medskip

{\bf Proof of Theorem \ref{thcrerr1}.}
Recall that
% \begin{proof}
 by virtue of Corollary \ref{coprtinv},  
%, it holds that
\begin{equation}\label{eqmh+mh}
||\tilde U^+-U^+||\le \frac{\mu}{1-\eta}||\tilde U^+||^2
||\tilde U-U||~{\rm if}~\eta=1-||\tilde U^+||~||\tilde U-U||>0
\end{equation}
(such a provision is assumed in Theorem \ref{thcrerr1}) and if $\mu\le (1+\sqrt 5)/2$ (cf. Lemma \ref{leprtpsdinv}). Therefore in this case  
\begin{equation}\label{eq|u+|} 
 ||U^+|| \le||\tilde U^+||+\frac{\mu}{1-\eta}||\tilde U^+||^2
||\tilde U-U||=
 ||\tilde U^+||+O(|E|).
 \end{equation} 
Combine the latter relationships with the equation $\tilde U=\tilde M H$ and Lemma \ref{lew123} and deduce claim (i) of Theorem \ref{thcrerr1}.
Its claim (ii) follows from  claim (i) and (\ref{eqnunu+}).
 
% \begin{eqnarray} 
%\end{eqnarray}

%------------------------------------------------------------------------------
  
\subsection{Proof of Theorem \ref{therrfctr}}  
 
Readily deduce Theorem \ref{therrfctr}
by combining bound (\ref{eqerrnrm}) with the following lemma.
 \begin{lemma}\label{leerrfctr}
Under the assumptions of Theorem \ref{therrfctr}  let $\Sigma_2$, $C_1$,
and  $C_2$ denote the matrices of 
(\ref{eqerrnrm})--(\ref{eqmmr}).
 Then (i) $|C_2|\le|H|$,  (ii) $|\Sigma_2|\le |E|$, and  (iii) $||C_1^{+}||^{-1}\ge (\nu_{{\rm sp},r,n}\nu_{{\rm sp},r,l}^+||H_r^+||)^{-1}- 
    4\alpha||H||.$ 
     \end{lemma}
 \begin{proof}  
 Claim (i) follows because the matrix $T_2^*$ is orthogonal, and so
  $|C_2|=|T_2^*H|\le  
 |T_2^*|~|H|\le|H|$.
   
 Lemma \ref{letrnc} for $M=AB+E$ implies claim (ii). 

It remains to estimate
   the norm $||C_1^+||$. At first we
do this in the special case where   
     $E=O$ and $M=\tilde M=AB$
     for $B\in \mathcal G^{r\times n}$
     (cf. (\ref{eqfctrg})); in this case we write $\tilde C_1:=C_1$.  
        
    Let $A=S_A\Sigma_AT_A^*$ 
   and  $B=S_B\Sigma_BT_B^*$
    be SVDs.
    Then  
 $$AB=S_APT_B^*~{\rm for}~  
    P=\Sigma_AT_A^*S_B\Sigma_B,$$    
where $P,\Sigma_A$, $T_A^*$, $S_B$, and $\Sigma_B$  are $r\times r$ matrices. Let 
  $P=S_P\Sigma_PT_P^*$ be SVD, write
$$\Sigma:=\Sigma_P,~  
   S:=S_AS_P,~{\rm and}~ 
   T^*:=T_P^*T_B^*,$$
    and observe that  $S$ and 
   $T^*$ are orthogonal matrices
   of sizes $m\times r$ and $r\times n$,
   respectively. Therefore
   $AB=S\Sigma T^*$ is SVD. Furthermore this is the top rank-$r$ SVD  because 
   $\rank(AB)=r$.   
   
   Hence
 $$\tilde C_1=T^*H=T_P^*T_B^*H.$$
   Recall that  $S_B$ and $\Sigma_B$
   are $r\times r$ matrices and 
   that
   $B\in \mathcal G^{r\times n}$ and
   deduce from SVD $B=S_B\Sigma_BT_B^*$ that $T_B^*=S_B^*\Sigma_B^{-1}B$.
   Substitute this expression and obtain that 
   $$\tilde C_1=T_P^*S_B^*\Sigma_B^{-1}BH.$$
   Notice that 
   $T_P$ and $S_B$ are $r\times r$ orthogonal  matrices, recall 
   Lemma \ref{lehg}, and deduce that 
   $$|\tilde C_1^{+}|\le |\Sigma_B|~|(BH)^+|.$$
   Substitute $|\Sigma_B|=|B|=\nu_{r,n}$
   and obtain 
   $$|\tilde C_1^{+}|\le \nu_{r,n} |(BH)^+|.$$
   
   Now let $H=S_H\Sigma_HT_H^*$ be SVD.
   Then $BS_H\in \mathcal G^{r\times l}$ by virtue of Lemma \ref{lepr3}.
  Therefore  
 $$|(BH)^+|\le \nu_{r,l}^+|\Sigma_{H_r}^{-1}|=
  \nu_{r,l}^+|H_r^+|.$$
   Substitute this inequality into the above bound on $|\tilde C_1^{+}|$ and obtain 
\begin{equation}\label{eqc1+}
|\tilde C_1^{+}|\le \nu_{r,n}\nu_{r,l}^+|H_r^+|.
\end{equation}
    Next let $M=AB+E$
 where possibly $E\neq O$, $\rank(M)>r$, and
$\tilde C_1\neq C_1$,  assume that
  $$\alpha:=||E||_F/(\sigma_{r}(M)-\sigma_{r+1}(M))\le 0.2$$
 and deduce from Theorem \ref{thsngspc}     that there exists a unitary basis 
 $B_{r,{\rm right}}$ of  the space
 $\mathcal R(T_M)$
 of the top
    $r$ right singular vectors of  $M$
    within the Frobenius norm  bound
    $4\alpha$
   from a  unitary basis for   
    the space $\mathcal R(T_{AB})$ of the top
    $r$ right singular vectors of  $M$.
   
    Recall that $||C_1^+||^{-1} =\sigma_r(C_1)$ (cf. (\ref{eqnrmcnd})
    and that the singular values  of a matrix are invariant in its unitary transformation. 
    
    Furthermore perturbation 
    of the matrix $T$ within the norm bound $4\alpha$  
    causes perturbation of $C_1$ within
     $4\alpha||H||$ and not more that that
     for $\sigma_r(C_1)$ by virtue of
     Lemma \ref{lesngr}. Hence
%     under such  a perturbation of the %matrix $AB$ by $E$, it holds
    $$||C_1^+||^{-1} =\sigma_r(C_1)\ge   
    \sigma_r(\tilde C_1)-4\alpha ||H||=
 ||(\tilde C_1)^+||^{-1}-4\alpha ||H||.$$ 
    Substitute bound (\ref{eqc1+})
    and
    arrive at claim (iii) of Lemma 
    \ref{leerrfctr}. 
    \end{proof}

%------------------------------------------------------------------------------

\section{Computation of Sampling and Re-scaling Matrices}\label{ssrcs}
 
%------------------------------------------------------------------------------

We begin with the following simple computations.
Given an $n$ vectors 
${\bf v}_1,\dots,{\bf v}_n$
of dimension $l$, write
 $V=({\bf v}_i)_{i=1}^n$
 and compute $n$ leverage scores
 
%------------------------------------------------------------------------------

\begin{equation}\label{eqsmpl}
p_i={\bf v}_i^T{\bf v}_i/||V||^2_F,
i=1,\dots,n.
\end{equation}
Notice that $p_i\ge 0$ for all $i$ and
 $\sum_{i=1}^np_i = 1$.

%------------------------------------------------------------------------------

Next assume that some leverage scores $p_1,\dots,p_n$
are given to us and recall  \cite[Algorithms 4 and 5]{DMM08}.
For a fixed positive 
integer $l$ they sample  
either exactly $l$ columns 
 of an input matrix $W$ (the $i$th column
 with probability $p_i$)
or at most $l$ its columns in expectation
(the $i$th column
 with probability $\min\{1, lp_i\}$),
respectively.

%------------------------------------------------------------------------------

\begin{algorithm}\label{algsmplex} 
{\rm [The Exactly($l$) Sampling and Re-scaling.]}

%------------------------------------------------------------------------------

\begin{description}
 
%------------------------------------------------------------------------------

\item[{\sc Input:}] 
Two integers $l$ and $n$ such that $1\le l\le n$ and $n$ nonnegative scalars $p_1,\dots,p_n$
such that $\sum_{i=1}^np_i = 1$.

%------------------------------------------------------------------------------

\item[{\sc Initialization:}]
Write $S:=O_{n,l}$ and $D:=O_{l,l}$. 

\item[{\sc Computations:}]
 
(1)  For $t = 1,\dots,l$ do

Pick $i_t\in \{1,\dots,n\}$ such that
 Probability$(i_t = i) = p_i$;

$s_{i_t,t} := 1$;

$d_{t,t} = 1/\sqrt{lp_{i_t}}$;

end

\medskip

(2) Write $s_{i,t}=0$ for all pairs of 
$i$ and  $t$
unless $i=i_t$.

%------------------------------------------------------------------------------

\item[{\sc Output:}] 
$n\times l$ sampling matrix $S=(s_i,t)_{i,t=1}^{n,l}$ and
$l\times l$ re-scaling matrix 
$D=\diag(d_{t,t})_{t=1}^l$.

\end{description}
\end{algorithm}
The algorithm performs $l$ searches in the set $\{1,\dots,n\}$, $l$ multiplications,
$l$ divisions, and the computation
of  $l$ square roots.

%------------------------------------------------------------------------------

\begin{algorithm}\label{algsmplexp} 
{\rm [The Expected($l$) Sampling  and Re-scaling.]} 

%------------------------------------------------------------------------------

\begin{description}
 
%------------------------------------------------------------------------------

\item[{\sc Input, Output and Initialization}] are as in Algorithm   \ref{algsmplex}. 
\item[{\sc Computations:}]
Write $t := 1$;

for $t = 1,\dots, l-1$ do

for $j = 1,\dots, n$ do

Pick $j$ with probability
$\min\{1, lp_j\}$;

if $j$ is picked, then

$s_{j,t}: = 1$;

$d_{t,t} := 1/\min\{1,\sqrt{lp_j}\}$;

$t := t + 1$;

end

end

%------------------------------------------------------------------------------

\end{description}

%------------------------------------------------------------------------------

\end{algorithm}

%Computations of  
Algorithm \ref{algsmplexp} involves  $nl$ memory cells.
$O((l+1)n)$ flops, and the computation
of  $l$ square roots.
 
Obtain  the following results from 
 \cite[Lemmas 3.7 and 3.8]{BW17} (cf. \cite{RV07}).
\begin{theorem}\label{thsngvsmpl} {\rm [The Impact of Sampling and Re-scaling on the Singular Values
of a Matrix.]}

 Suppose that  
$n > r$,  $V\in \mathbb C^{n\times r}$ and $V^TV = I_r$.
Let $0<\delta\le 1$ and  
$4r \ln(2r/\delta)<l$. 
Define leverage scores by equations
(\ref{eqsmpl}) and then compute
the sampling and re-scaling matrices $S$ and $D$ by 
applying Algorithm \ref{algsmplexp}.
Then bounds (\ref{eqsngvsmpl}) hold
  with a probability at least $1-\delta$.

%------------------------------------------------------------------------------

\end{theorem}

Notice that $||D^{-1}||_F\le \sqrt r$.

\begin{theorem}\label{thnrmsmpl}
{\rm [The Impact of Sampling and Re-scaling on the Frobenius Norm
of a Matrix.]}
Define the sampling and scaling matrices 
$S$ and $D$ as in Theorem \ref{thsngvsmpl}. 
Then for an $m\times n$ matrix $W$
it holds with a probability at least 0.9 that $||WSD||_F^2\le ||W||^2_F$.
\end{theorem}
 
%------------------------------------------------------------------------------
       
%\begin{theorem}\label{thcurgge}
%The randomized algorithms of \cite[Lemma %3.7 and 3.8]{BW17}  applied to an 
%$\times n$ tall-skinny matrix $M$, such %that $m\gg n$ (and extending the study in \cite{RV07}
% and \cite{DMM08})  output  a $q\times n$ 
%submatrix of the matrix $M$, 
%for $q=4ck\ln(2k/\delta)$ and a %sufficiently large constant factor $c$,
%that whp
%has all its singular values within a 
% factor $1+\phi$
 %from those of the matrix $M$ where 
 %$\phi\rightarrow 0$ as $q\rightarrow 
 %\infty$. The algorithms involve
% order of $mn$ flops up to a polylog factor. 
% \end{theorem}

%-----------------------------------------------------
                                                                                                                                                                                                                                                                                                                                                                                                                                                           
\section{Superfast Transition from an LRA to a CUR LRA}\label{slratocur}

%------------------------------------------
%------------------------------------------------------------------------------
  
\subsection{Superfast transition from an LRA to top SVD}\label{slrasvd}

%------------------------------------------
 
% We begin with an auxiliary transition of %independent interest from an LRA of a %matrix to its top SVD;
%(see again Figure \ref{fig5}); 
%We solve the problem for
%  more general LRA of  matrix given by %the product of three factors. 
    
\begin{algorithm}\label{alglratpsvd}
{\rm (Superfast Transition from an LRA to the Top SVD. Cf. \cite[Algorithms 5.1 and 5.2]{HMT11}.)}
 
%------------------------------------------------------------------------------

\begin{description}

%------------------------------------------------------------------------------

\item[{\sc Input:}]
Four matrices 
$A\in \mathbb C^{m\times l}$, 
$W\in \mathbb C^{l\times k}$, 
$B\in \mathbb C^{k\times n}$, and  
$M\in \mathbb C^{m\times n}$
such that 
$$M=AWB+E,~||E||= O(\tilde\sigma_{r+1}),~
r\le\min\{k,l\},~k\ll m,~{\rm and}~l\ll n~
{\rm for}~\tilde\sigma_{r+1}~{\rm of ~Lemma~\ref{letrnc}}.$$
%(This turns into bound  (\ref{eqlrk}) for %$k=l$ and $V=I_n$.)
%------------------------------------------------------------------------------

\item[{\sc Output:}]
Three matrices  
$S\in \mathbb C^{m\times r}$ (unitary),  
$\Sigma\in \mathbb C^{r\times r}$ 
(diagonal), and  
$T^*\in \mathbb C^{r\times n}$  (unitary)
such that $M=S\Sigma T^*+E'$ for 
$||E'||= O(\tilde\sigma_{r+1})$.                                                                                                                                                    

%------------------------------------------------------------------------------
%-------------------------------------------------------------------------------

\item[{\sc Computations:}]
\begin{enumerate}
\item
 Compute QRP rank-revealing factorization  of the matrices $A$ and $B$:
$$A=(Q~|~E_{m,l-r})RP~{\rm and}~B=P'R'
 \begin{pmatrix}Q'\\E_{k-r,n}\end{pmatrix}
 $$ where  $Q\in \mathbb C^{m\times r}$,
 $Q'\in \mathbb C^{r\times n}$ and 
 $||E_{m,l-r}||+ 
 ||E_{k-r,n}||=O(\tilde\sigma_{r+1})$. Substitute the expressions for $A$ and $B$
 into the matrix equation
$M=AWB+E$ and obtain
$M=                                                                           QUQ'+E'$ where 
$U=RPVP'R'\in \mathbb C^{r\times r}$ and
 $||E'||=O(\tilde\sigma_{r+1})$. 
 \item
Compute 
 SVD $U=\bar S\Sigma \bar                                                                                    T^*$.
Output the $r\times r$  diagonal matrix $\Sigma$.
 \item
Compute and output the unitary matrices
$S=Q\bar S$ and  $T^*=\bar T^                                                                                                                                                                                                                                      *Q'$.
\end{enumerate}
%------------------------------------------------------------------------------
\end{description}
%------------------------------------------------------------------------------
\end{algorithm}

%------------------------------------------------------------------------------

This  algorithm 
uses $ml+lk+kn$ memory cells and $O(ml^2+nk^2)$ flops. 
 
%------------------------------------------------------------------------------
  
\subsection{Superfast transition from top SVD to a CUR LRA}\label{ssvdcur}
 
%------------------------------------------

Given SVD of a rank-$r$ matrix $M$, the following superfast algorithm computes
CUR decomposition of $M$.

%------------------------------------------------------------------------------
    
\begin{algorithm}\label{algsvdtocur}
{\em [Superfast Transition from the Top SVD to a CUR LRA.]}
 
%------------------------------------------------------------------------------

\begin{description}

%------------------------------------------------------------------------------

\item[{\sc Input:}]
Five integers $k$, $l$, $m$, $n$, and $r$ such that $1\le r\le k\le m$
and  $r\le l\le n$ and matrices  
$M\in \mathbb C^{m\times n}$,
$\Sigma\in \mathbb C^{r\times r}$ 
(diagonal),
$S\in \mathbb C^{m\times r}$, and  
$T\in \mathbb C^{n\times r}$  (both unitary)
such that $1\le r\ll\min\{m,n\}$, 
$M=S\Sigma T^*E$.                                                                                                                                                    

%------------------------------------------------------------------------------

\item[{\sc Output:}]
Three matrices 
$C\in \mathbb C^{m\times l}$, 
$U\in \mathbb C^{l\times k}$, and
$R\in \mathbb C^{k\times n}$   
such that $C$ and $R$ are submatrices of $M$ and
$$M=CUR.$$ 

%------------------------------------------------------------------------------

\item[{\sc Computations:}]
\begin{enumerate}
\item
By applying the algorithms of \cite{GE96} or  \cite{P00} to the matrices $S$ and $T$
compute their  $k\times r$
and $r\times l$ submatrices
$S_{\mathcal I,:}$
 and $T^*_{:,\mathcal J}$,
respectively.
Output the CUR factors   
$C=
S\Sigma T^*_{:,\mathcal J}$
and $R=
S_{\mathcal I,:}\Sigma T^*$.
\item 
Compute and output a nucleus 
$U=M_{k,l,r}^+=M_{k,l}^+=
T_{:,\mathcal J}\Sigma^{-1} S_{\mathcal I,:}^*$
for the CUR generator  
$M_{k,l}=
S_{\mathcal I,:}\Sigma T^*_{:,\mathcal J}$.
\end{enumerate}
%------------------------------------------------------------------------------
\end{description}
%------------------------------------------------------------------------------
\end{algorithm}

%------------------------------------------------------------------------------

Correctness of the algorithm is immediately verified:
$CUR=(S\Sigma T^*_{:,\mathcal J})
(T_{:,\mathcal J}\Sigma^{-1} S_{\mathcal I,:}^*)(S_{\mathcal I,:}\Sigma T^*)=
S\Sigma T^*$.

The algorithm uses $kn+lm+kl$ memory cells and 
$O(ml^2+nk^2)$ flops, and so it is superfast for $k\ll m$ and $l\ll n$.

Let us estimate the norm $||U||$.
Clearly $||M_{k,l}||\le ||M||$.
Recall that by virtue of Theorem \ref{thcndprt}
$||S_{\mathcal I,:}^+||\le t_{m,l,h}$ and
$||T_{:,\mathcal J}^{*+}||\le t_{n,k,h}$
for $t_{q,s,h}$ of (\ref{eqtqsh})
(where we can choose, say, h=1.1)
and that 
$||\Sigma^{-1}||=||M^+||=1/\sigma_r(M)$. Combine these bounds with
  Lemma \ref{lehg} and deduce that
  \begin{equation}\label{eqsvdtocur}
||U||=||W_{k,l}^+||\le 
||S_{\mathcal I,:}^+||~||\Sigma^{-1}||~|| T_{:,\mathcal J}^{*+}||\le
t_{m,l,h}t_{n,k,h}/\sigma_r(W). 
\end{equation}

%------------------------------------------------------------------------------
   
\begin{remark}\label{retpsvdrnd}
At stage 1 we can apply the randomized algorithm of
Remark \ref{resngv},  involving only $O((m+n)r+kl\min\{k,l\})$ flops, instead of 
 the more expensive  determinitic algorithms of \cite{GE96} or  \cite{P00}. 
Moreover the following bound on the norm $||U||$ would replace  (\ref{eqsvdtocur}):
\begin{equation}\label{eqsvdtocur1}
||U||\le \Big(1+\sqrt{4r\ln (2r/\delta_1)/l}\Big)~
\Big(1+ \sqrt{4r\ln (2r/\delta_2)/k}\Big)/\sigma_r(W)
\end{equation}
 with a probability at least 
 $(1-\delta_1)(1-\delta_2)$ for any fixed pair of positive $\delta_1$ and $\delta_2$ not exceeding 1.
\end{remark}

\section{Two Small Families of Hard Inputs for Superfast LRA}\label{shrdin}

\begin{example}\label{exdlt}
Define $\delta$-{\em   matrices} of rank 1 filled with zeros except  for a  single entry filled with 1. There are exactly $mn$ such $m\times n$ matrices,
e.g.,  
 four matrices of size $2\times 2$: 
$$ \begin{pmatrix}1&0\\0&0\end{pmatrix},~
 \begin{pmatrix}0&1\\0&0\end{pmatrix},~ \begin{pmatrix}0&0\\1&0\end{pmatrix},~ \begin{pmatrix}0&0\\0&1\end{pmatrix}.$$ 
 \end{example}
The output matrix  of any  superfast  algorithm approximates
nearly $50\%$ of all these matrices   
as poorly as the matrix filled with the vales 1/2 does.
%, that is, with relative error 1. 
Indeed a superfast algorithm only depends on a small subset of all $mn$ input entries, and so its output is invariant in the input values at all the other entries.
In contrast nearly $mn$ pairs of $\delta$-matrices vary on  these entries by 1. The approximation by a single value is off by at least 1/2 for one or both of the matrices of such a pair, that is, it is off at least as much as  the approximation by the matrix filled with 
the values 1/2. Likewise if a superfast LRA algorithm is randomized and  accesses an input entry with a probability $p$, then 
with probability $1-p$ it approximates some 
$\delta$-matrix with an error at least 1/2 at that entry.
 
Furthermore 
if a superfast algorithm has been applied to a  perturbed   
$\delta$-matrix $W$ and only accesses its nearly vanishing entries, then it would optimize LRA over the class of nearly vanishing matrices
 and would  never detect its failure to approximate the only entry of the matrix $W$ close to 1. 
\begin{remark}\label{redltdns}
$\delta$-matrices are
sparse, but subtract the rank-1 matrix filled with 1/2 from every
$\delta$-matrix and obtain 
 a family of {\em dense} matrices of  rank 2 that are not close to sparse matrices, but similarly to Example \ref{exdlt}, LRA of this family computed by any superfast algorithm is at best within 1/2, that is, no better  than by the trivial matrix  filled with zeros.
\end{remark}

%-----------------------------------------------------

\section{Superfast a posteriori error  estimation for LRA in a special case}\label{spstr}

%-----------------------------------------------------
          
In a very special but important case we obtain a posteriori error  estimates 
simply by applying the customary basic rules of {\em hypothesis 
testing for the variance of a Gaussian variable.} In this case we
need no upper bound on the 
 error norm of  rank-$r$ approximation.
   
% \begin{remark}\label{reestnrk}
% The estimates enable us to
%  bound the numerical rank of an input %matrix.  
% \end{remark}

Namely we can do this in the case where
 the error matrix  $E$ of an LRA 
  has enough  entries, say,  100 or more, and where they are the 
observed i.i.d. values of a single random variable. This is realistic, for example, where the deviation 
of the matrix $W$ from  its rank-$r$ approximation is due to the errors of
 measurement or rounding.

 In this case the Central Limit Theorem implies that the distribution of the variable is close to Gaussian
(see \cite{EW07}). 
Fix a pair of integers $q$ and $s$
such that $qs$ is large enough (say, exceeds 100), but $qs=O((m+n)kl)$ 
and hence $qs\ll mn$; then  
 apply our tests just to
a random $q\times s$ submatrix
of the $m\times n$ error matrix.

  Under this policy we compute the error matrix 
 at a dominated arithmetic cost in $O((m+n)kl)$ 
but still verify correctness  
 with high confidence, 
by applying the rules of hypothesis 
testing for the variance of a Gaussian variable.

%------------------------------------------------------------------------------
Let us specify this basic process for the sake of completeness.
Suppose that we have observed the values  
$g_1,\dots,g_K$
of a Gaussian random variable $g$ with a mean value $\mu$
and a variance $\sigma^2$ and that we have computed 
 the observed average value and variance 
 $$\mu_K=\frac{1}{K}\sum_{i=1}^K g_i~
{\rm and }~\sigma_K^2=\frac{1}{K}\sum_{i=1}^K |g_i-\mu_K|^2,$$
respectively.
Then, for a fixed reasonably large $K$,
both 
$${\rm Probability}~\{|\mu_K-\mu|\ge t|\mu|\}~{\rm and~Probability}\{|\sigma_K^2-\sigma^2|\ge t\sigma^2\}$$ 
converge to 0 exponentially fast as $t$ grows to the infinity
(see \cite{C46}).

%------------------------------------------------------------------------------

\noindent {\bf Acknowledgements:}
Our work was supported by NSF Grants
 CCF--1563942 and CCF--133834
and PSC CUNY Award 69813 00 48.
We are also very grateful to
E. E. Tyrtyshnikov for the challenge
of formally supporting empirical power of C--A iterations,
to N. L. Zamarashkin for his ample comments on his work with A. Osinsky  on LRA via volume maximization and on the first draft of  \cite{PLSZ17}, and to
  S. A. Goreinov, 
 I. V. Oseledets, A. Osinsky,  E. E. Tyrtyshnikov, and N. L. Zamarashkin for valuable reprints and pointers to relevant bibliography.

%We are also grateful 
% to Vsevolod Oparin
%for a pointer to some  bibliography.
%and to the reviewers for valuable comments.

%------------------------------------------------------------------------------

%\begin{remark}\label{repr3}
% {\em [Extension to the matrices filled %with iid entries from a finite set.]}
%Suppose that  the matrices  $F$ and $H$  %of Theorem \ref{thrnd} and  the matrix 
%of claim (ii) of Theorem \ref{thlrnk} are
%not Gaussian but are filled with iid %random variables sampled under the
%uniform probability distribution from 
%a finite set $\mathcal S$ of cardinality 
%$|\mathcal S|$. Then each of the matrices %$F$, $H$, $FA$,  and $AH$  of that %theorem has full rank $r$ 
%with a probability at least $1-r/|
%\mathcal S|$, and likewise
%the equation 
%$\mathcal R(W)=\mathcal R(M)$ of claim %(ii) of Theorem \ref{thlrnk}  
%  holds with a probability at least $1-%r/|S|$ (see  \cite{DL78}, \cite{Z79},  
%\cite{S80},  
%and \cite{PW08} for the proof %techniques).
%\end{remark}

%------------------------------------------------------------------------------

%------------------------------------------------------------------------------

\end{document}